\documentclass[11pt]{article}
\usepackage{amsfonts}
\usepackage{amssymb}
\usepackage{latexsym}
\usepackage{multicol}
\usepackage{amsmath}
\usepackage{amsthm}
\usepackage[latin1]{inputenc}
\usepackage[english]{babel}
\usepackage{psfrag}
\usepackage{epsfig}
\usepackage{subfigure}
\usepackage{color}
\usepackage{hyperref}
\usepackage{float}
{\theoremstyle{remark}
  \newtheorem{remark}{ Remark }[section]
   } 
\let\oldremark\remark
\let\oldendremark\endremark
\renewenvironment{remark}
  {\vspace{0.05cm} \small \oldremark \vspace{0.05cm}}
  {\oldendremark}

\newtheorem{theorem}{Theorem}[section]
\newtheorem{conjecture}{Conjecture}[section]
\newtheorem{proposicio}{Proposition}[section]
\newtheorem{lema}{Lemma}[section]

\DeclareRobustCommand{\fdd}{
\ifmmode
\else \leavevmode \unskip \penalty999 \hbox{} \nobreak \hfill
\fi
\quad\hbox{\qedsymbol}}

                {\qed \end{trivlist}}

\newcommand{\mnu}{{\hat{\nu}}}
\newcommand{\nudelta}{{\nu}}

\newcommand{\cO}{\mathcal{O}}

\newcommand{\cP}{\mathcal{P}}
\newcommand{\nR}{\mathbb{R}}

\newcommand{\ii}{{\mbox{\rm i}\,}}

\newcommand{\DF}[1]{\Delta F_{#1}}       
\newcommand{\DFt}[2][1]{\Delta F_{#2}^{\{#1\}}} 
\newcommand{\DFn}[1]{\widetilde{\Delta F_{#1}}} 
\newcommand{\Fn}[2]{\widetilde{F}_{#1}^{#2}} 

\newcommand{\bx}{{\bf x}}
\newcommand{\by}{{\bf y}}

\DeclareMathOperator{\sech}{sech}

\begin{document}

\title{Splitting of the separatrices after a Hamiltonian-Hopf bifurcation under periodic forcing}

\author{E. Fontich, C. Sim\'o, A. Vieiro\\ Departament de Matem\`atiques i Inform\`atica
\\ Universitat de Barcelona \\ BGSMath \\ Gran Via, 585, 08007, Barcelona, Spain}

\maketitle


\begin{abstract}
\noindent We consider the effect of a non-autonomous periodic perturbation on a
2-dof autonomous system obtained as a truncation of the Hamiltonian-Hopf normal
form. Our analysis focuses on the behaviour of the splitting of the invariant
2-dimensional stable/unstable manifolds. We analyse the different changes of
dominant harmonic in the splitting functions. We describe how the dominant
harmonics depend on the quotients of the continuous fraction expansion of the
periodic forcing frequency. We have considered different frequencies
including quadratic irrationals, frequencies having continuous
fraction expansion with bounded quotients and frequencies with unbounded quotients. The
methodology used is general enough to systematically deal with all these
frequency types. All together allow us to get a detailed description of the
asymptotic splitting behaviour for the concrete perturbation considered.
\end{abstract}

\section{Introduction}

In this work we consider the $(2+\frac{1}{2})$-dof Hamiltonian system
\begin{equation} \label{system_intro}
H(x_1,x_2,y_1,y_2,t) = H_0(x_1,x_2,y_1,y_2) + \epsilon H_1(x_1,x_2,y_1,y_2,t),
\end{equation}
where 
$$
H_0(x_1,x_2,y_1,y_2)= x_1 y_2 - x_2 y_1 + \nu \left(\frac{x_1^2+x_2^2}{2} + \frac{y_1^2+y_2^2}{2} \left( -1 + \frac{y_1^2+y_2^2}{2} \right) \right),
$$
and
$$
H_1(x_1,x_2,y_1,y_2,t)=\frac{y_1^5}{(d-y_1)(c-\cos(\theta))}, \quad \theta=\gamma t + \theta_0.
$$
We shall fix concrete values of $c$, $d$, $\gamma$ and $\epsilon$, and consider
$\nu>0$ as a perturbative parameter. The parameter $\theta_0 \in [0,2\pi)$ is an
initial time phase.

Our motivation to consider that concrete system is to study some dynamical
properties related to the Hamiltonian-Hopf bifurcation under a periodic
forcing. Then, to start with, in Section~\ref{Sec:SokolskiiNF} we briefly
review the reduction to Sokolskii normal form (NF) for a 2-dof Hamiltonian
system that undergoes a Hamiltonian-Hopf bifurcation.  

The truncation of the Sokolskii NF provides an integrable approximation of the
dynamics. The above unperturbed system $H_0$ is simply the lowest order
truncation that captures the main dynamical features of the
Hamiltonian-Hopf bifurcation. Some basic facts concerning the dynamics of $H_0$
are summarized in Section~\ref{Sec:ConcreteSystem}. 

For $\nu>0$ the origin becomes of complex-saddle type. For $\epsilon=0$ the 2D
stable/unstable invariant manifolds coincide. But for small and fixed
$\epsilon>0$ the perturbation $H_1$ creates a splitting of these invariant
manifolds. In Section~\ref{Sec:ConcreteSystem} we also discuss some nice properties
of the chosen perturbation. Such splitting of the invariant manifolds becomes
exponentially small in $\nu$ as $\nu \rightarrow 0$. In Section~\ref{sect_num}
we perform a numerical computation of the splitting functions.  Quadruple
precision arithmetics is used to integrate (\ref{system_intro}) in order to get
a sample of points on $W^u(0)$ and $W^s(0)$ that allows us to compute the
splitting function in a fundamental domain. Different bifurcations are detected
examining the nodal lines of the splitting functions as $\nu$ varies.

The corresponding Poincar\'e-Melnikov function is analytically investigated in
Section~\ref{Sec:Splitting} by means of a combination of numerical, symbolical
and theoretical tools. The splitting problem considered is non-perturbative
($\epsilon$ is fixed) and singular (when $\nu=0$ the system is not hyperbolic)
and the  use of the Melnikov approximation to study the splitting is not
theoretically justified. In Section~\ref{Sec:ValMel} we compare the results of the splitting
obtained in Section~\ref{sect_num} with those from (a suitable truncation, adding only the
relevant terms) the Melnikov approximation derived in Section~\ref{Sec:Splitting},
getting a remarkable agreement.

In Section~\ref{sect_PoincMel} we further analyse the Poincar\'e-Melnikov
function by taking advantage of the concrete properties of the
system and of the perturbation to give explicit details of the asymptotic
behaviour of the splitting. In particular, we look for the concrete values of
the parameter $\nu$ for which a change in the dominant harmonic is detected and
we study how these values asymptotically behave.

As expected, the Diophantine properties of the frequency $\gamma$ of the perturbation $H_1$
play a key role in the analysis performed in Section~\ref{sect_PoincMel}. Note
that we do not assume that we have a concrete frequency, instead our hypothesis
are on the properties of the continuous fraction expansion (CFE) of $\gamma$.
Section~\ref{Sec:Otherfreq} is devoted to illustrate the behaviour of the splitting for
different frequencies $\gamma$. In particular, we show examples where
some of the best approximants of $\gamma$ never become a dominant harmonic in
the splitting functions.

Finally, Section~\ref{Sec:Conclusion} summarizes the results and describes
related future work problems.  

Five appendices complement the discussions through the text. In
Appendix~\ref{autonomous} we study the splitting under an autonomous
perturbation of the unperturbed system.  The simple asymptotic behaviour of the
splitting is well-understood in this situation in contrast with the
non-autonomous perturbation case studied in this work.
Appendix~\ref{exempleduffing} however illustrates that in the autonomous case,
taking a non-entire perturbation, the analysis of the splitting by considering
individual terms of the series expansion of the perturbation can lead to a
larger dominant exponent of the Melnikov function.  This is not expected in the
non-autonomous case since the dominant term comes from the quasi-periodic
properties of the splitting asymptotic behaviour. Appendix~\ref{Sec:Regularity}
discusses about the role of the regularity of the non-autonomous perturbation
in $t$ in the asymptotic behaviour of the splitting.

When describing the asymptotic behaviour of the splitting of the invariant
manifolds for system (\ref{system_intro}) we will see that for large
intervals of $\nu$ the dominant harmonic coincides for both splitting
functions. However, there are small intervals of $\nu$ where the dominant
harmonics differ. In Appendix~\ref{split-difusion} we comment on the expected
consequences that this fact has in what concerns the (local) diffusive
properties of the system for very small values of $\nu$.

In the last appendix we focus on the presence of hidden harmonics, that is, harmonics
associated to best approximants of $\gamma$ that never become a dominant harmonic
of the splitting function. As said, hidden harmonics
are shown for some frequencies in Section~\ref{Sec:Otherfreq}. 
We prove in Appendix~\ref{conseq_best} that, under generic conditions, it is
not possible to have two consecutive best approximants of $\gamma$ which are not related to
a dominant harmonic of the splitting function when some nearby quotients of the
CFE of $\gamma$ are large enough. A more general situation can be found in \cite{FonSimVie-2}.

The theoretical derivations presented in this work provide a satisfactory and
complete description of the asymptotic behaviour of the splitting of
separatrices of the system (\ref{system_intro}). On the other hand, a complete
rigorous proof of the results included here will require 
\begin{enumerate}
\vspace{-0.2cm}
\item to bound the effect of higher order terms of the expansion of the
splitting function in powers of $\epsilon$ to guarantee that the first order Poincar\'e-Melnikov
function provides the dominant term of the splitting behaviour, and 
\vspace{-0.2cm}
\item to check that the contribution of the non-dominant harmonics of the
Poincar\'e-Melnikov approximation does not change the dominant term of the
asymptotic expansion of the splitting behaviour.
\end{enumerate}
Even if we do not address formally any of the previous items, the numerical
results that we present provide a strong numerical evidence supporting them.

\section{The theoretical framework: the Hamiltonian-Hopf bifurcation} \label{Sec:SokolskiiNF}

For the reader's convenience, in this section we briefly summarize some details of
the analysis of the Hamiltonian-Hopf bifurcation.

Consider a one-parameter family of Hamiltonian systems
$H_\mnu(x_1,x_2,y_1,y_2)$ which undergo a Hamiltonian-Hopf bifurcation.
Assume that for $\mnu>0$ the origin is elliptic and becomes
complex unstable for $\mnu<0$. This implies that the eigenvalues of the
linearised Hamiltonian system suffer a Krein collision: for $\mnu>0$ the linear
system has two pairs of purely imaginary eigenvalues $\pm \ii \omega_1$ and $\pm
\ii \omega_2$. These pairs meet in a double pair $\pm \ii \omega$, $\omega >0$, on
the imaginary axis for $\mnu=0$ (Krein collision) and they become a hyperbolic
quartet  $\pm \alpha \pm \ii \omega$, $\alpha,\omega>0$ for $\mnu<0$.

Let $\mathbb{P}_k$ be the set of homogeneous polynomials of degree $k \in
\mathbb{N}$.  Consider the Taylor expansion at $0$ of $H_{\mnu}$ expressed as
$$
H_{\mnu}=\sum_{k \geq 2} \sum_{j\geq 0} \mnu^j H_{k,j}  , \qquad \text{where } H_{k,j}
\in \mathbb{P}_k \text{ for all } k\geq 2, \ j\geq0.
$$
The first step is to reduce the quadratic part $H_{2,0}$ to a canonical NF
(i.e. a NF obtained via a symplectic change of coordinates). After doing this
reduction the strategy will be to use a Lie series methodology to successively
(order by order) simplify (as much as possible) the terms $H_{2,j}, \, j\geq
1$, and $H_{k,j}, \, k\geq3, j \geq0$.

The possible canonical forms for quadratic Hamiltonians were obtained in
\cite{Will37}. In the case of two pairs of (double) purely
imaginary eigenvalues $H_{2,0}$ can be reduced to the so-called Williamson NF
\begin{equation} \label{WNF}
H_{2,0}= -\omega (x_2 y_1 - x_1 y_2) + \frac{1}{2} (x_1^2 + x_2^2).
\end{equation}

The next step involves normalising higher order terms of $H_{\mnu}$.  The fact
that the linearization at the Hamiltonian-Hopf bifurcation point is non-semisimple
makes the NF reduction a little bit more involved, see
\cite{MeySch71,VdM82,MeyerHall,Han07,PalYan00}. A standard procedure to deal
with the terms of order $(k,j)$ is to look for a change of variables given by
the time-$1$ map of a Hamiltonian $G \in \mathbb{P}_k$.  In such a case the
corresponding change transforms $H_{\mnu}$ into
\begin{equation} \label{2bis}
\tilde{H}_{\mnu} = \sum_{i \geq 0} \frac{1}{i!} \ \text{ad}^i_{H_{\mnu}}(G),
\end{equation}
where $\text{ad}_{F}(G)=\{F,G\}$ denotes the usual adjoint operator defined in
terms of the Poisson bracket
$$
\{F,G\} = \left( \frac{\partial{F}}{\partial x_1} \frac{\partial{G}}{\partial y_1} - \frac{\partial{F}}{\partial y_1} \frac{\partial{G}}{\partial x_1} \right) + \left( \frac{\partial{F}}{\partial x_2} \frac{\partial{G}}{\partial y_2} -  \frac{\partial{F}}{\partial y_2} \frac{\partial{G}}{\partial x_2} \right). 
$$
Collecting the terms of $\tilde{H}_{\mnu}$ of order $(k,j)$ in (\ref{2bis}) we get $H_{k,j} +
\text{ad}_{H_2}(G)$, meaning that the change of coordinates allows us to remove
the terms $H_{k,j}$ of $H_{\mnu}$ that belong to $\text{Im}
\,\text{ad}_{H_2}(G)$.  The Fredholm alternative implies that $\mathbb{P}_k=
\text{Im}\, \text{ad}_{H_2} \oplus \text{Ker} \, \text{ad}^\top_{H_2}$, where
$\text{ad}^\top_{H_2}$ denotes the transpose operator.  Then, as indicated in
\cite{Elpetal87,MeyerHall}, a systematic way to proceed is to look, at each
order $(k,j)$ of the normalisation procedure, for $G \in \mathbb{P}_k$ such that
\begin{equation} \label{homological}
H_{k,j} + \text{ad}_{H_2}(G) \in \text{Ker} \, \text{ad}^\top_{H_2}.
\end{equation}

Moreover, in the (symplectic, $\Omega=dx_1 \wedge dy_1 + dx_2 \wedge
dy_2 = dR \wedge dr + d \Theta \wedge d \theta$) new coordinates
\begin{equation}\label{Sokcoord}
y_1=r \cos(\theta), \qquad y_2=r \sin(\theta), \qquad R= (x_1 y_1 + x_2 y_2)/r, \qquad \Theta =x_2 y_1 - x_1 y_2,
\end{equation}
the transpose linear system (i.e. the system with equations defined by the
matrix $J (D^2 H_2)^\top$) reduces to $H_2^\top = -\omega \Theta + \frac{1}{2}
r^2$, see \cite{MeyerHall}. Then (\ref{homological}) implies that the
normalised (formal) Hamiltonian is given by
\begin{equation} \label{SNF}
\text{NF}(H_{\mnu})=\omega \Gamma_1 + \Gamma_2 + \sum_{\substack{k,l,j \geq 0 \\ k + l +j \geq 2}} a_{k,l,j} \, \Gamma_1^{k} \, \Gamma_3^{l} \, \mnu^j,
\end{equation}
where
$$
\Gamma_1 = x_1 y_2 - x_2 y_1, \qquad \Gamma_2 = (x_1^2+ x_2^2)/2 \qquad \text{and} \qquad \Gamma_3= (y_1^2+ y_2^2)/2.
$$
This is the so-called Sokolskii NF \cite{Sok74}, see \cite{MeyerHall,GaiGel11}
for further details on its derivation. Note that:

\begin{itemize}
\item  We have seen that there exists a formal change of variables $C$ (not
convergent in general) such that reduces the given system to the NF
(\ref{SNF}). Moreover, $C$ is symplectic, see for example \cite{SimVall01}. If
the quadratic part of the original system is already in Williamson NF, then the
change is near-the-identity.

\item The reduced Hamiltonian (\ref{SNF}) is formally integrable and possesses
$\Gamma_1$ as an extra (formal) integral of motion. The original Hamiltonian is only
formally integrable (that is, the truncation at any order is integrable) and
the difference between the Hamiltonian $H_{\mnu}$ and the
formal series $\text{NF}(H_{\mnu}) \circ C^{-1}$ is beyond all orders.

\item The reduction to a NF is achieved by means of successive changes of
coordinates to normalize order by order the full Hamiltonian. Each of the
changes of the normalization procedure reduces the domain where the
truncated NF gives a good approximation. For a fixed perturbation parameter $\mnu$, there is
an optimal truncation order of the NF that minimizes the
bound of the error between the Hamiltonian and the NF in a suitable domain
around the fixed point. Note that the optimal order depends discontinuously on $\mnu$
because it jumps on the integers. See, e.g., \cite{Nei84,Sim94}
\end{itemize}

Next we discuss some features of the invariant manifolds of the origin for
$\text{NF}(H_{\mnu})$.  In particular, $\{\Gamma_1,\Gamma_2\}=0$ and
$\{\Gamma_1,\Gamma_3\}=0$ and hence, as we have said, $\Gamma_1$ is a first
integral of $\text{NF}(H_{\mnu})$. Therefore $\Gamma_1=0$ on the invariant
manifolds of the origin. On the other hand, these manifolds lie on
$\text{NF}(H_{\mnu})=0$. From (\ref{SNF}), making explicit the lowest order
terms of $\text{NF}(H_{\mnu})$, we have
\begin{eqnarray} \label{NF4}
\text{NF}(H_{\mnu}) &=&
          \omega  \Gamma_1 +  \Gamma_2 + \mnu ( a_{1,0,1} \Gamma_1 + a_{0,1,1} \Gamma_3 )+
         a_{2,0,0} \Gamma_1^2 + a_{1,1,0} \Gamma_1 \Gamma_3 + a_{0,2,0} \Gamma_3^2  \\
\nonumber     & & + \mathcal{O}(\mnu^2 (\Gamma_1 + \Gamma_3), \mnu (\Gamma_1+\Gamma_3)^2,(\Gamma_1+\Gamma_3 )^3).
\end{eqnarray}
Then, the 2D stable and unstable invariant manifolds $W^{s/u}({\bf 0})$ are given
by the relation 
\begin{equation} \label{WUS}
\Gamma_2 + \mnu \, a_{0,1,1} \Gamma_3 + a_{0,2,0} \Gamma_3^2 +
\mathcal{O}(\mnu^2 \Gamma_3, \mnu \Gamma_3^2, \Gamma_3^3)=0. 
\end{equation}
We want to have real invariant manifolds $W^{s/u}({\bf 0})$, which requires
$\Gamma_2, \Gamma_3 >0$ (otherwise they lie in the complex domain). This means that $\mnu
a_{0,1,1}<0$ and, since we have assumed that for $\mnu<0$ the origin is a
complex-unstable fixed point, we must have $a_{0,1,1}>0$. 
Moreover, in such a case, for $a_{0,2,0}>0$ the  invariant
manifolds $W^{u/s}(0)$ live in a finite domain which,
requiring the same order for the three dominant terms in (\ref{WUS}), has
size $\Gamma_2 =\mathcal{O}(\mnu^2)$ and $\Gamma_3= \mathcal{O}(\mnu)$. However for
$a_{0,2,0}<0$ the invariant manifolds may be unbounded. For the first case we
introduce the new parameter $\nudelta$ by $\mnu=-\nudelta^2$, and the rescaling
$x_i = \nudelta^2 \tilde{x}_i$, $\omega y_i= \nudelta \, \tilde{y}_i$, $i=1,2$,
$\omega t =\tilde{t}$, see \cite{MeySch71}. For concreteness, we shall consider
$\nu>0$.  After this non-canonical change of variables the system is again
Hamiltonian and the corresponding Hamiltonian is
\begin{equation} \label{scaledNF}
\text{NF}(\tilde{H}_{\nudelta}) = \tilde{\Gamma}_1 + \nudelta \left( \tilde{\Gamma}_2 + 
                  a \tilde{\Gamma}_3 + \eta \tilde{\Gamma}_3^2 \right) + 
                  \mathcal{O}(\nudelta^2),
\end{equation}
where
\begin{equation} \label{aeta}
a=-a_{0,1,1}/\omega^2 \qquad \text{ and } \qquad \eta=a_{0,2,0}/\omega^4.
\end{equation}
Hence, as it was pointed out in \cite{McSMey03}, for $\eta>0$ the invariant
manifolds $W^{u/s}(0)$ are bounded while for $\eta <0$ they may be unbounded.
Henceforth, we assume $a<0$ and $\eta>0$.

\vspace{0.2cm}
\begin{remark} \label{remark_eigen}
\small
From (\ref{NF4}) one checks that the eigenvalues of the linearisation at the
origin of the original system $H_{\mnu}(x_1,x_2,y_1,y_2)$  are given by
$\lambda=\pm \ii \omega \pm \sqrt{a_{0,1,1}} (-\mnu)^{1/2} + \mathcal{O}(\mnu)$.
Then, for $\mnu<0$, one has  $\text{Re} (\lambda) = \pm \omega \sqrt{-a}
(-\mnu)^{1/2} + \mathcal{O}(\mnu)$ and $\text{Im} (\lambda) = \pm \omega +
\mathcal{O}(\mnu)$.
\end{remark}

\section{The system: a periodic perturbation of the truncated NF} \label{Sec:ConcreteSystem}

In this section we provide some details on the concrete system (\ref{system_intro}) studied in this paper.

\subsection{The unperturbed system} \label{sec3p1} 

Our starting point is the truncated (ignoring
$\mathcal{O}(\nudelta^2)$ terms) Sokolskii NF Hamiltonian (\ref{scaledNF}).
According to Section~\ref{Sec:SokolskiiNF}, for $a<0$
and $\eta>0$ the invariant manifolds of the origin are bounded.
The rescaling $\tilde{x}_i \to (-\sqrt{\eta}/a) \tilde{x}_i$, $\tilde{y}_i \to
(\sqrt{-\eta/a}) \tilde{y}_i$, $i=1,2$, $\nu \to \sqrt{-a} \nu$, reduces the
truncated Hamiltonian (\ref{scaledNF}) to the case $a=-1$ and $\eta=1$.  To
simplify notation, we denote the rescaled variables and parameter simply by
$(x_1,x_2,y_1,y_2)$ and $\nu$, respectively. We also introduce $\bx=(x_1,x_2)$,
$\by=(y_1,y_2)$, and we denote by $H_0$ the corresponding truncated
Hamiltonian. Hence, $H_0$ is just given by 
\begin{equation} \label{H0} H_0(\bx,\by)=
\Gamma_1 + \nu (\Gamma_2 - \Gamma_3 + \Gamma_3^2), 
\end{equation} 
where $\Gamma_1=x_1 y_2 -x_2 y_1$, $\Gamma_2 = (x_1^2+x_2^2)/2$ and $\Gamma_3 =
(y_1^2+y_2^2)/2$. The system $H_0$ is defined on the symplectic manifold
$(\mathbb{R}^4,\Omega)$ with $\Omega=dx_1 \wedge dy_1 + dx_2 \wedge dy_2$, and
the equations of motion are
$$
\begin{array}{rcr} 
\dot{x}_1 & \! = \! & -x_2 + \nu y_1 (y_1^2 + y_2^2 -1), \\
\dot{x}_2 & \! = \! &  x_1 + \nu y_2 (y_1^2 + y_2^2 -1), 
\end{array}
\qquad
\begin{array}{rcr} 
\dot{y}_1 & \! = \! & -y_2 - \nu x_1, \\
\dot{y}_2 & \! = \! &  y_1 - \nu x_2. 
\end{array}
$$

As follows from Section~\ref{Sec:SokolskiiNF}, $H_0$ is integrable and $\Gamma_1$ is an
independent first integral of the system. The origin is a fixed point of $(\ref{H0})$
with eigenvalues $\pm \nu \pm \ii$. For $\nu>0$, the
origin is of complex-saddle type and the invariant manifolds $W^{u/s}(\bf 0)$
are given by $\{H_0=0\}\cap \{\Gamma_1=0\}$. To elucidate the dynamics of
$H_0$ it is convenient to introduce (non-symplectic) polar coordinates
\begin{equation} \label{polar}
x_1=R_1 \cos(\psi_1), \quad x_2=R_1 \sin(\psi_1), \quad y_1=R_2 \cos(\psi_2), \quad y_2=R_2 \sin(\psi_2),
\end{equation}
where $R_1,\ R_2 >0$ and $\psi_1, \psi_2 \in [0,2 \pi)$. The equations of motion become
\begin{equation} \label{eqmotR1R2}
\begin{array}{rclrcl}
\dot{R}_1 &= & \nu R_2 (R_2^2 -1) \cos(\psi_2-\psi_1),   & \dot{\psi}_1  &=& 1 + \nu (R_2^2 -1) R_2 \sin(\psi_2-\psi_1) / R_1, \\
\dot{R}_2 &= &-\nu R_1 \cos(\psi_2-\psi_1),             & \dot{\psi}_2  &=& 1 + \nu  R_1 \sin(\psi_2-\psi_1) / R_2. \\
\end{array}
\end{equation}
One has  $\Gamma_1= R_1 R_2 \sin(\psi_2-\psi_1)$, and hence $\Gamma_1=0$ implies $\sin(\psi_2-\psi_1)=0$.
One can distinguish two cases: either $\psi_1-\psi_2 =0 \ (\text{mod} \, 2 \pi)$ or $\psi_1-\psi_2 = \pi \ (\text{mod} \, 2 \pi)$.
Each of these cases defines a system for $R_1,R_2>0$. 
But, since the changes $R_1 \to -R_1$ and $R_2 \rightarrow -R_2$ reduce the
system of one of the cases to the other one, it is enough to consider
one of the cases if one allows $R_1,R_2$ to be negative.\footnote{The virtual
singularities at $R_1=0$ and $R_2=0$ play no role. They are due to the use of
polar coordinates. Indeed, using the so-called Sokolskii coordinates
\cite{Sok74} one can remove one of them. We note, however, that there are not
globally defined polar coordinates around a 2-dof complex-saddle singularity, see \cite{LerUma92}.}
To fix ideas, we consider $\psi_1-\psi_2 = \pi \ (\text{mod} 2 \pi)$. Then, the
restriction of the dynamics on $\{\Gamma_1=0\}$ for the $(R_1,R_2)$-components is just given by the equations
related to the Duffing Hamiltonian $K= \nu (R_1^2 -R_2^2 + R_2^4 /2)/2$ (with
the symplectic 2-form $\Omega_K=dR_2 \wedge dR_1$). The local positive 
branch of the homoclinic orbit $\gamma(t)$ of $K$, with $R_1,R_2 >0$, corresponds
to the unstable manifold of the origin.  It follows from (\ref{eqmotR1R2})
that along the invariant manifolds $\psi_1= \psi_2 - \pi$ and $\psi_2=t + \psi_0$,
where $\psi_0 \in  [0,2\pi)$ is an arbitrary phase. Moreover, the invariant
manifolds of the unperturbed system (\ref{H0}) are foliated by homoclinic
orbits $\gamma_{\psi_0} (t)=(x_1(t),x_2(t),y_1(t),y_2(t))$ given by
\begin{equation} \label{homoH0}
x_1(t) \!=\! -R_1(t) \cos(\psi), \ \ x_2(t) \!=\! -R_1(t) \sin(\psi), \ \ y_1(t) \!=\! R_2(t) \cos(\psi), \ \ y_2(t) \!=\! R_2(t) \sin(\psi),
\end{equation}
being $\psi=t+\psi_0$, 
$R_1(t)= \sqrt{2} \sech(\nu t) \tanh(\nu t)$, and $R_2(t)=\sqrt{2}\sech(\nu t)$.
In particular, $\gamma_{\psi_0}(t)$ has singularities at $t= (2n+1) \ii \pi / 2 \nu$, $n \in \mathbb{Z}$.  

\subsection{The perturbation} 
We proceed by adding a periodic perturbation to
(\ref{H0}). Concretely, as stated in the Introduction, we consider 
\begin{equation} \label{perturbedsystem}
H(\bx,\by,t) = H_0({\bx},{\by}) + \epsilon H_1({\bx},{\by},t),
\end{equation}
where 
$H_0({\bx},{\by})$ is the unperturbed Hamiltonian (\ref{H0}), which depends on $\nu$, and 
$$
H_1({\bx},{\by},t)=g(y_1)f(\theta)=\frac{y_1^5}{d-y_1} \frac{1}{c-\cos(\theta)},
$$
where $\theta=\gamma t + \theta_0$, with $\gamma\in \mathbb{R} \setminus
\mathbb{Q}$ and $\theta_0 \in [0,2\pi)$ is an initial phase, $d>\sqrt{2}$ and $c>1$.
The parameter $\epsilon$ is considered to be small and fixed. We choose $\gamma
=(\sqrt{5}-1)/2$, $d=7$, $c=5$ and $\epsilon=10^{-3}$ for the majority of
computations through the paper, but we do not restrict to these values in the
theoretical considerations. In particular, we will give details on how to deal with
other irrational frequencies $\gamma$ and the role of their arithmetic properties 
in the asymptotic splitting behaviour as $\nu \to 0$.

\vspace{0.2cm}

The following comments motivate and somehow justify the perturbation (\ref{perturbedsystem}) considered.
\vspace{-0.2cm}
\begin{enumerate}
\item A generic autonomous perturbation would create a splitting of
separatrices. This case resembles the splitting of a $(1+1/2)$-dof Hamiltonian
system (by considering the reduction to the energy level where the separatrices lie). A
direct analysis of the Poincar\'e-Melnikov function in this case reveals an exponentially
small behaviour in the parameter $\nu$ of the splitting measured as a variation of $\Gamma_1$. We
summarize in Appendix~\ref{autonomous} the theoretical results and some concrete
numerical simulations of the behaviour of the splitting for an autonomous
perturbation. 

\item The phenomena becomes much richer under a non-autonomous perturbation since
different frequencies interact. Consider the particular case of the
perturbation (\ref{perturbedsystem}).  Around the invariant manifolds the
unperturbed system possesses the internal frequency $1$ in $t$, see
(\ref{homoH0}).  Then, we choose $\gamma \in \mathbb{R}\setminus \mathbb{Q}$ in
$H_1$ so that the effect of the perturbation resembles that of a quasi-periodic
forcing. Concretely, when one restricts the perturbation $\epsilon
H_1(\bx,\by,t)=\epsilon g(y_1)f(\theta)$ to the unperturbed invariant manifolds
$W^{u/s}({\bf 0})$ of $H_0(\bx,\by)$, since $y_1$ has a factor periodic in $t$ as
(\ref{homoH0}) shows, one gets a quasi-periodic function in $t$ with basic
frequencies $(1,\gamma)$. As will be shown, some of the linear combinations of
the basic frequencies are slower (hence they average in a worst way) and
describe the behaviour of the dominant terms of the splitting of the invariant
manifolds.

\item $H_0$ is an entire function of $\bx,\by$. The perturbation $H_1$ in a
neighbourhood of the unperturbed invariant manifolds is real analytic with
respect to $\bx,\by$ (because, in particular, $y_1 \lesssim \sqrt{2}$ and we
choose $d=7$) and it is analytic in $t$. This implies that the amplitude of
the dominant term of the Poincar\'e-Melnikov approximation of the splitting
function decreases exponentially in the parameter $\nu$, see details in Appendix~\ref{Sec:Regularity}. 

\item The Fourier coefficients of the even function $f(\theta)=
(c-\cos(\theta))^{-1} = \sum_{j\geq 0} c_j \cos(j \theta)$ are given by 
\begin{equation} \label{expcs}
c_0=1/\sqrt{c^2 -1}, \qquad c_j=2 c_{0}/(c+\sqrt{c^2-1})^j \quad \text{ for }  j\geq 1.
\end{equation}
In particular, the Fourier coefficients decay as $1/(c+\sqrt{c^2-1})^j$ (this
is related to the fact that $f(\theta)$ has poles at $\pm \ii \log(c+\sqrt{c^2-1})$).

On the other hand, the Taylor series of $g(y_1)$ is given by
\begin{equation} \label{tayg}
g(y_1)=\frac{y_1^5}{d-y_1}= \sum_{k \geq 0} d^{-k-1} y_1^{5+k}.
\end{equation}
It contains all powers $y_1^k$ for $k\geq 5$, but does not contain terms in the
other three variables.

The choice of $c=5$ and $d=7$ guarantees a fast enough decay but still allows
us to differentiate the role of the different harmonics in the
Poincar\'e-Melnikov function. See also Remark~\ref{remark_teogen} below.

\item Finally, there is also a practical reason: $H_1$ is simple enough so that
quadruple precision numerical integration of the full system can be carried out
in a reasonable CPU time.   
\end{enumerate}

\begin{remark} 
Note that the perturbation $H_1(\bx,\by,t)$ preserves the fixed point at the
origin. Instead one could consider perturbations such that the origin becomes a
periodic orbit. We do not deal with this situation in this paper, but note that
the description given here also applies to this case.  
\end{remark}

\subsection{The splitting function} \label{Sec:splfun}

In the following sections we study the invariant manifolds $W^{u/s}(\bf 0)$ of
the system (\ref{system_intro}) and the asymptotic behaviour of their
splitting as $\nu \rightarrow 0$. Here we introduce the notation we shall use
to refer to the splitting function and its different approximations. 

We write $H_0=G_1+ \nu G_2$, where $G_1=\Gamma_1$ and
$G_2=\Gamma_2 - \Gamma_3 + \Gamma_3^2$. $G_1$ and $G_2$ are
first integrals of $H_0$. They are independent first integrals except for points on the surface 
$x_1=\pm y_2\sqrt{y_1^2+y_2^2-1},\;x_2=\mp y_1 \sqrt{y_1^2+y_2^2-1}$ 
(which includes, in
particular, the origin and the periodic orbit $x_1=x_2=0,\ y_1^2+y_2^2=1$).
The unperturbed invariant manifolds are given by $G_1=G_2=0$.

Given $\epsilon \geq 0$, for $i=1,2$, we denote by $F_i^u$ (resp. $F_i^s$) the restriction of
$G_i$ to the invariant manifolds $W^{u}({\bf 0})$ (resp. $W^{s}({\bf 0})$).
For $\epsilon$ small, the invariant manifolds $W^{u/s}({\bf 0})$ can be represented as graphs
$g_{u/s}:\mathbb{R}^2 \to \mathbb{R}^4$, $g_{u/s}(\psi_0,\theta_0)
= (\psi_0, \theta_0, F_1^{u/s}(\psi_0,\theta_0), F_2^{u/s}(\psi_0,\theta_0))$.
Each component of the graph $g_{u/s}$ defines a 2-dimensional surface in $\mathbb{R}^3$,
they are referred below by $F_1^{u/s}$-graph and $F_2^{u/s}$-graph of  $W^{u}({\bf 0})$, 

The splitting function $(\DF{1}, \DF{2})$ is defined by
\begin{equation} \label{splfun}
\DF{i}(\psi_0,\theta_0)=F_i^{u}(\psi_0,\theta_0) - F_i^s(\psi_0,\theta_0), \quad i=1,2.
\end{equation}

The splitting function (\ref{splfun}) can be expanded as
$$
\DF{i} = \DFt{i} +  \DFt[2]{i} + \dots,
$$
where  $\DFt[k]{i}(\psi_0,\theta_0) = \epsilon^k M_k(\psi_0,\theta_0)$,
$|M_k|=\mathcal{O}(1)$.  Hence, $(\DFt{1}, \DFt{2})$ is
the first-order Poincar\'e-Melnikov approximation (in powers of $\epsilon$) of the splitting function.

Below, we perform direct numerical computations of the invariant manifolds to
obtain approximations $\Fn{i}{u/s}$ of the components of the graph function,
$i=1,2$. From them we compute numerical approximations $\DFn{i}$ of
the components of the splitting function $\DF{i}$. 

Finally, the first-order approximation $(\DFt{1}, \DFt{2}))$ can be expanded in 
Fourier series in $(\psi_0,\beta_0)$.  Truncating them we obtain
approximations that can be evaluated symbolically.  In Section~\ref{Sec:ValMel}
we compare the results obtained symbolically from suitable truncations of
$(\DFt{1}, \DFt{2})$ with the numerical approximations $(\DFn{1},\DFn{2})$.

\begin{remark} \label{remark_teogen}
There are several theoretical works concerning the splitting of invariant
manifolds in presence of a quasi-periodic forcing, we refer to
\cite{DelGelJorSea97,SimVall01,DelGut05,DelGonGut14,DelGonGut14-2,DelGonGut14-3}.
A common hypothesis is that all the Fourier harmonics in $t$ and all the Taylor
series terms in ${\bf x},{\bf y}$ appear in the corresponding expansions. Then
they use generic analytic decay of the coefficients to bound the dominant
term of the Melnikov function. The perturbation considered in this work,
although does not have all the required terms, behaves similarly. In future
works we plan to investigate the effect of absence of harmonics and/or Taylor
terms in the perturbation and the consequences it has in the behaviour of the
splitting of the invariant manifolds and in the dynamics around them. In
particular, a higher order Melnikov analysis could be needed to describe the
splitting in such a situation. 
\end{remark}

\section{Numerical computations of the splitting: dominant harmonics and
nodal lines} \label{sect_num}

We present some numerical computations concerning the
invariant manifolds $W^{u/s}(\bf 0)$ and their splitting for small values of $\nu$. 
We compute $\DFn{1}$ for a mesh of points in a fundamental domain (see below) as the
difference of the value $\tilde{F}_1^u$ obtained for a point on $W^u(\bf 0)$ and the
``corresponding'' point on $W^s(\bf 0)$. We describe below how to assign the
corresponding point by using coordinates in a fundamental domain of the invariant
manifolds. Similarly, we also compute $\DFn{2}$. 

It is useful to consider the Poincar\'e section $\Sigma=\max
\{y_1^2 + y_2^2\}=\max R_2^2$, see (\ref{polar}).  Note that the invariant
manifolds $W^{u/s}(\bf 0)$ of the unperturbed system ($\epsilon=0$) intersect
$\Sigma$ in the curve $x_1=x_2=0$, $y_1^2 + y_2^2=2$. This is no longer true for
$\epsilon >0$ because of the changes of order $\epsilon$ due to the perturbation (see Fig.~\ref{wsf1f2}).
Moreover, there is a (exponentially small in $\nu$) splitting of the
invariant manifolds $W^{u/s}(\bf 0)$ for $\epsilon \neq 0$.

The illustrations in this section are for $\gamma=(\sqrt{5}-1)/2$ (golden frequency) and
for values of $\nu$ of the form $\nu_i=2^{-i}$, $i \geq 0$. For the
computation of the invariant manifolds and their splitting we proceed as
follows: 
\begin{enumerate}
\item We consider a fundamental domain of $W^{u}(\bf 0)$. This is given by a
2-dimensional torus $\mathcal{T}$.

\item The propagation of $\mathcal{T}$ up to $\Sigma$ gives a 2-dimensional
torus, say $\mathcal{T}_\Sigma$. The invariant manifolds $W^{u/s}(\bf 0)$ in
$\mathbb{R}^4$ are then given as the $\tilde{F}_1^{u/s}$ and the $\tilde{F}_2^{u/s}$-graphs over
$\mathcal{T}_\Sigma$. The initial ``angle'' and ``time'' phases $\psi_0$ and
$\theta_0$ are local coordinates in $\mathcal{T}_\Sigma$.

\item To get the $\Fn{1}{u/s}$ and $\Fn{2}{u/s}$-graphs over $\mathcal{T}_\Sigma$
we propagate a set $\{\tilde{\psi}_{0,k}, \tilde{\theta}_{0,j}\}$, $0\leq k,j \leq 512$,
of initial points in $\mathcal{T}$ (i.e. a total number of $2^{18}$ initial
conditions) until they reach the Poincar\'e section $\Sigma$. Concretely we select
the initial conditions as follows. We fix $R_2=10^{-12}$, set
$R_1=R_2(1-R_2^2/2)$ and define $y_1^u=R_2 \cos(\tilde{\psi}_0), \ y_2^u=R_2
\sin(\tilde{\psi}_0), \ x_1^u=R_1 \cos(\tilde{\psi}_0),  \ x_2^u=R_1 \sin(\tilde{\psi}_0)$. This gives
an initial condition on $W^u$. By symmetry, $y_1^s=y_1^u, \
y_2^s=y_2^u, \ x_1^s=-x_1^u, \ x_2^s= -x_2^u$ defines an initial
condition on $W^s$.

\item For the propagation step, the numerical integration is performed using an
ad-hoc implemented high-order Taylor time-stepper scheme with quadruple precision. 

\item To compute the difference (i.e. the splitting) between $W^{u}(\bf 0)$
and $W^{s}(\bf 0)$ we need to compare them at the same points of
$\mathcal{T}_\Sigma$. Hence, we select an equispaced mesh of angles $\psi_0$ and $\theta_0$
within $\mathcal{T}_\Sigma$, and refine the initial conditions in
$\mathcal{T}$ (we select the initial guess from the set of previously computed
points in $\Sigma$) using a Newton method.
\end{enumerate}

To give some illustrations we choose $\nu=2^{-4}$ and $\nu=2^{-6}$. For those
two values of $\nu$ the $\Fn{1}{s}$-graph (resp. $\Fn{2}{s}$-graph) of the stable manifold
$W^{s}(\bf 0)$ over $\mathcal{T}_\Sigma$ is shown in Fig.~\ref{wsf1f2} left (resp. right).

\begin{figure}[p]
\psfrag{F1}{$\Fn{1}{s}$}
\psfrag{F2}{$\Fn{2}{s}$}
\psfrag{psi0}{$\psi_0$}
\psfrag{beta}{$\theta_0$}
\begin{center}
\epsfig{file=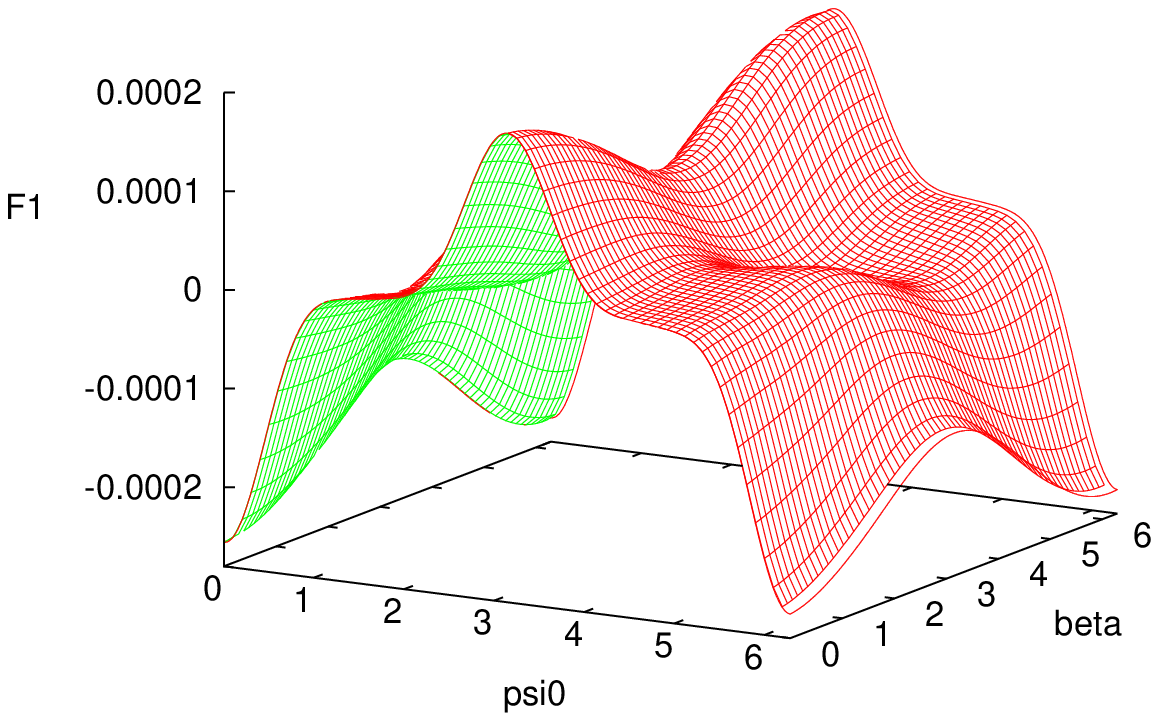,width=75mm,angle=0}
\epsfig{file=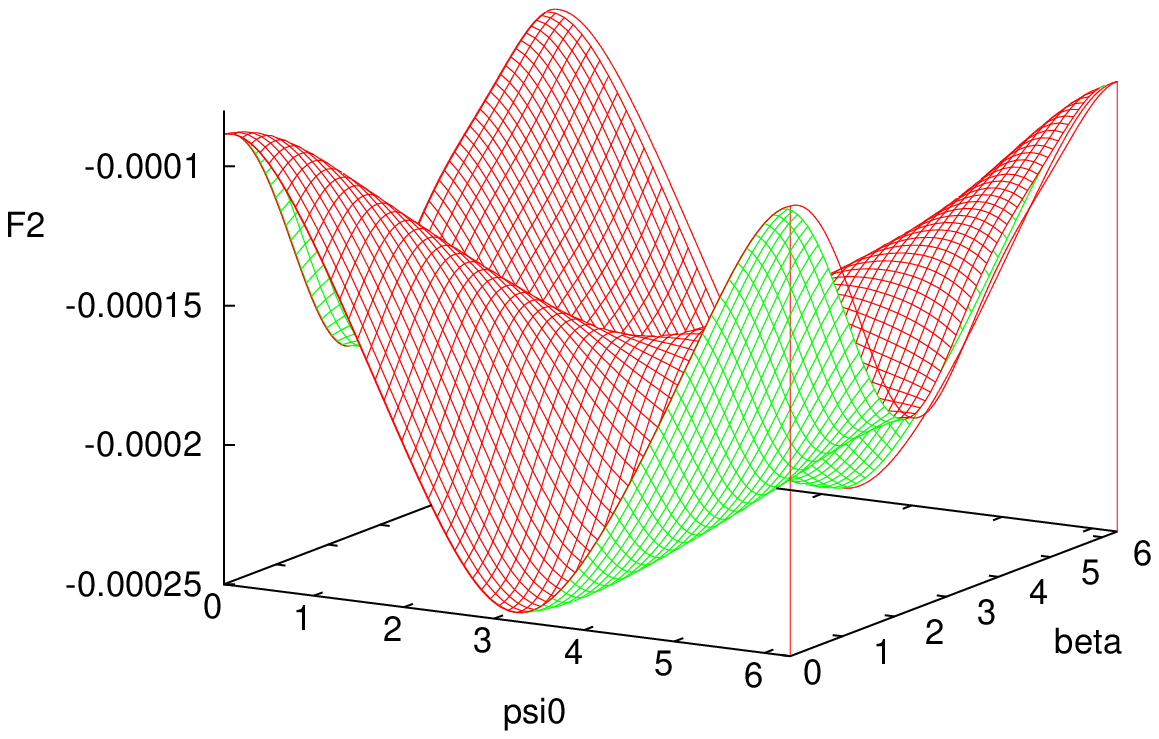,width=75mm,angle=0} \\[-11pt] 
\epsfig{file=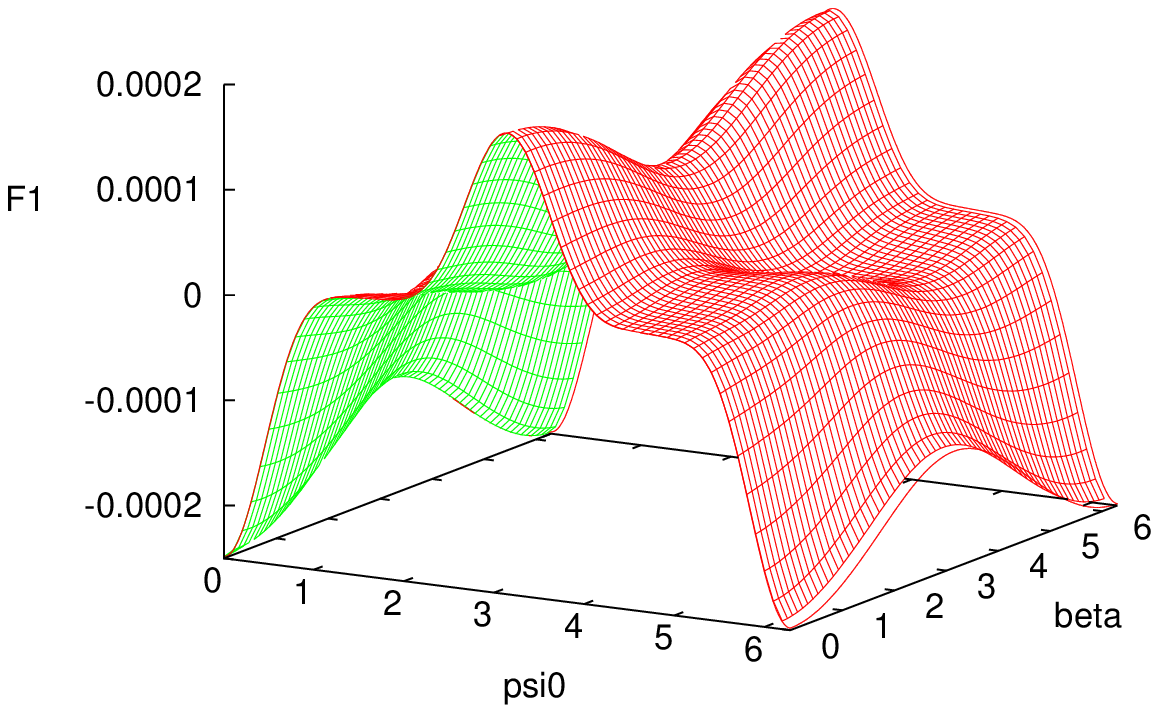,width=75mm,angle=0}
\epsfig{file=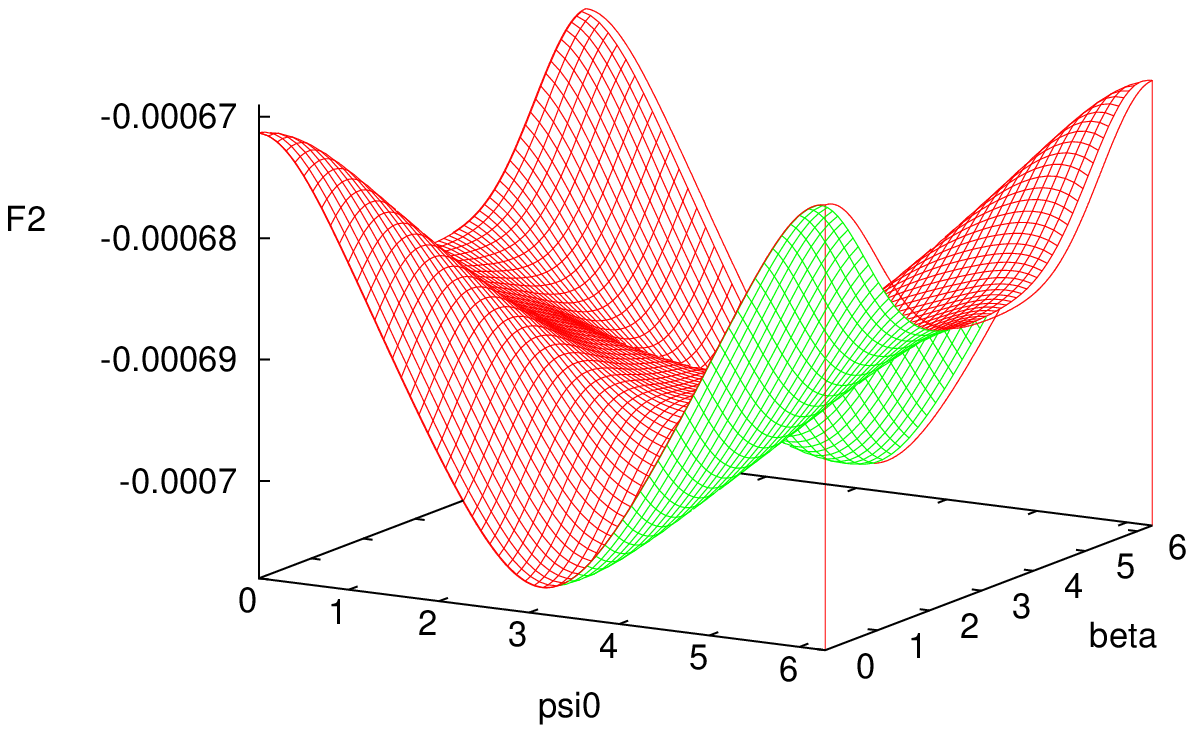,width=75mm,angle=0}
\caption{First row: for $\nu=2^{-4}$ we represent the $\Fn{1}{s}$-graph (left) and
the $\Fn{2}{s}$-graph (right) of $W^{s}(\bf 0)$ over $\mathcal{T}_\Sigma$. Second row: the same for $\nu=2^{-6}$. } 
\label{wsf1f2}
\end{center}
\psfrag{dF1}{$\DFn{1}$}
\psfrag{dF2}{$\DFn{2}$}
\psfrag{psi0}{$\psi_0$}
\psfrag{beta}{$\theta_0$}
\begin{center}
\epsfig{file=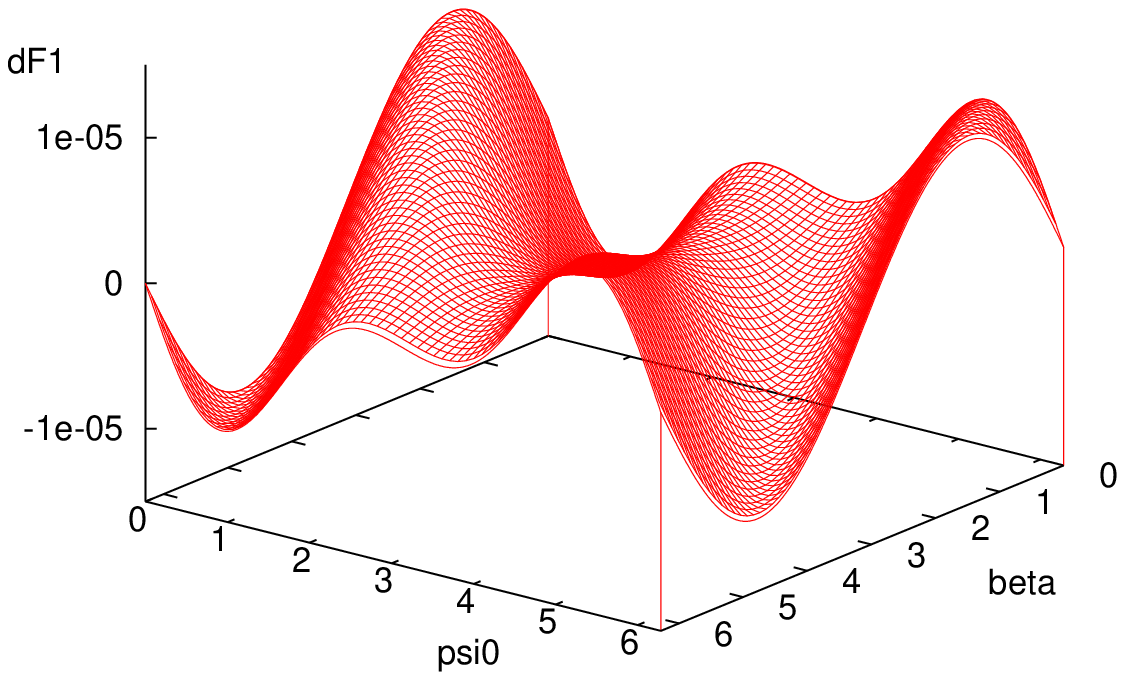,width=75mm,angle=0}
\epsfig{file=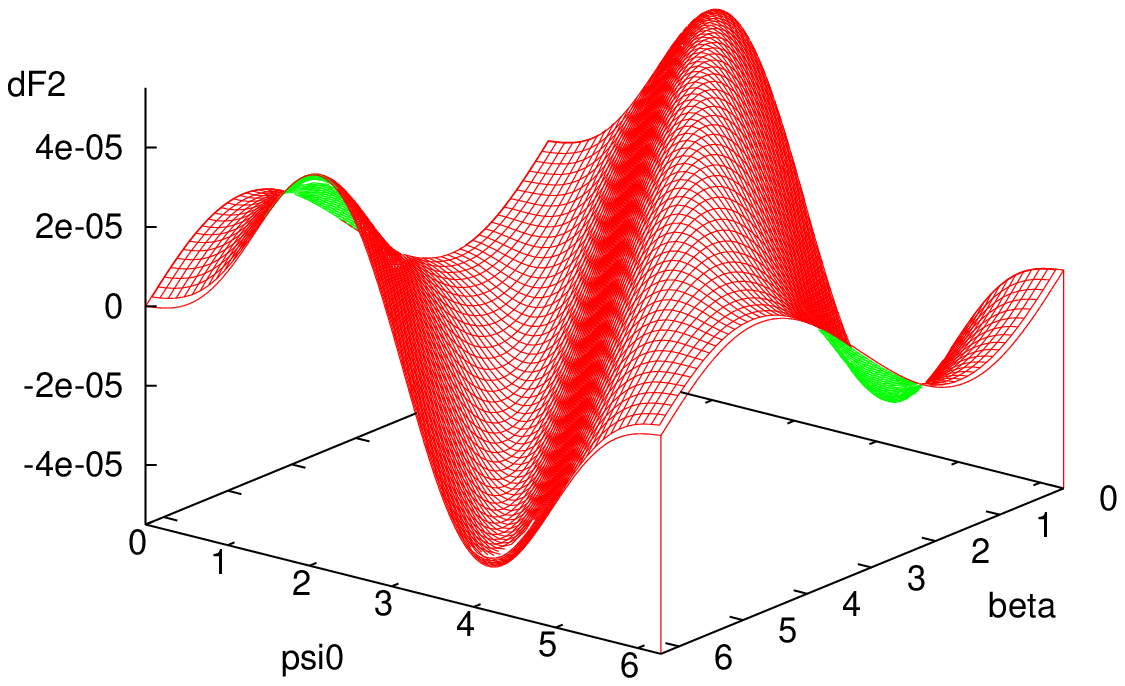,width=75mm,angle=0} \\[-11pt] 
\epsfig{file=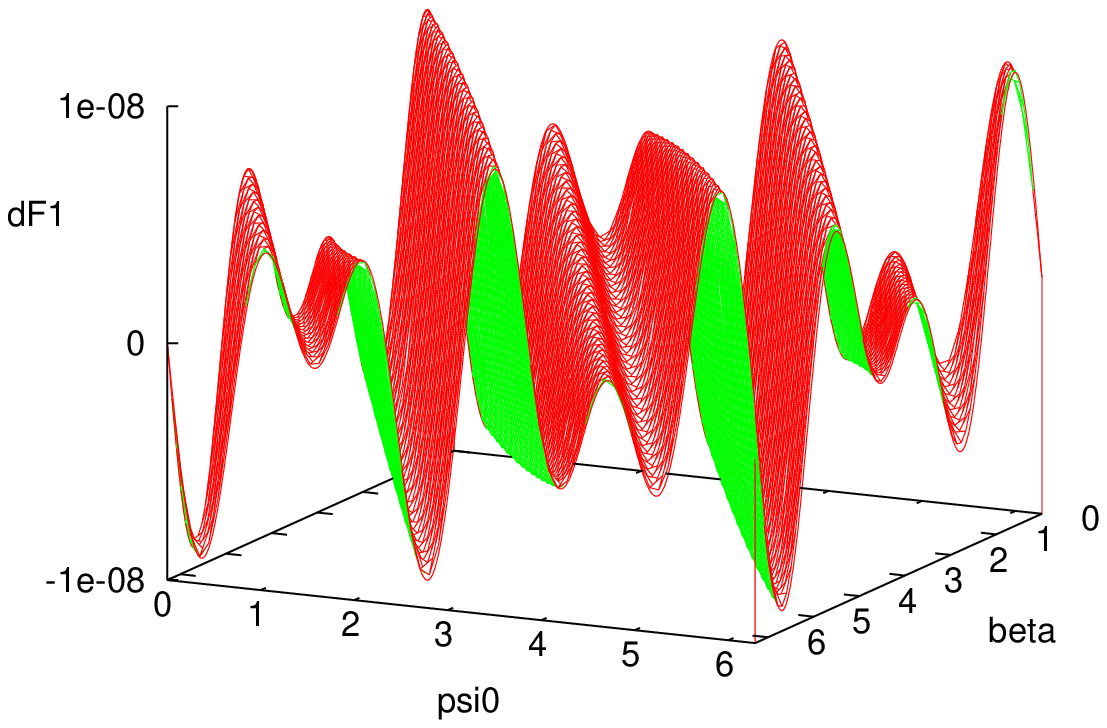,width=75mm,angle=0}
\epsfig{file=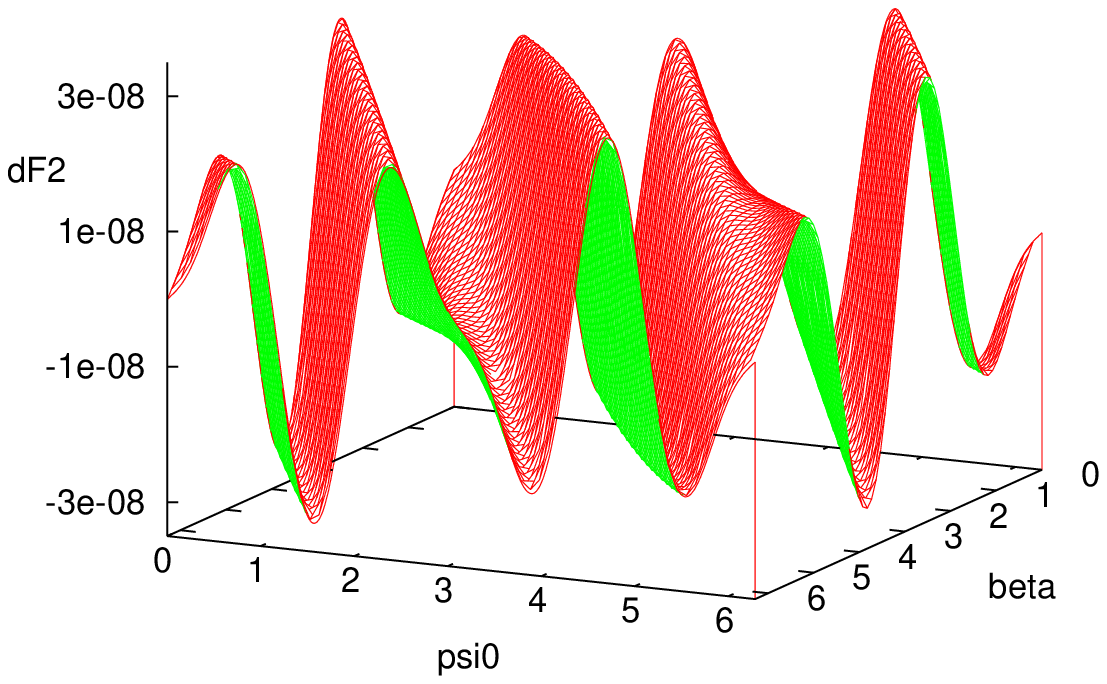,width=75mm,angle=0}
\caption{First row: For $\nu=2^{-4}$ we represent the $\DFn{1}$ (left)
and the $\DFn{2}$ (right).  Second row: $\DFn{1}$ and
$\DFn{2}$ for $\nu=2^{-6}$. } 
\label{spf1f2}
\end{center}
\end{figure}

There are no appreciable differences (using the scale of the plots) between
the graphs corresponding to $W^{s}(\bf 0)$ shown in Fig.~\ref{wsf1f2} and the corresponding
plots for the graphs of the unstable manifold $W^{u}(\bf 0)$. This is because
the splitting $(\DF{1},\DF{2})$ becomes exponentially small with respect to
$\nu$. We show in Fig.~\ref{spf1f2} the splitting $(\DFn{1},\DFn{2})$ for the
same values of $\nu$ as in Fig.~\ref{wsf1f2}. 

We note that while the graphs remain similar for those selected values of $\nu$
(although the vertical range changes for the $\tilde{F}_2$-graph representations), see
Fig.~\ref{wsf1f2}, the dominant harmonic of the Fourier expansion with respect
to $(\psi_0,\theta_0) \in \mathcal{T}_{\Sigma}$ of $\DFn{1}$ has changed from
$\nu=2^{-4}$ to $\nu=2^{-6}$, see Fig.~\ref{spf1f2}. A change of the dominant
harmonic of $\DFn{2}$ for these two values of $\nu$ is also observed.  Moreover,
for $\nu=2^{-6}$ the dominant harmonic of $\DFn{1}$ is different from the
dominant harmonic of $\DFn{2}$, as can be appreciated from the number of
oscillations of the left/right plots of the second row of Fig.~\ref{spf1f2}.
Concretely, for $\nu=2^{-4}$ the $(1,1)$ harmonic dominates for both $\DFn{1}$
and $\DFn{2}$, while for $\nu=2^{-6}$ the $(3,5)$-harmonic dominates for
$\DFn{1}$ and the $(2,3)$ harmonic dominates for $\DFn{2}$.

We can look for the so-called nodal lines. These are the zero level curves of
$\DF{1}$ or $\DF{2}$, i.e. where either the $F_1$-splitting
or the $F_2$-splitting vanishes. For $(\psi_0,\theta_0) \in \mathcal{T}_{\Sigma}$
the nodal lines for some values of $2^{-4.301} \leq \nu \leq 2^{-2.443}$ are shown in
Fig.~\ref{nodals1}. The nodal lines for some smaller values of $\nu$, up to
$2^{-6.235}$, are shown in Fig.~\ref{nodals2}. The values of $\nu$ shown have been selected
so that a change of $10^{-3}$ in $\log_2(\nu)$ produces a topological change of the nodal lines.
The intersections between the nodal lines correspond to homoclinic points and
the changes in the topology of the nodal lines correspond to 
passages from a dominant harmonic to another one (either in $\DF{1}$ or in $\DF{2}$),
 see \cite{SimVall01}. Hence, when decreasing $\nu$ many changes of dominant harmonic
have been detected.  We summarize them in
Table~\ref{taulanubif}. Concretely, we detect a topological change of the $\DFn{1}$ or
$\DFn{2}$ nodal lines for $\nu \in (\nu_1,\nu_2)$. The values of
$\nu_1$ and $\nu_2$ and the dominant harmonics at $\nu_1$ and $\nu_2$ are shown
in the table.

As expected the dominant harmonics of $\DFn{1}$ and $\DFn{2}$ are the elements of the
Fibonacci sequence, since they are related to the best approximants of the
golden number frequency $\gamma$. Observe that the appearance of a new harmonic
happens first for $\DFn{1}$ and later for $\DFn{2}$. These appearances take place
alternatively.  Later on we will estimate the changes in $\DF{1}$ and $\DF{2}$
carefully. The fact that the harmonics in $\DF{1}$ and $\DF{2}$ coincide for large
ranges of $\nu$ has some dynamical consequences in the diffusion properties
(see Appendix~\ref{split-difusion}).

\begin{table}
\begin{center}
\begin{tabular}{|c|c|l|}
\hline
& & \\[-0.4cm]
$-\log_2 \nu_2$ &  $-\log_2 \nu_1$ &  Change of the dominant harmonics of $\DFn{1}, \DFn{2}$ \\
& & \\[-0.4cm]
\hline
2.443   &  2.444 & \hspace{2.5cm} (1,0), (1,0)  $\longrightarrow$  \textcolor{red}{(1,1)}, (1,0) \\
2.676   &  2.677 & \hspace{2.5cm} (1,1), (1,0)  $\longrightarrow$  (1,1), \textcolor{blue}{(1,1)} \\
4.112   &  4.113 & \hspace{2.5cm} (1,1), (1,1)  $\longrightarrow$  \textcolor{red}{(1,2)}, (1,1) \\
4.300   &  4.301 & \hspace{2.5cm} (1,2), (1,1)  $\longrightarrow$  (1,2), \textcolor{blue}{(1,2)} \\
5.133   &  5.134 & \hspace{2.5cm} (1,2), (1,2)  $\longrightarrow$  \textcolor{red}{(2,3)},(1,2) \\
5.428   &  5.429 & \hspace{2.5cm} (2,3), (1,2)  $\longrightarrow$  (2,3), \textcolor{blue}{(2,3)} \\
5.971   &  5.972 & \hspace{2.5cm} (2,3), (2,3)  $\longrightarrow$  \textcolor{red}{(3,5)}, (2,3) \\
6.234   &  6.235 & \hspace{2.5cm} (2,3), (3,5)  $\longrightarrow$  (3,5), \textcolor{blue}{(3,5)} \\
\hline
\end{tabular}
\caption{The third column lists the dominant harmonic of $\DFn{1}$ and the dominant harmonic of
 $\DFn{2}$ (both separated by a comma) for the value $\nu=\nu_2$ (left hand side of the arrow)
and for $\nu=\nu_1$ (right hand side of the arrow). 
The $(m_1,m_2)$ harmonic corresponds to the frequency $m_1 \psi_0 - m_2 \theta_0$ of the Fourier
expansion of $\DFn{i}$.
The values of $\nu_2$ and
$\nu_1$, shown in the first and second columns, are such that a bifurcation takes place for $\nu \in (\nu_1,\nu_2)$. We highlight the changes in
the dominant harmonic of $\DFn{1}$ in red while those of $\DFn{2}$
are marked in blue. }
\label{taulanubif}
\end{center}
\end{table}

\begin{figure}[p]
\begin{center}
\epsfig{file=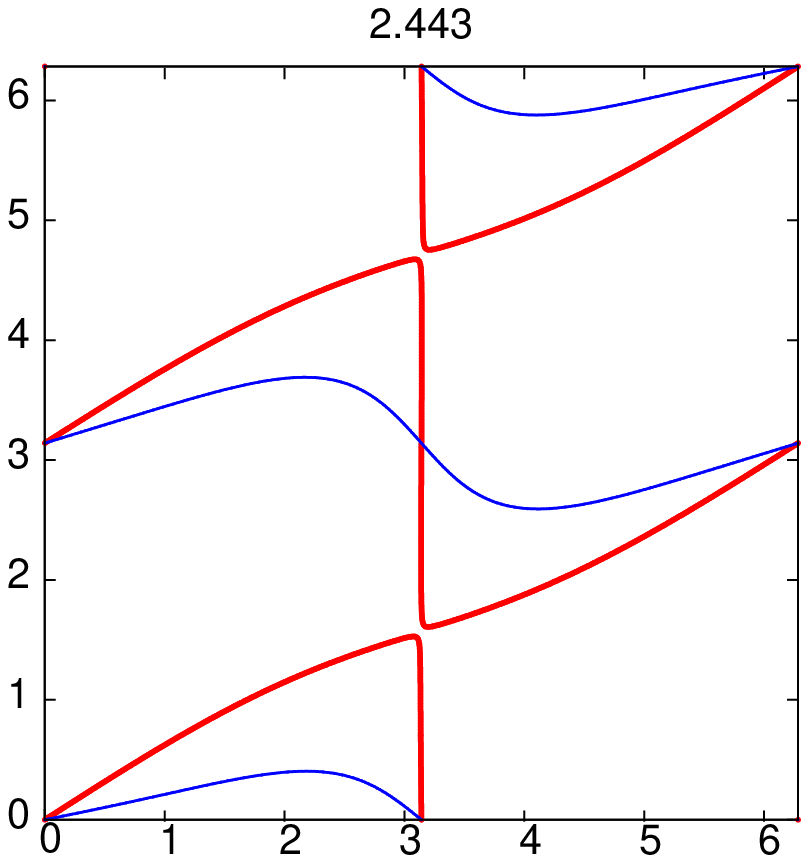,width=75mm,angle=0}
\epsfig{file=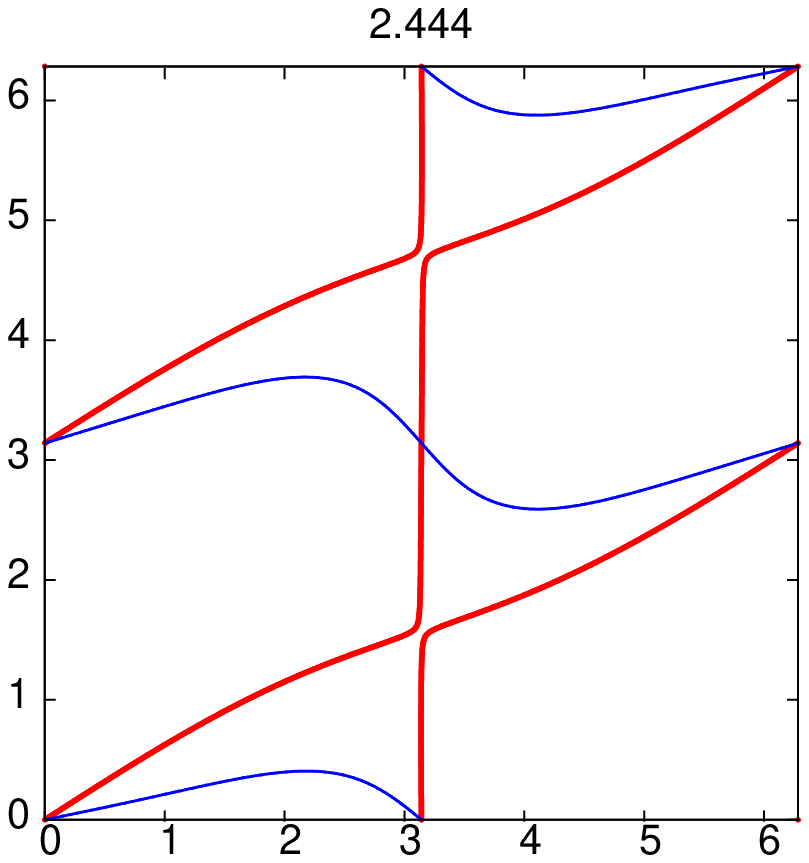,width=75mm,angle=0}
\epsfig{file=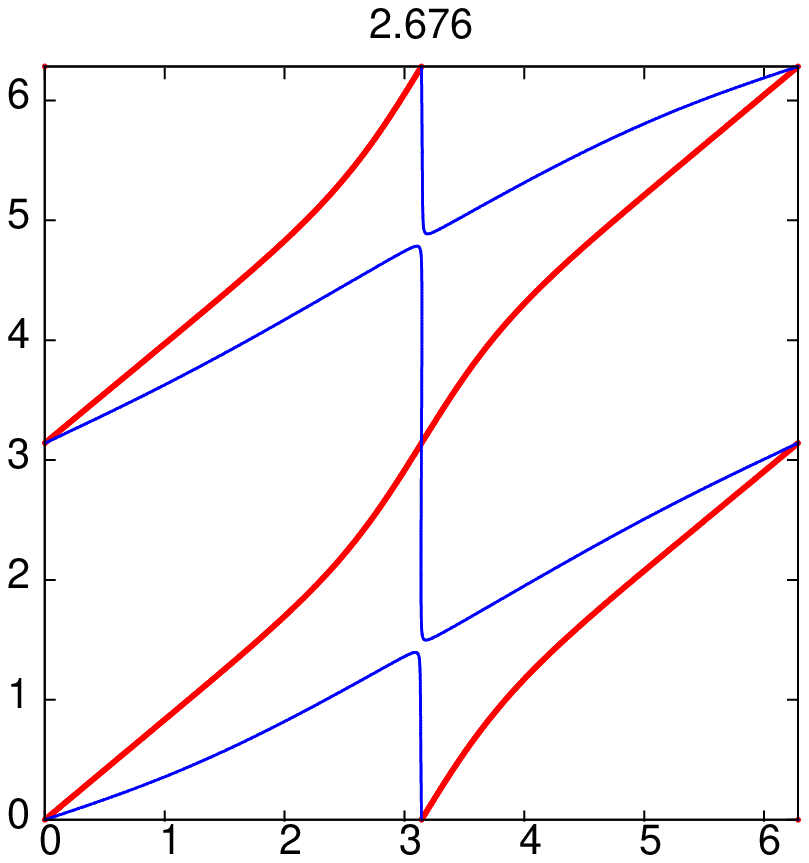,width=75mm,angle=0}
\epsfig{file=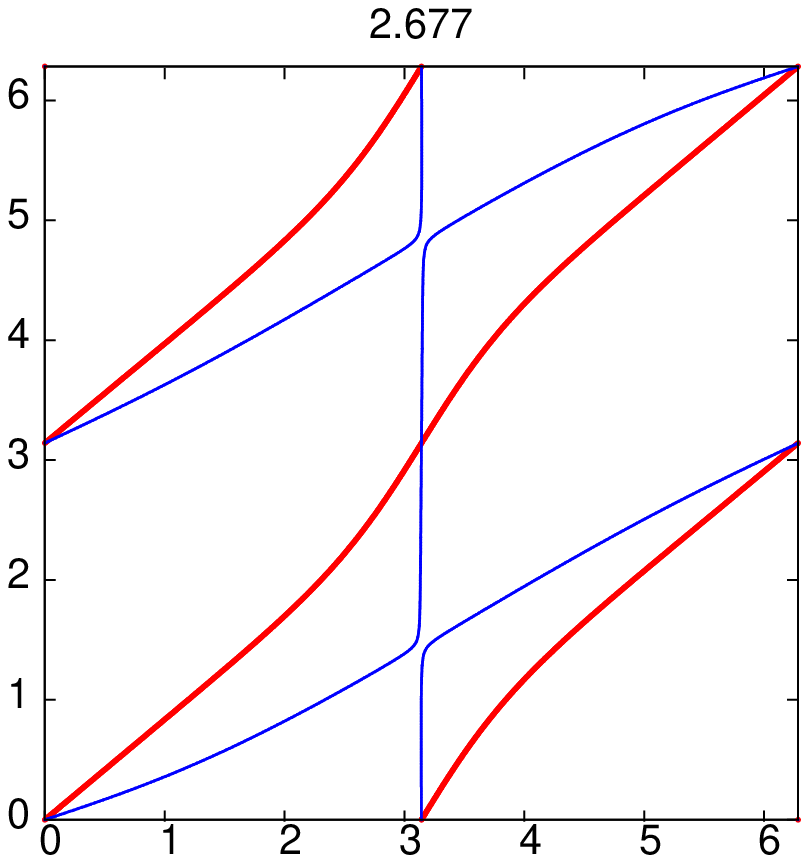,width=75mm,angle=0}
\epsfig{file=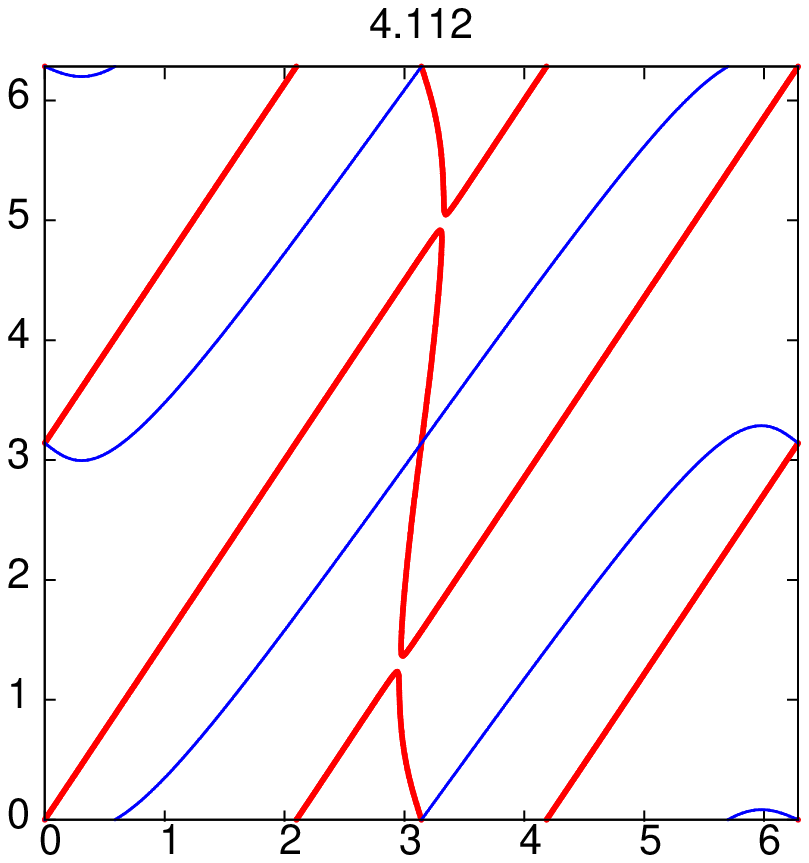,width=75mm,angle=0}
\epsfig{file=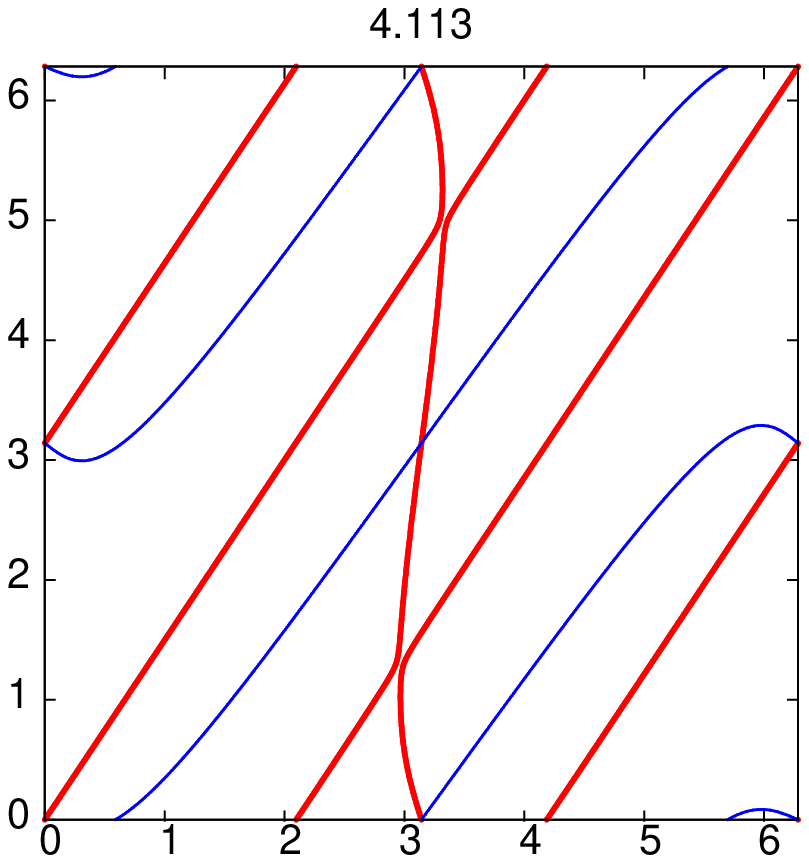,width=75mm,angle=0}
\epsfig{file=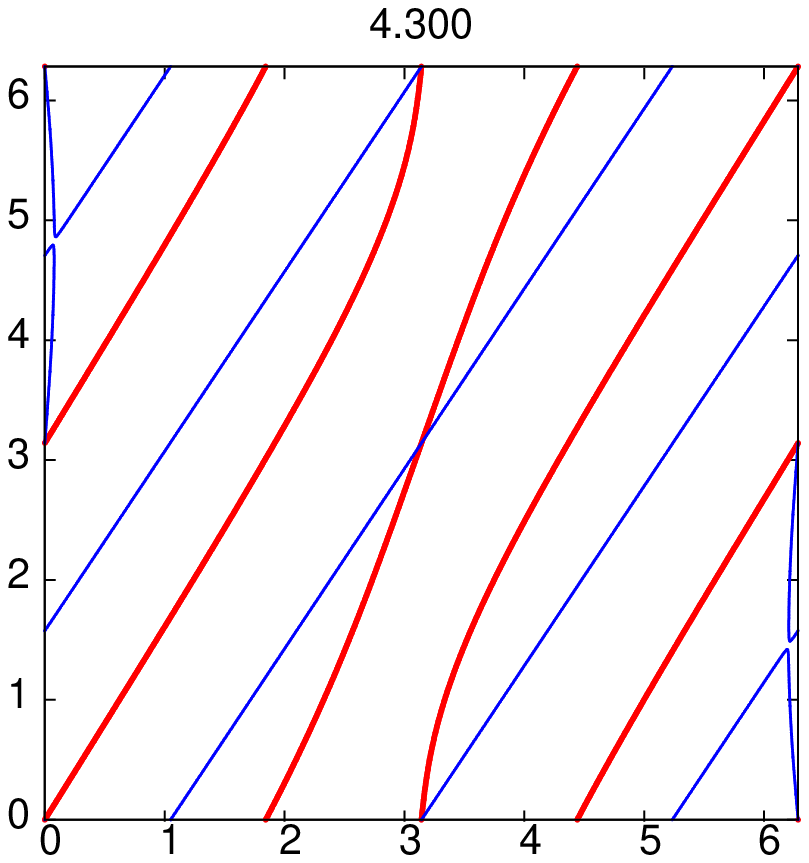,width=75mm,angle=0}
\epsfig{file=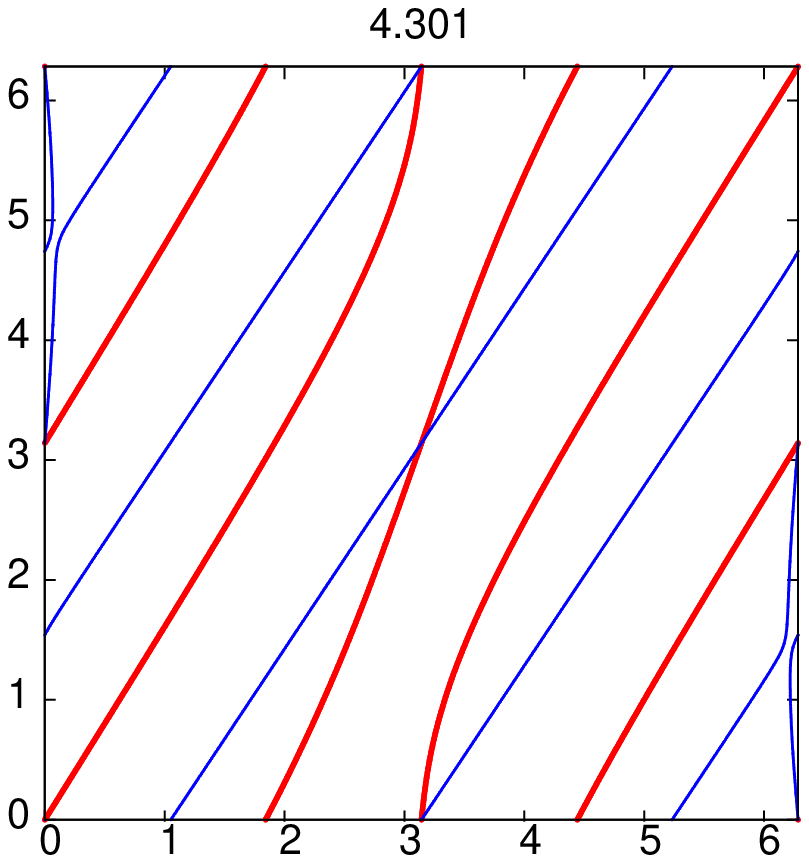,width=75mm,angle=0}
\caption{Nodal lines of $\DFn{1}$ are shown in red. In blue we
represent the ones related to $\DFn{2}$. The squares $[0,2\pi]^2$ represents
the tori parameterized by $(\psi_0,\theta_0)$. Each row corresponds to two different
values of the decreasing parameter $\nu$: before (left) and after (right) the bifurcation (values of
$\nu \geq 2^{-4.301}$).} 
\label{nodals1}
\end{center}
\end{figure}

\begin{figure}[p]
\begin{center}
\epsfig{file=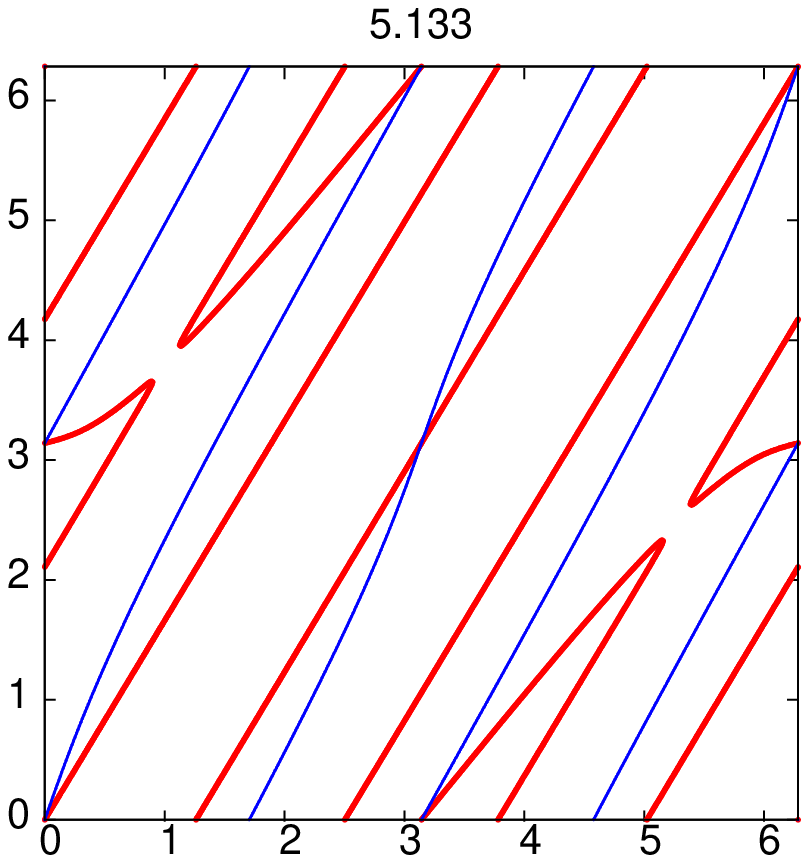,width=75mm,angle=0}
\epsfig{file=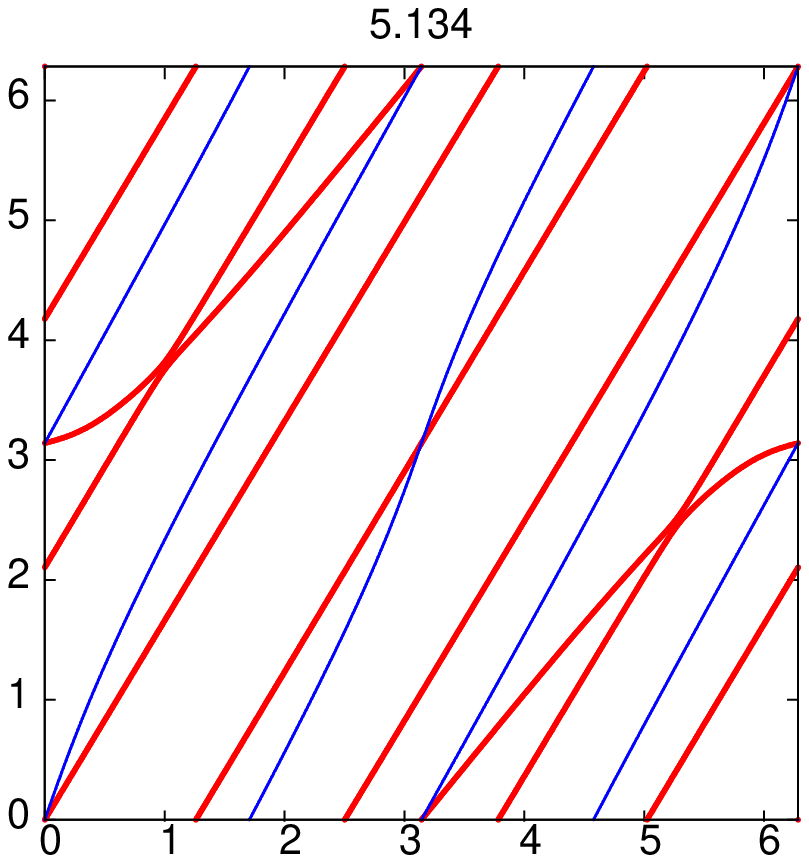,width=75mm,angle=0}
\epsfig{file=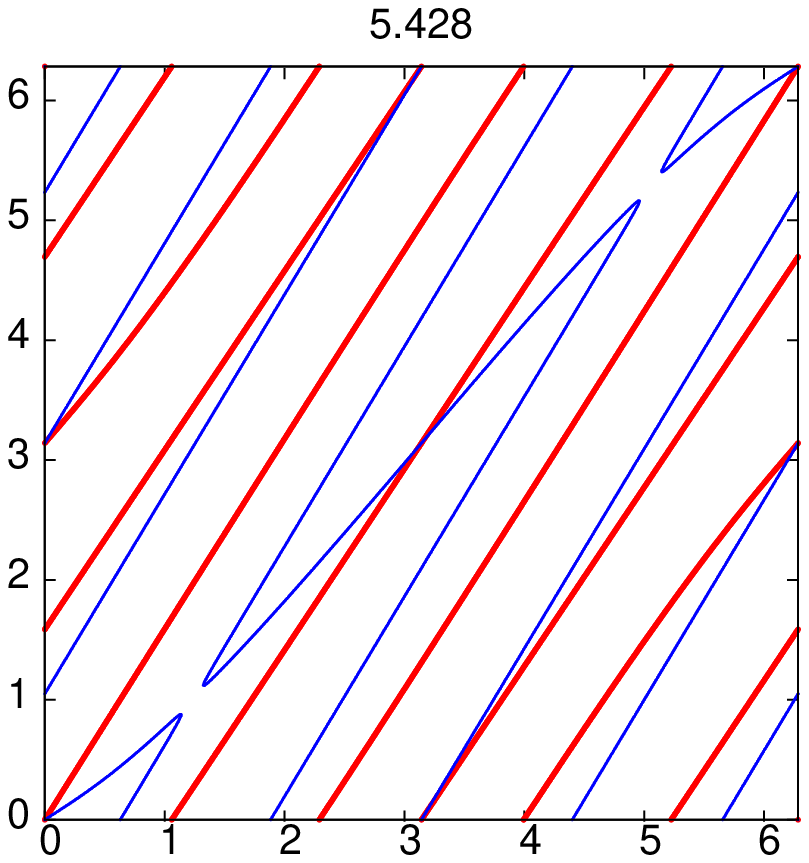,width=75mm,angle=0}
\epsfig{file=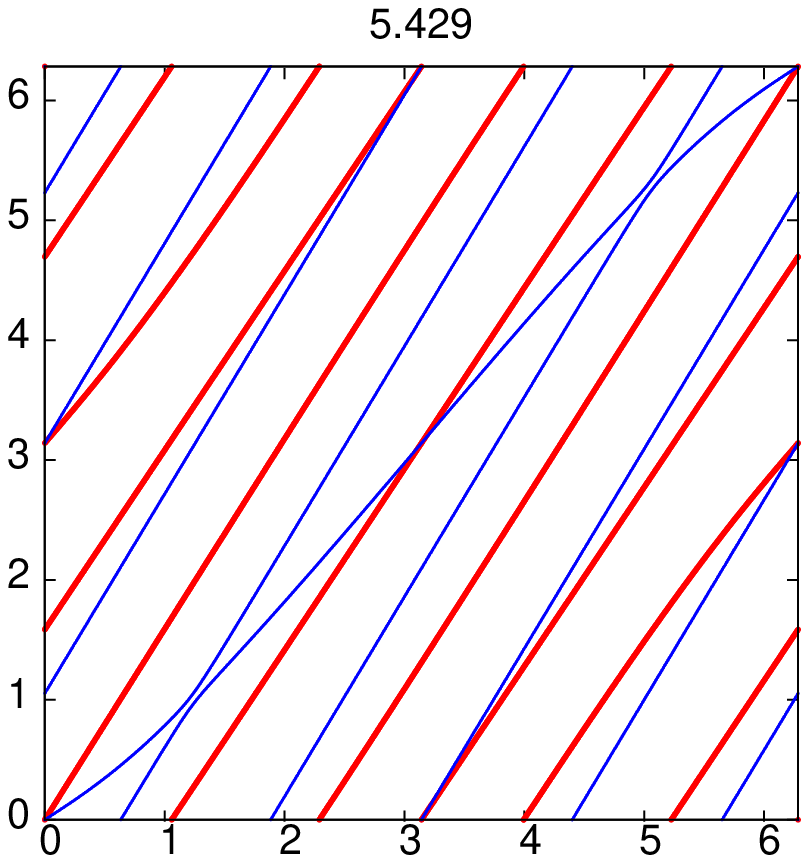,width=75mm,angle=0}
\epsfig{file=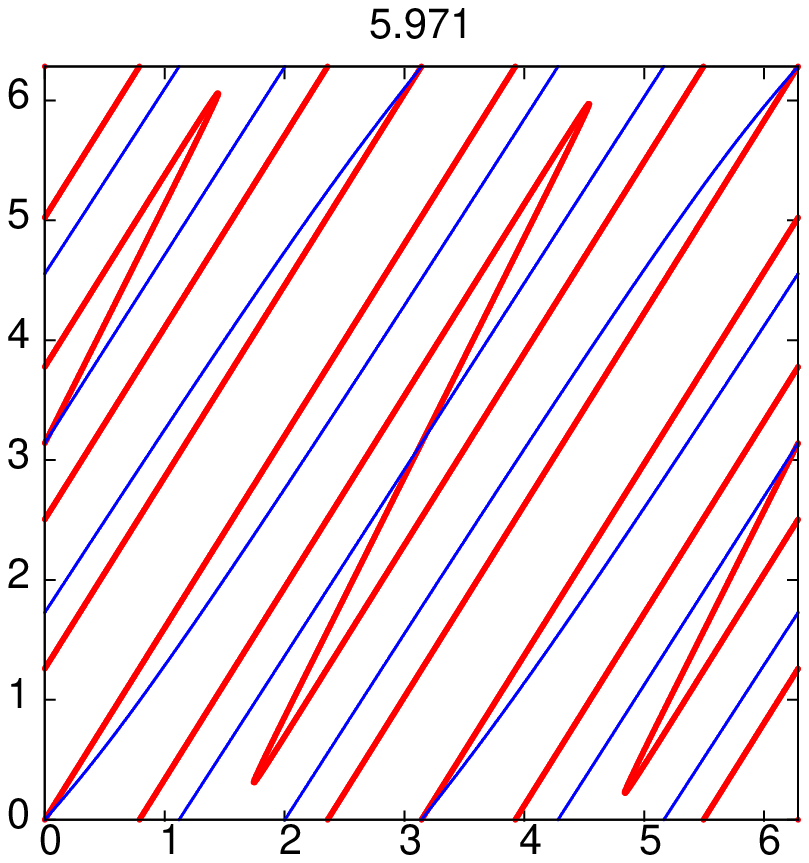,width=75mm,angle=0}
\epsfig{file=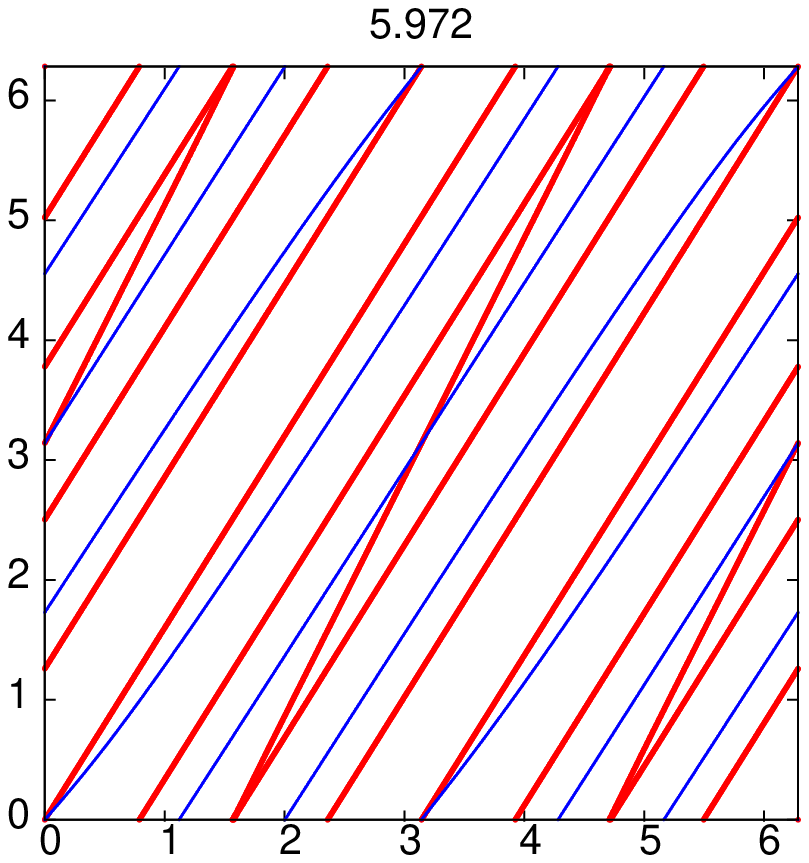,width=75mm,angle=0}
\epsfig{file=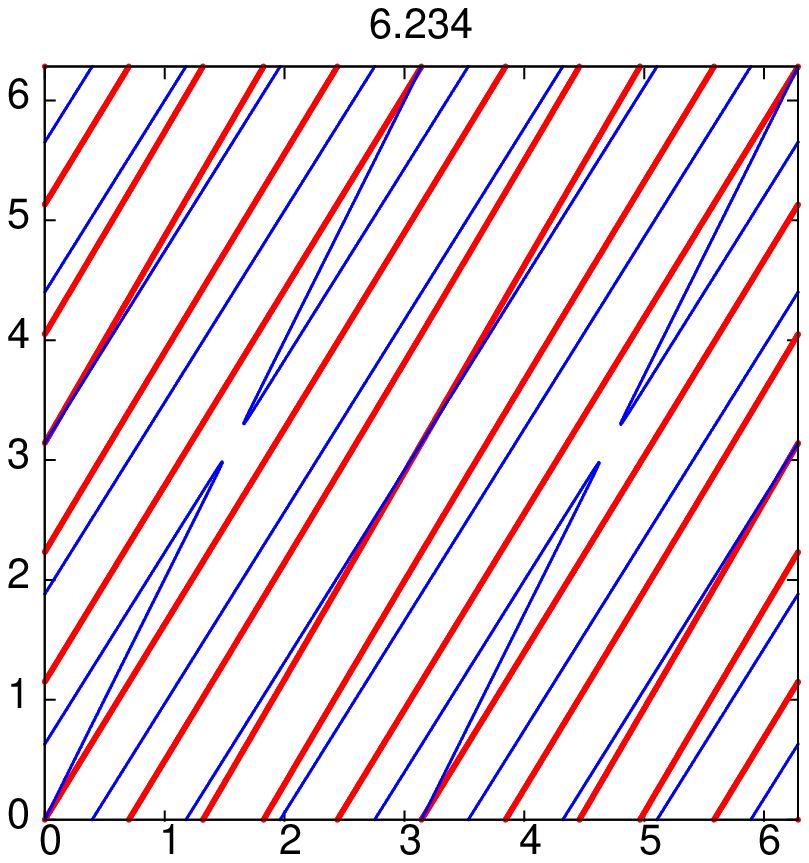,width=75mm,angle=0}
\epsfig{file=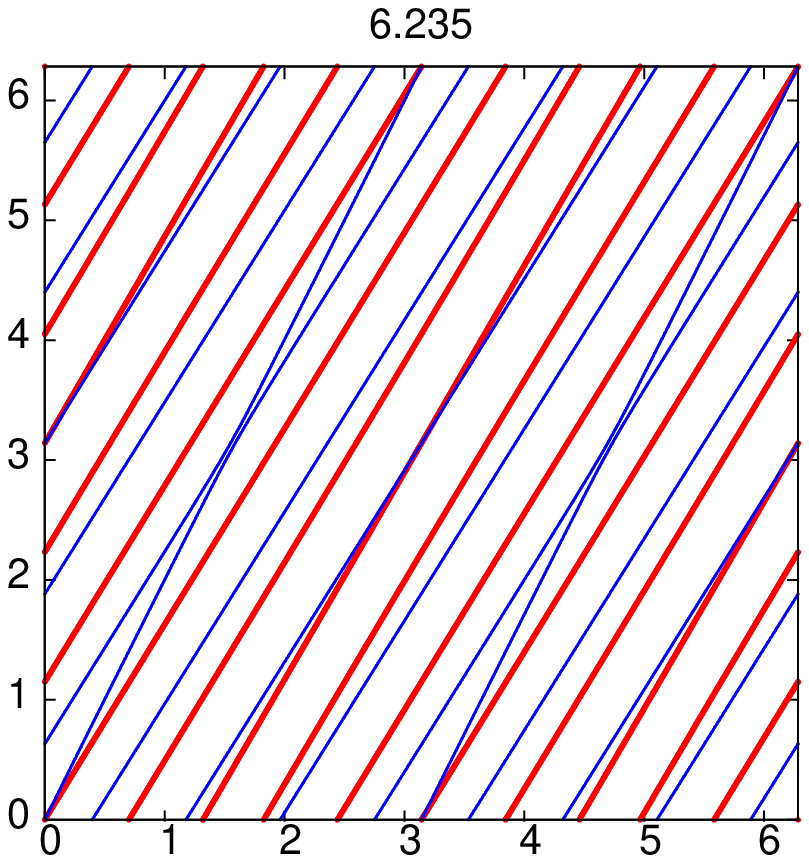,width=75mm,angle=0}
\caption{Continuation of Fig.~\ref{nodals1}: nodal lines for values of
$2^{-6.235} \leq \nu<2^{-4.301}$.} 
\label{nodals2}
\end{center}
\end{figure}

\section{The splitting of the invariant manifolds} \label{Sec:Splitting} 

For the unperturbed system $H_0$ given in (\ref{H0}) the 2-dimensional invariant
manifolds of the origin $W^{u/s}(\bf 0)$ coincide. But this is no longer true
for the perturbed system (\ref{perturbedsystem}), the Hamiltonian perturbation
$\epsilon H_1(\bx,\by,t)$ causes the splitting of the invariant manifolds
$W^{u/s}(\bf 0)$. We will study the behaviour of the splitting of $W^{u/s}(\bf
0)$ as $\nu \rightarrow 0$ (i.e. as the system reduces hyperbolicity) for a
fixed $\epsilon \neq 0$.

As it is well-known the splitting is related to the nearest singularities to the
real axis of the time-parameterization of the unperturbed homoclinic trajectory.
In our case the singularities are located at $\tau_0=\pm \ii \pi/2\nu$.
Moreover, the perturbation $H_1(\bx,\by)$ adds a space singularity $\rho$
located at $y_1=d$ and a time singularity related to $\hat{\theta}=\pm \ii \log(c+\sqrt{c^2-1})$
that restricts the domain of convergence of $f(\theta)$. The three
singularities play a role in the asymptotic behaviour of the splitting as will
be shown later on. We refer to \cite{GuaSea12} where a quasi-periodic perturbation with
state singularities was considered.

\subsection{The derivation of the Poincar\'e-Melnikov function} \label{derMel}

To obtain the expression for the Poincar\'e-Melnikov vector we proceed in a standard way so we just
shortly describe its derivation. 

Let $t_0\in \mathbb{R}$, $\zeta^s_0=(x_0^s, y_0^s)\in W^{s}(\rm{0})$,  
$\zeta^u_0=(x_0^u, y_0^u)\in W^{u}(\rm{0})$ and
$\zeta^{s,u}(t)=(x^{s,u}(t), y^{s,u}(t))$ be the solutions of the Hamiltonian system $H_0+ \epsilon H_1$ such that 
$$
\zeta^{s,u} (t_0)=\zeta^{s,u}_0 =(x_0^{s,u}, y_0^{s,u}).
$$
Clearly we have $\lim_{t\to \infty}\zeta^{s}(t) = \lim_{t\to -\infty}\zeta^{u}(t) =0 $. Then, for $i=1,2$,
\begin{align*}
G_i(\zeta^{s}(t)) & -G_i(\zeta^{s}_0) = \int_{t_0} ^t \frac{d}{dt}[G_i\circ \zeta^s](s)\, ds \\
& =  \int_{t_0} ^t DG_i(\zeta^s(s)) [J DH_0^\top (\zeta^s(s)) + \epsilon J DH_1^\top (\zeta^s(s)) ]\, ds 
= \epsilon \int_{t_0} ^t \{G_i, H_1 \} \circ \zeta^s (s)\, ds ,
\end{align*}
and taking limit when $t$ goes to $\infty$ we get 
$$
G_i(\zeta_0^s)=G_i(\zeta^{s} (t_0)) = -\epsilon \int_{t_0} ^\infty \{G_i, H_1 \} \circ \zeta^s (s)\, ds.
$$
In the same way 
$$
G_i(\zeta_0^u)=G_i(\zeta^{u} (t_0)) = \epsilon \int_{t_0} ^{-\infty} \{G_i, H_1 \} \circ \zeta^u (s)\, ds.
$$
Actually $\zeta^{s,u}(t)$ depend on $\epsilon$. Let $\zeta^{0}(t)$  denote the solution of the system when $\epsilon=0$, with initial condition $\zeta_0$ for $t=t_0$.
We use $(\psi_0,\theta_0)$, which parameterize the unperturbed manifold, to also parameterize
$W^{s,u} (0)$, and we consider $\zeta^{s}_0$, $\zeta^{u}_0$ the points on 
$W^{s} (0)$, $W^{u} (0)$ parameterized by $(\psi_0,\theta_0)$. 
Recall from Section~\ref{Sec:splfun} that $F_i^{u/s}$ denotes the restriction of $G_i$ to $W^{u/s}({\bf 0})$. 
By perturbation theory of invariant manifolds we have
\begin{align*}
\zeta^{s} (t) -\zeta ^0(t) &= O(\epsilon) , \quad \mbox{ uniformly in $t$ for } t\in [t_0, \infty), \\
\zeta^{u} (t) -\zeta ^0(t) &= O(\epsilon) , \quad \mbox{ uniformly in $t$ for } t\in ( -\infty, t_0]. 
\end{align*}
Since $(\psi_0, \theta_0)$ and $t_0$ are not independent we assume that $t_0=0$,
that is, the corresponding Poincar\'e-Melnikov integrals depend on the two
phase variables $\psi_0$ (initial ``angle'' phase, see (\ref{homoH0})) and
$\theta_0$ (initial ``time'' phase, see (\ref{perturbedsystem})).  
Therefore the splitting function is given by 
$$
F_i^u(\zeta^{u}_0) -F_i^s(\zeta^{s}_0) =\epsilon \int_{-\infty} ^\infty \{G_i, H_1 \} \circ \zeta^0 (s)\, ds
+ O(\epsilon^2) =: \epsilon M_i(\psi_0,\theta_0) + \mathcal{O}(\epsilon^2).
$$
Below we denote by $(\DFt{1}(\psi_0,\theta_0),\DFt{2}(\psi_0,\theta_0)) = (\epsilon M_1(\psi_0,\theta_0),\epsilon M_2(\psi_0,\theta_0))$ the so-called (first order) 
Poincar\'e-Melnikov approximation function.

\subsection{The expression of the Poincar\'e-Melnikov integrals}  \label{sect_expr_PoincMel}

As before, see (\ref{perturbedsystem}), we write
$H_1(\bx,\by,t)=g(y_1)f(\theta)$ where the expansions of $f$ and $g$ are given
in (\ref{expcs}) and (\ref{tayg}), respectively. 
Since the Poisson brackets are 
$$
\{G_1,H_1\}= y_2 f(\theta) g'(y_1), \qquad  \{G_2,H_1\}= x_1 f(\theta) g'(y_1),
$$
and
\begin{equation} \label{expds}
g'(y_1) = \sum_{k \geq 0} d_k y_1^{4+k}, \quad \mbox{ where } d_k =(5+k)d^{-k-1},
\end{equation}
the Poincar\'e-Melnikov approximation of the splitting distance is 
\begin{align}
\DFt{1}&=4 \epsilon \int_{-\infty}^{\infty}  \sin(t+\psi_0) \, f(\gamma t+\theta_0) \, \sum_{k \geq 0}  \frac{ \sqrt{2^{k+1}} \, d_k \, (\cos (t+\psi_0) )^{4+k}}{(\cosh(\nu t))^{5+k}} dt , \nonumber \\[-0.25cm]
 & \label{DG1exp} \\[-0.25cm]
\DFt{2}&=-4 \epsilon \int_{-\infty}^{\infty}  f(\gamma t+\theta_0) \, \sum_{k \geq 0}  \frac{ \sqrt{2^{k+1}} \, d_k \, (\cos (t+\psi_0) )^{5+k} \, \sinh(\nu t)}{(\cosh(\nu t))^{6+k}} dt, \nonumber
\end{align}
where, for simplicity, we have not written the dependence on $\psi_0,\theta_0$ in $\DFt{i}$.  

Since the Poincar\'e-Melnikov integral is linear with respect to the
perturbation $\epsilon H_1(\bx,\by,t)$ we can write
$M(\psi_0,\theta_0)=(M_1(\psi_0,\theta_0),M_2(\psi_0,\theta_0))$ as an infinite
sum and analyse the contribution to the splitting of each individual term of
the series of $H_1$.

The Fourier series of the terms of the form $(\cos(\psi))^m$  and
$(\cos(\psi))^m \sin(\psi)$, for $m \in \mathbb{Z}^+$, that appear in the
previous equations are given by
\begin{equation}  \label{ccs}
(\cos(\psi))^m= \!\! \sum_{i=0}^{E\left(\frac{m}{2}\right)} \!\! a_{m,i} \cos((m-2i)\psi), \qquad  (\cos(\psi))^m \sin(\psi)= \!\! \sum_{i=0}^{E\left(\frac{m+1}{2}\right)} \!\!  b_{m,i} \sin((m+1-2i)\psi),
\end{equation}
where $E(x)$ denotes the integer part of $x$, and
\begin{align} \label{ccscoefs}
a_{m,i} &= \frac{1}{2^{m-1}} \left(\begin{array}{c}m\\i\end{array}\right), \ 0\leq i < m/2, \quad \mbox{ and } \quad a_{m,m/2}=\frac{1}{2^m} \left(\begin{array}{c}m\\m/2\end{array}\right) \mbox{ if } m \mbox{ even,} \\
b_{m,i} &= \frac{m+1-2i}{2^m(m+1)}    \left(\begin{array}{c}m+1\\i\end{array}\right),  \ 0\leq i \leq (m+1)/2. \nonumber
\end{align}

To compute the Poincar\'e-Melnikov integral for a general perturbation
$g(\bx,\by) f(\theta)$ the following comments apply:
\begin{itemize}
\item An expression of the form $x_1^{i_1}x_2^{i_2}y_1^{j_1}y_2^{j_2}$ in the Poisson bracket, when evaluated on the homoclinic orbit, becomes
\[ (-1)^{i_1+i_2}2^{(i_1+i_2+j_1+j_2)/2}(\cos(\psi))^{i_1+j_1}(\sin(\psi))^
{i_2+j_2}\frac{(\sinh(\nu t))^{i_1+i_2}}{(\cosh(\nu t))^{2i_1+2i_2+j_1+j_2}}.\]
The trigonometric terms can be reduced to the sum of expressions of the form
$(\cos(\psi))^m$ or $(\cos(\psi))^m\sin(\psi)$, depending on whether
$i_2+j_2$ is even or odd. In a similar way the hyperbolic terms can be reduced
to the sum of negative powers of $\cosh(\nu t)$ or to such a sum times
$\sinh(\nu t)$, depending on whether $i_1+i_2$ is even or odd.  

\item Using the expansions (\ref{ccs})-(\ref{ccscoefs}), and assuming the
expansion of the time-periodic part is \[ f(\theta)=\sum_{j\geq
0}a_j\cos(j\theta)+\sum_{j> 0}b_j\sin(j\theta), \qquad a_j,b_j \in \mathbb{R}\] 
the integrals required to evaluate $\DFt{1}, \DFt{2}$ can be reduced to integrals of  
the product of $(\cosh(\nu t))^{-n},n\geq 1$ or $(\cosh(\nu t))^{-n}\sinh(\nu
t),n\geq 2,$ by a function of the form 
\[ \cos(k\psi)\cos(j\theta),\quad \cos(k\psi)\sin(j\theta),\quad \sin(k\psi)\cos(j\theta)\quad {\mbox{or}}\quad \sin(k\psi)\sin(j\theta), \qquad   k,j \in \mathbb{Z}^+.\]

\item Recall that $\psi=t+\psi_0$ and $\theta=\gamma t+\theta_0$. Expanding $f(\theta)$ 
and $\cos(\psi)$ and taking into account that the integrals of odd functions in $\nR$ are zero,
the computation of $\DFt{i}$, $i=1,2$, reduces to the computation of integrals of the form 
\begin{equation} \label{I1I2}
I_1(s,\nu,n)=\int_{\nR} \frac{\cos(st)}{(\cosh(\nu
t))^n} \, dt , \ n \geq 1, \quad I_2(s,\nu,n)=\int_{\nR} \frac{\sinh(\nu t)\sin(st)}{(\cosh(\nu
t))^n} \, dt, \ n \geq 2,
\end{equation}
for $\nu \neq 0$ (we will only be interested in $\nu>0$), where we have
introduced the parameter $s = k\pm j\gamma$. 

\item Furthermore, one has 
$$\displaystyle{I_2(s,\nu,n)=\frac{s}{\nu(n-1)}I_1(s,\nu,n-1)}, \quad n\geq 2.$$
Hence, it suffices to compute $I_1(s,\nu,n)$.

\item One has
\[ I_1(s,\nu,1) = \frac{\pi}{\nu} \frac{1}{\cosh (s\pi/(2\nu))}, \qquad I_1(s,\nu,2) = \frac{s \pi}{\nu^2} \frac{1}{\sinh (s\pi /(2\nu))},\]
and, integrating by parts twice, one obtains
\[I_1(s,\nu,n)= \frac{s^2 + (n-2)^2 \nu^2}{\nu^2 (n-1)(n-2)} I_1(s,\nu,n-2), \quad n \geq 3.\]
That is,
\[\nu^n I_1(s,\nu,n) = \frac{\pi}{(n-1)! \cosh(s\pi/(2\nu))} P_{n-1}(s,\nu), \quad \text{ for $n$ odd,} \]
and
\[\nu^n I_1(s,\nu,n) = \frac{\pi}{(n-1)! \sinh(s\pi/(2\nu))} P_{n-1}(s,\nu),  \quad \text{ for $n$ even,} \]
where $P_j(s,\nu)$, $j \geq 0$, are the homogeneous polynomials of degree $j$ in $(s,\nu)$ 
that satisfy the recurrence 
\begin{equation} \label{recurP}
P_0(s,\nu)=1, \qquad P_1(s,\nu)=s, \qquad P_j(s,\nu)=(s^2+(j-1)^2 \nu^2) P_{j-2}(s,\nu), \ j \geq 2.
\end{equation}
In particular, we see that the terms in the series of $\DFt{i}$, $i=1,2$, decay
to zero at least as $\exp(-|s| \pi/(2\nu))$ as $\nu \rightarrow 0$. We note
that, however, the functions $\DFt{i}$, $i=1,2$, may decay in a slower way, see
Appendix~\ref{exempleduffing}.
\end{itemize}

At this point we have all the ingredients to produce an algorithm to obtain
expressions for $\DFt{1}, \DFt{2}$ with any accuracy. 

\begin{remark} \label{remark_k1k2}
The analyticity domain in the spatial coordinates $(\bx,\by)$ and the
analyticity strip in time $t$ of the perturbation $\epsilon H_1(\bx,\by,t)$ can
be, in general, of different size. 
Denote by $m(\bx,\by,\theta)$ a term of the Taylor-Fourier expansion of $H_1(\bx,\by,t)$, where $t=(\theta-\theta_0)/\gamma$. That is,
$m(\bx,\by,\theta)$ is a monomial of degree $k_1 \geq 0$ in $(\bx,\by)$ with the harmonic $k_2 \in \mathbb{Z}$ in $\theta$.
Assume that there exist $\rho_1,\rho_2>0$ such that the coefficient $m$ of this monomial satisfies 
\[ |m| \leq M \exp(-k_1 \rho_1 - |k_2| \rho_2),\]
with $M>0$ and where $(\bx,\by)$ belongs to a compact domain containing the
unperturbed real separatrices.\footnote{In particular, this assumption holds for the concrete example (\ref{perturbedsystem}) considered in this paper for $c>1$ and $d> \sqrt{2}$, see the expansions (\ref{expcs}) and (\ref{tayg}).}
Then, the contribution $T(k_1,k_2)$ of the
monomial to the Poincar\'e-Melnikov integral is of the form 
$$
T(k_1,k_2) \sim \epsilon A \nu^B \exp(-k_1 \rho_1 -|k_2| \rho_2) \exp\left(\frac{-|s| \pi}{2\nu} \right), \text{ with } s=k_1-|k_2|\gamma, A>0.
$$

We note that it may happen that $T(k_1,k_2)$ dominates the
behaviour of the splitting for $\nu$ small even if $k_1$ and $k_2$ (and the
total order $k=k_1+|k_2|$) are large provided that $k_1- |k_2| \gamma$ is small
enough. For example, consider $\rho:=\rho_1=\rho_2$ and assume that $\gamma$
verifies $|k_1 - |k_2| \gamma| > C |k|^{-\tau}$ with $\tau \geq 1$. 
Then $T(k_1,k_2) \sim T(k)=\epsilon A \nu^B \exp(-k
\rho) \exp (-C \pi/2 \nu k^{\tau} )$ and the largest contribution is
obtained for $k=k_* \sim (C \pi \tau/ 2 \rho \nu)^{1/(1+\tau)}$, which gives a term
$T(k_*) = \mathcal{O}(\exp(-c/\nu^{1/(\tau+1)}))$. This agrees, provided $\tau>1$,
with the exponentially small remainder obtained after an optimal number of
steps of the averaging procedure for a quasi-periodic function, see details in
\cite{Sim94}. When $\tau=1$ there are many terms that give the same
contribution and the exponentially small (in $\nu$) upper bound in the
averaging procedure gains an extra logarithmic term \cite{Sal04}.

\end{remark}

Summarizing, using (\ref{ccs}) and (\ref{ccscoefs}), we can rewrite (\ref{DG1exp}) as
\begin{align*} 
\DFt{1} &= \epsilon \int_{-\infty}^{\infty} \sum_{k \geq 0} \sum_{0 \leq 2 i \leq 4+k} \sum_{j \geq 0} d_k b_{4+k,i} c_j 2^{\frac{5+k}{2}} \frac{1}{(\cosh(\nu t))^{5+k}} \sin((k+5-2i)\psi) \cos(j \theta) \, dt, \\
\DFt{2} &= -\epsilon \int_{-\infty}^{\infty} \sum_{k \geq 0} \sum_{0 \leq 2 i \leq 5+k} \sum_{j \geq 0} d_k a_{5+k,i} c_j 2^{\frac{5+k}{2}} \frac{\sinh(\nu t)}{(\cosh(\nu t))^{6+k}} \cos((k+5-2i)\psi) \cos(j \theta) \, dt.
\end{align*}
Taking into account that $\psi= t + \psi_0$, $\theta=\gamma t + \theta_0$, and expanding the terms
$\sin(\ell (t+\psi_0)) \cos(j(\gamma t + \theta_0))$ and $\cos(\ell (t+\psi_0)) \cos(j(\gamma t + \theta_0))$, where $\ell=k+5-2i$,
in the previous expression one reduces to evaluate integrals $I_1(s,\nu,n)$ and $I_2(s,\nu,n)$, given by (\ref{I1I2}), where
$s=\ell \pm j \gamma$. Concretely, using the expansions
\begin{align*}
 \sin(\ell(t+\psi_0)) & \cos(j(\gamma t+\theta_0))=\frac{1}{2}\left[\sin((\ell+j\gamma)t+
   \ell \psi_0+j\theta_0)+\sin((\ell-j\gamma)t+\ell \psi_0-j\theta_0)\right] \\
 =&\frac{1}{2}\left[\sin((\ell+j\gamma)t)\cos(\ell \psi_0+j\theta_0)+
                    \cos((\ell+j\gamma)t)\sin(\ell \psi_0+j\theta_0)\right] \\
 &+\frac{1}{2}\left[\sin((\ell-j\gamma)t)\cos(\ell \psi_0-j\theta_0)+
                    \cos((\ell-j\gamma)t)\sin(\ell \psi_0-j\theta_0)\right], \\
 \cos(\ell(t+\psi_0))& \cos(j(\gamma t+\theta_0))=\frac{1}{2}\left[\cos((\ell+j\gamma)t+
   \ell \psi_0+j\theta_0)+\cos((\ell-j\gamma)t+\ell \psi_0-j\theta_0)\right] \\
 =&\frac{1}{2}\left[\cos((\ell +j\gamma)t)\cos(\ell \psi_0+j\theta_0)-
                    \sin((\ell +j\gamma)t)\sin(\ell \psi_0+j\theta_0)\right] \\
 & +\frac{1}{2}\left[\cos((\ell -j\gamma)t)\cos(\ell \psi_0-j\theta_0)-
                    \sin((\ell -j\gamma)t)\sin(\ell \psi_0-j\theta_0)\right],
\end{align*}
one obtains
\begin{align} \label{exprdG1dG2_1} 
\nonumber \DFt{1} & =   \displaystyle{\epsilon \sum_{j\geq 0} c_j \sum_{k\geq 0} 2^{\frac{3+k}{2}} d_k \!\! \sum_{0 \leq 2 i \leq 4+k} \!\!\!\!\! b_{4+k,i} \sum_{l = \pm 1} I_1(k\!+\!5\!\!-2i\!+\! l j \gamma, \nu, k\!+\!5) \sin((k\!+\!5\!-\!2i)\psi_0 \!+\! lj \theta_0), } \vspace{0.2cm} \\
\DFt{2} & =   \displaystyle{-\epsilon \sum_{j\geq 0} c_j \sum_{k\geq 0} 2^{\frac{3+k}{2}} d_k \!\! \sum_{0 \leq 2 i \leq 5+k} \!\!\!\!\! a_{5+k,i} \sum_{l = \pm 1} I_2(k\!+\!5\!-\!2i\!+\! l j \gamma, \nu, k\!+\!6) \sin((k\!+\!5\!-\!2i)\psi_0 \!+\! lj \theta_0).} 
\end{align}

\section{Comparison between the splitting and the Melnikov approximation} \label{Sec:ValMel}

Note that our example fits within a non-perturbative ($\epsilon$ is considered fixed)
singular (when $\nu \rightarrow 0$ the system loses hyperbolicity)
splitting case as described in \cite{DelRam99}. The fact that the splitting
is well-approximated by the Poincar\'e-Melnikov (vector) approximation $\epsilon
M(\psi_0,\theta_0)$ must be justified in this context since, a priori, the
$\mathcal{O}(\epsilon^2)$ error terms in the Poincar\'e-Melnikov approach can
dominate for small enough values of $\nu$. 

That is, in order to justify the use of the Poincar\'e-Melnikov approach one
has to estimate the relative error term by checking that the constant in
$\mathcal{O}(\epsilon^2)$ decays together with $\nu$ in an exponentially small
way, becoming dominated by the (exponentially small in $\nu$) term
$\mathcal{O}(\epsilon)$. The necessity of estimating the relative error was
observed in \cite{San82}, see also \cite{DelRam99,GelLaz01}. 

The rigorous justification of the validity of the Melnikov integral is an
interesting but difficult problem. In this work we are not going to deal with
it. Instead, in this subsection, we choose $\gamma=(\sqrt{5}-1)/2$ and we
compare the {\em amplitude of the splitting} $\DFn{1}$ and $\DFn{2}$ computed directly
using numerical methods (i.e.  computing the invariant manifolds and the
difference between them in a mesh of points) as explained in
Section~\ref{sect_num} with the values obtained by using the Poincar\'e-Melnikov
integral and the recurrences detailed in Section~\ref{sect_expr_PoincMel}. 

Let us clarify what we refer to by amplitude of the splitting. When we proceed
numerically we compute the maximum of the absolute value of the distance
between the invariant stable/unstable manifolds attained in a fundamental
domain (i.e. on a torus parameterized by the angles $(\psi_0,\theta_0)$).  On
the other hand, when we proceed by evaluating the first order
Poincar\'e-Melnikov approximation using the expressions (\ref{exprdG1dG2_1}),
we only take into account those terms that, for the value of $\nu$ considered, give a relative contribution larger
than $10^{-10}$ to the total sum. Adding the contributions of these harmonic
terms we obtain an approximation of $\DFt{i}$. In the following, both
quantities are referred as amplitude of the splitting and are denoted by
$|\DFt{i}|$.

For $\gamma=(\sqrt{5}-1)/2$ we have computed, for more than 1000 values of
$\log_2(\nu)$, the amplitude of the splitting using both approaches. The
results are displayed in Fig.~\ref{ampli2}.  Note that from
Remark~\ref{remark_k1k2}, since we have taken a constant type frequency $\gamma$,
we expect the contribution of each term of the Taylor-Fourier expansion of the
splitting function to be $\mathcal{O}(\exp(-c/\sqrt{\nu}))$. Accordingly we
display $\sqrt{\nu} \log(|\DFt{i}|/\epsilon)$ as a function of $\log_2(\nu)$
in the figure. The direct numerical computations are done up to $\nu
\lessapprox 2^{-7}$, since for smaller values of $\nu$ they require a large
number of digits and a large computing time. As an example, for $\epsilon=10^{-3}$, $\nu=2^{-10}$ 
one has $|\DFt{i}| = \mathcal{O}(10^{-32})$. However, we can compute the
Poincar\'e-Melnikov integral up to much smaller values of $\nu$.

\begin{figure}[ht]
\begin{center}
\epsfig{file=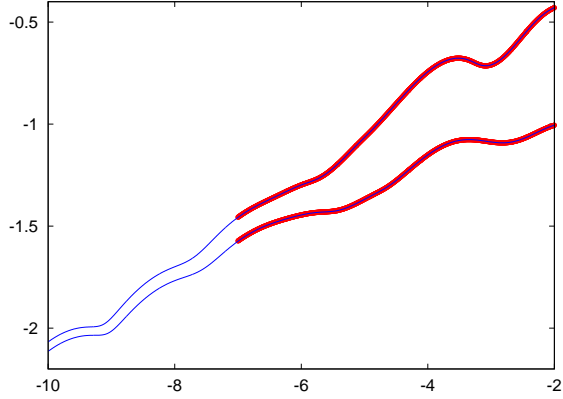,width=8cm} 
\end{center}
\caption{We represent $\sqrt{\nu} \log(|\DFt{i}|/\epsilon)$, for $i=1$ (bottom
curve) and $i=2$ (top curve), as a function of $\log_2(\nu)$.  In red points we
show the direct numerical computations of the amplitude of the splitting. The blue
line shows the values obtained using the Poincar\'e-Melnikov integrals
through the expressions derived theoretically to evaluate them. 
}
\label{ampli2}
\end{figure}

Note the excellent agreement between the numerical and the theoretical
methodologies. This accurate agreement supports the fact that the first order
Melnikov integral asymptotically describes the splitting. In particular, this
numerical check makes us confident to investigate the asymptotic behaviour of
the splitting for smaller values of $\nu$ using the first order approximation
of the splitting function given by the Poincar\'e-Melnikov integral. This is
the goal of the next section.

Finally, in Fig.~\ref{nterms} left we display the values of $\DFt{i}$,
$i=1,2$, for different $\nu \in [2^{-24}, 2^{-2}]$.  We have used a grid with
spacing $0.005$ in $\log_2(\nu)$. In the right plot we show the number of
harmonics that contribute to $\DFt{i}$, $i=1,2$. Each harmonic comes from
the contribution of different $(i,j,k)$-terms in expression
(\ref{exprdG1dG2_1}), where, as before, we only have taken into account those
terms with relative contribution larger than $10^{-10}$.  For the largest 
values of $\nu$ considered in the figure the dominant harmonic is computed as a
combination of up to $14$ different terms.  The values of $\DFt{i}$,
$i=1,2$, are shown in the left panel of the same figure. We observe, in
particular, that for $\nu < 2^{-12}$ the number of harmonics used to compute
the splitting function reduces to one with the exception of small intervals of
$\nu$ where two terms are used. This is related to the dominant harmonics of
the splitting function: for most values of $\nu$ only one term is relevant in
$\DFt{i}$ meaning that there is a dominant harmonic of the splitting
function, but when a change of dominant harmonic of the splitting function
takes place one has to consider two terms (meaning that two harmonics have a
similar contribution in such a range of $\nu$). Note also that the length of
the interval of $\nu$ where the computation of $\DFt{i}$ requires two harmonics decreases as
$\nu \rightarrow 0$. 

\begin{figure}[ht]
\begin{center}
\epsfig{file=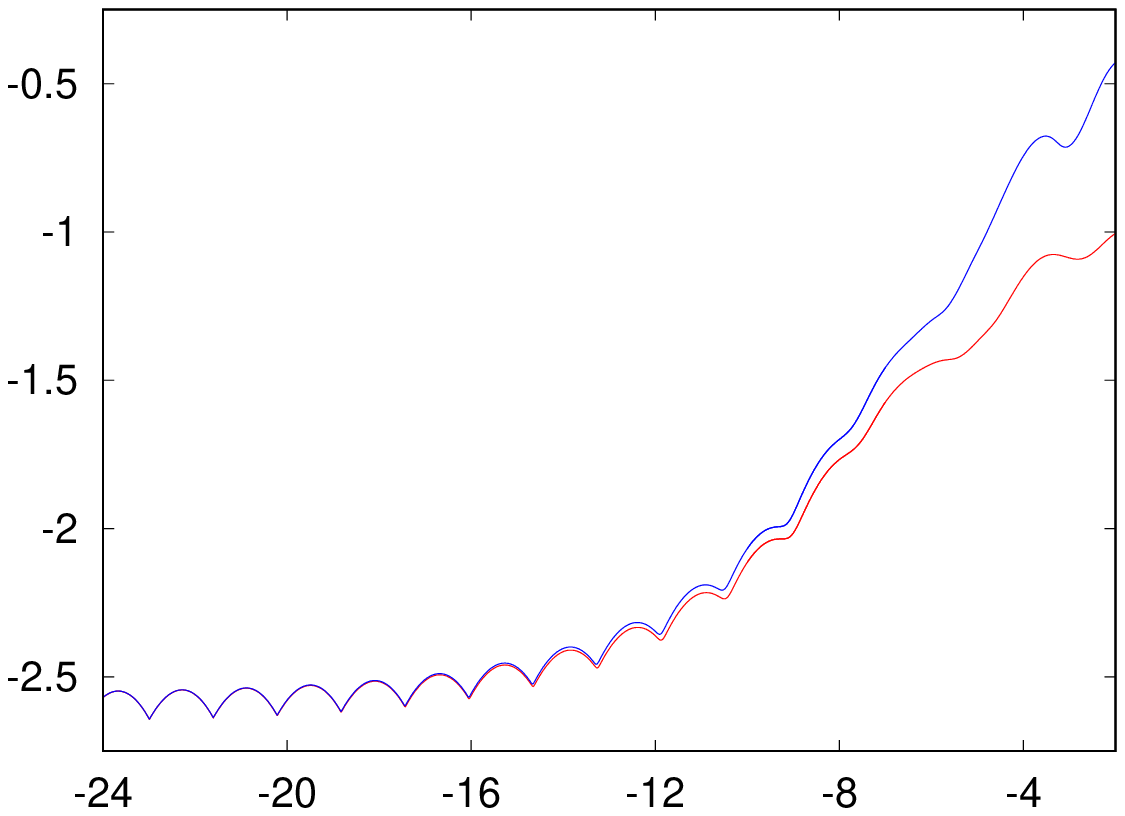,width=7cm} 
\epsfig{file=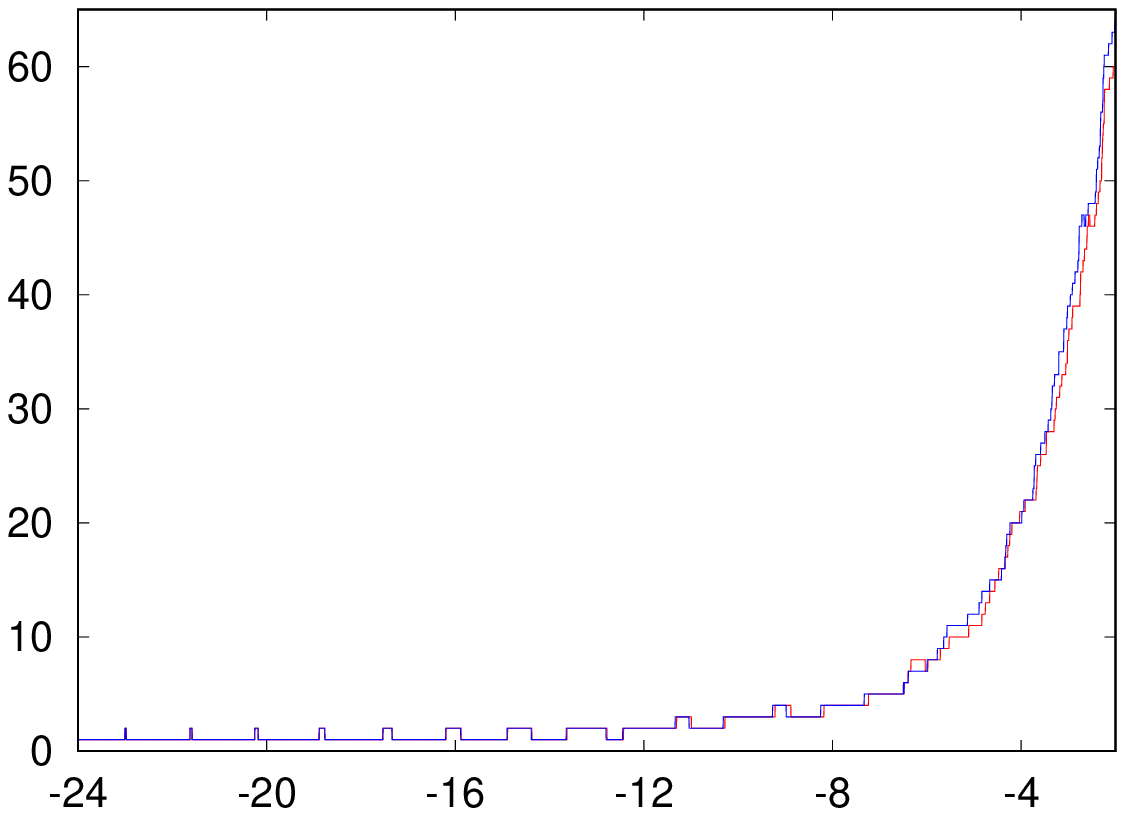,width=7cm} 
\end{center}
\caption{Left: We represent $\sqrt{\nu} \log(|\DFt{i}|/\epsilon)$, $i=1,2$.
Right: Number of harmonic terms considered to compute $\DFt{i}$.
In both panels $\log_2(\nu)$ ranges in the horizontal axis.
}
\label{nterms}
\end{figure}

\section{Asymptotic properties of the splitting behaviour}  \label{sect_PoincMel}

\subsection{The theoretical results}

Here we state the main theoretical results which are proven in the following subsections.

Consider system (\ref{perturbedsystem}) for $0 \leq \nu<\bar{\nu}\ll 1$ small enough,
with $\epsilon>0$, $c>1$, $d>\sqrt{2}$ and $\gamma\in \mathbb{R}\setminus\mathbb{Q}$.
From (\ref{exprdG1dG2_1}), we express $\DFt{i}$ as
\begin{equation} \label{dgiexp}
\DFt{i} =  \epsilon \sum_{m_1 \geq 0} \sum_{m_2 \in \mathbb{Z}} \hat{C}_{m_1,m_2}^{(i)} \sin(m_1 \psi_0 - m_2 \theta_0), \quad i=1,2,
\end{equation}
where $\hat{C}_{m_1,m_2}^{(i)} \in \mathbb{R}$.  We introduce the notation $C_{m_1,m_2}^{(i)}=|\hat{C}_{m_1,m_2}^{(i)}|$ to denote the amplitudes of the Fourier modes of $M_i(\psi_0,\theta_0)=\DFt{i}(\psi_0,\theta_0)/\epsilon$.

Given $m_1/m_2$, $m_j \in \mathbb{Z}\setminus \{0\}$, $j=1,2$, an approximant of
$\gamma$, let $c_{s,m_1/m_2} >0$ be the constant such that 
\[|s|=|m_1 - \gamma m_2| = \frac{1}{c_{s,m_1,m_2} m_1}.\] 
The constants $c_{s,m_1/m_2}$ are related to the arithmetic properties of $\gamma$ (see
Section~\ref{Sec:Otherfreq}). We shall denote the constant $c_{s,m_1/m_2}$ by
$c_{s,n}$ when $m_1/m_2$ is a best approximant of $\gamma$ (in the sense of the
continuous fraction expansions (CFE) of $\gamma$) and, in this case, $n$ refers
to the order of the best approximant. 

The following result provides a quantitative description of the way the
different harmonics contribute to $\DFt{i}$.

\begin{theorem} \label{mainresult1}
There exists a universal function $\Psi_1(L)$ such that
$$
\Psi_1(L)_{\mid L= c_{s,m_1/m_2} \nu m_1^2} \approx \sqrt{c_{s,m_1/m_2} \nu} \log C_{m_1,m_2}^{(1)}, 
$$
asymptotically when $\nu \rightarrow 0$. The function $\Psi_1(L)$ only depends on $\gamma$ through the additive term 
$k/\gamma$, where $k=c+\sqrt{c^2-1}$.
On the other hand, the function
\[\Psi_2(L)=\Psi_1(L)- \frac{\sqrt{L} \log{L}}{m_1}\]
satisfies
$$
\Psi_2(L)_{\mid L= c_{s,m_1/m_2} \nu m_1^2} \approx \sqrt{c_{s,m_1/m_2} \nu} \log C_{m_1,m_2}^{(2)}, 
$$
asymptotically when $\nu \rightarrow 0$.
\end{theorem}

The proof of Theorem~\ref{mainresult1} is given in Sections~\ref{7p2} and \ref{7p3}.	

If, from the arithmetic properties of $\gamma$, one can determine the
asymptotic behaviour of the constants $c_{s,m_1/m_2}$, assuming that they have
some defined asymptotic behaviour, then one can determine the dominant term (or
dominant terms) of the splitting. When $\nu \rightarrow 0$ the dominant terms
of the splitting are related to best approximants of $\gamma$. Our numerical
results and theoretical discussions support the following conjecture.

\begin{conjecture}\label{mainresult}
Let $(\nu_0,\nu_1)$, $\nu_0,\nu_1<\bar{\nu} \ll 1$, be an interval such that for all $\nu \in (\nu_0,\nu_1)$
the dominant harmonic in $\DFt{i}$ is the one associated to 
the best approximant $m_1/m_2$ of $\gamma$.
Then, for $\nu \in (\nu_0,\nu_1)$,
\[ |\DFt{i}|  \approx \epsilon \exp\left( \frac{\Psi_i(L)}{\sqrt{c_{s,m_1/m_2} \nu}} \right), \qquad  L = c_{s,m_1/m_2} \nu m_1^2, \qquad i=1,2, \]
where  $\Psi_2(L)=\Psi_1(L)+\mathcal{O}(\sqrt{c_{s,m_1/m_2}\nu})$, $\Psi_1(L) = \Psi_M + O(|L-L_M|^2)$, being 
$\Psi_M = \Psi_1(L_M)= \text{\em max}(\Psi_1(L)) \approx -4.860298$
and
$L_M \approx 0.26236$. 
\end{conjecture}

\begin{remark}
Notice that $c_{s,m_1/m_2}$ depends on $\nu$ through the arithmetic properties
of $\gamma$, as was explained in Remark~\ref{remark_k1k2}. When $\gamma$ is a
quadratic irrational then the constants $c_{s,m_1/m_2}$ remain bounded as $\nu
\to 0$. On the other hand, when the quotients of the CFE of $\gamma$ are
unbounded then the maxima of the constants $c_{s,m_1/m_2}$ grow when $\nu
\to 0$. Actually the exponents in the exponentially small part of $|\DFt{i}|$, $i=1,2$, depend on $\nu$
through the behaviour of $\nu c_{s,m_1/m_2}$. See Section~\ref{Sec:Otherfreq} for some examples.  
\end{remark}

Explicit expressions of the functions $\Psi_i(L)$ are derived in the following
subsections. Note that $\Psi_1$ does not depend on the arithmetic properties of
$\gamma$ and $\Psi_2$ depends only on $m_1$. The dependence is through $L$ which depends on the
constant $c_{s,m_1/m_2}$ and the approximant $m_1/m_2$. This allows us to provide a
methodology to study the asymptotic behaviour for any frequency $\gamma$. 
Note that the dominant harmonic $(m_1,m_2)$ changes when $\nu \rightarrow 0$
and so does the constant $c_{s,m_1/m_2}$. This allows us to study the values of $\nu$ for
which a change of the dominant harmonic in $\DFt{i}$ is expected.

Conjecture~\ref{mainresult} asserts that the dominant term of the series
expansion of the Melnikov function gives the correct exponent of the splitting
behaviour, that is, that 
\[ \sqrt{\nu} \log\left(\frac{|\DFt{i}|}{\epsilon} \right) \approx \sqrt{\nu} \log C_{m_1,m_2}^{(i)} \approx \frac{1}{\sqrt{c_{s,m_1/m_2}}} \Psi_i(L), \qquad L = c_{s,m_1/m_2} \nu m_1^2,\]
where the second approximation comes from Theorem~\ref{mainresult1}.
See Section~\ref{Sec:non-dominant} for a more detailed discussion. Here we just want
to emphasize that this dominant term, which is related to the approximants of
$\gamma$, has a larger order of magnitude than the remaining terms of the
series. For example, for a constant type $\gamma$, if the dominant harmonic corresponds
to the linear combination $s=m_1-m_2\gamma$, then 
one expects $m_1,m_2 \sim 1/\sqrt{\nu}$ (see Remark~\ref{remark_k1k2}). This gives a term of order
$\exp(-c/\sqrt{\nu})$ much larger than the order $\exp(-c/\nu)$ expected for
the terms with other combinations. We stress that this is a purely quasi-periodic effect
related to the existence of two frequencies in the system. Indeed, if one
considers a one frequency forcing of the 1-dimensional separatrix dynamics, the effect of any of the terms of the Melnikov
series can be of the same (or similar) order than the dominant one. Hence all terms can contribute to change the
dominant exponent, see Appendix~\ref{exempleduffing}.

\subsection{The amplitude of the harmonics associated to approximants of $\gamma$} \label{7p2}

We consider first $M_1(\psi_0,\theta_0)=\DFt{1}/\epsilon$. Given $m_1, \ m_2
\in \mathbb{Z}$ we look for the expression of $C_{m_1,m_2}^{(1)}=|\hat{C}_{m_1,m_2}^{(1)}|$ in
(\ref{dgiexp}). In (\ref{exprdG1dG2_1}) we choose $l=-1$ to get the more
relevant terms, then one has $k=m_1 + 2 i -5$ and $j=m_2$ so that
\begin{align*}
C_{m_1,m_2}^{(1)} & = c_{m_2} \sum_{i \geq 0} 2^{(m_1+2i-2)/2} d_{m_1+2i-5} \, b_{m_1+2i-1,i} \, I_1(s,\nu,k+5)  \\
& \approx  c_{m_2} \sum_{i \geq 0} 2^{(m_1+2i-2)/2} d_{m_1+2i-5} \, b_{m_1+2i-1,i} \, \frac{2 \pi}{\nu} \frac{(s/\nu)^{m_1+2i-1}}{(m_1+2i-1)!} \, \Pi_{m_1+2i}(\nu/s) \exp\left(-\frac{\pi s}{2 \nu}\right),
\end{align*}
where $s=|m_1-m_2 \gamma|$ and $\Pi_{r}(\nu/s)=s^{1-r} P_{r-1}(s,\nu)$, $r\geq 1$. 
Note that to lighten the notation we have used $s$ for $|s|$ in this section,
hoping that no confusion will be produced.

From (\ref{recurP}) one has the recurrence
\begin{equation} \label{recurPi}
\Pi_1=\Pi_2=1, \qquad \Pi_r(w)=(1+(r-2)^2 w^2) \Pi_{r-2}(w), \ r \geq 2, 
\end{equation}
where $w = \nu/s$. In the expression for $C_{m_1,m_2}^{(1)}$ above (and the one for
$C_{m_1,m_2}^{(2)}$ at the end of this section) we have used the approximations
$\cosh^{-1}(\pi s/ (2 \nu)), \sinh^{-1}(\pi s/ (2 \nu)) \sim 2 \exp(-\pi s
/(2\nu))$, valid when $w^{-1}=s/\nu$ is large enough, the relative error being
$\mathcal{O}( \exp(-\pi s /\nu ))$.

From (\ref{expcs}), (\ref{expds}) and (\ref{ccs}) it follows that
$$
c_{m_2} \! =\! \frac{2}{ \rho_c^{m_2} \sqrt{c^2 -1}}, \ d_{m_1+2i-5} \! = \! \frac{m_1+2i}{d^{m_1+2i-4}}, \ b_{m_1+2i-1,i} \! = \! \frac{m_1}{2^{m_1+2i-1}(m_1+2i)} \binom{m_1+2i}{i},
$$
where $\rho_c = c+\sqrt{c^2-1}$. Then,
\begin{equation} \label{dG1}
C_{m_1,m_2}^{(1)} \approx= \frac{2}{  \rho_c^{m_2} \sqrt{c^2 -1}} \frac{2 \pi}{\nu} d^3 \exp\left(-\frac{\pi s}{2\nu}\right) m_1 2^{-m_1/2} \left(\frac{s}{\nu d}\right)^{m_1-1} \frac{d}{m_1!} \, S_A, 
\end{equation}
where $S_A$ denotes the sum
\begin{equation} \label{sumAi}
S_A= \sum_{i \geq 0} A_i, \qquad 
A_i=A_i(m_1,\nu,s)= \frac{m_1+2i}{2^i (m_1+i)! i!} \left(\frac{s}{\nu d}\right)^{2i} \frac{m_1!}{d} \, \Pi_{m_1+2i}( \nu/s ).
\end{equation}

Similarly, for $\DFt{2}$ one obtains, given $m_1,m_2$, that
\begin{equation} \label{dG2}
C_{m_1,m_2}^{(2)} \approx \frac{2}{\rho_c^{m_2} \sqrt{c^2 -1}} \frac{2 \pi s}{\nu^2} d^3 \exp\left(-\frac{\pi s}{2\nu}\right) 2^{-m_1/2} \left(\frac{s}{\nu d}\right)^{m_1-1} \frac{d}{m_1!} \, S_A.
\end{equation}

According to (\ref{dG1})-(\ref{dG2}) the contribution of the integral to $\DFt{i}$, $i=1,2$,
related to the approximant $m_1/m_2$ is $\mathcal{O}\left(\exp(-\frac{\pi s}{2\nu})\right)$, where the
contribution of finite negative powers of $\nu$ has been neglected. Hence those harmonics
associated to the smallest values of $s$ play the most important role. These
are expected to be related (asymptotically as $\nu \rightarrow 0$) with the
best approximants of $\gamma$.

From the expressions of $C_{m_1,m_2}^{(i)}$, $i=1,2$, we have the following result.
\begin{proposicio} \label{propquotient}
Given $m_1\geq 0$ and $m_2 \in \mathbb{Z}$ we have
\[ C_{m_1,m_2}^{(1)} = L \ C_{m_1,m_2}^{(2)} + \mathcal{O}(\exp(-\pi s / \nu)),\]
where $L=\nu m_1^2 c_{s,m_1/m_2}$ and $c_{s,m_1/m_2}>0$ is such that $s=|m_1-\gamma m_2|= \frac{1}{c_{s,m_1/m_2} m_1}$.
\end{proposicio}
\begin{proof}
From (\ref{dG1}) and (\ref{dG2}) one has $C_{m_1,m_2}^{(1)} / C_{m_1,m_2}^{(2)}
\approx  m_1 \nu/s = m_1^2 c_{s,m_1/m_2} \nu = L$ for all $m_1 \geq 0$ and $m_2 \in
\mathbb{Z}$.
\end{proof}

This relation explains the difference between $\DF{1}$ and $\DF{2}$ in
Fig.~\ref{ampli2}, see also Fig.~\ref{spf1f2} first row. Recall that it was
numerically observed that the dominant harmonic of both splittings coincide for
large ranges of $\log_2(\nu)$ when $\nu \to 0$ (see Table~\ref{taulanubif}).

\subsection{The universal function $\Psi_1(L)$ associated to an approximant $m_1/m_2$ of $\gamma$.} \label{7p3}

In this section we consider $m_1/m_2 \approx \gamma$ an approximant (not
necessarily a best approximant of $\gamma$). For concreteness we will focus on
$\DFt{1}$, the expression of $\Psi_2(L)$ will follow directly from
Proposition~\ref{propquotient}.

Given $m_1/m_2 \approx \gamma$, we express $C_{m_1,m_2}^{(1)}$ given by (\ref{dG1}) as
$$
C_{m_1,m_2}^{(1)} =\cP_f \cP_F S_A, \qquad \text{and } S_A=\cP_M \cP_Q,
$$
where
\[\cP_f = \frac{4 \pi d^4 m_1}{s \sqrt{c^2 -1}}, \qquad 
\cP_F = \frac{1}{\rho_c^{m_2} 2^{m_1/2} m_1!} \exp\left( - \frac{\pi s}{2 \nu} \right) \left( \frac{s}{\nu d}\right)^{m_1},\]
and $\cP_M$ denotes the dominant term of $S_A$ (i.e. the term of $S_A$ which
gives the maximum contribution to the sum) and $\cP_Q=S_A/\cP_M$.

To get intuition about how to proceed we perform some numerical investigations
considering (temporary!) $\gamma=(\sqrt{5}-1)/2$. Fig.~\ref{cdh} left shows the
behaviour of $\DFt{1}$. We see different changes of dominant harmonic as
$\nu \rightarrow 0$ that are marked with points. We represent the behaviour of
$S_A=\cP_M \cP_Q$ for $\gamma= (\sqrt{5}-1)/2$ in Fig.~\ref{cdh}
right. These two terms play the role of a factor which ranges in a finite
interval away from zero. In particular, this means that the change of harmonic should be
detected in the prefactor $\cP_f \cP_F$. The important term is $\cP_F$ since
$\cP_f$ does not depend on $\nu$ explicitly and, for the $m_1/m_2$ approximant
giving the maximum contribution to the splitting function for a fixed $\nu$,
behaves as a power of $\nu$, hence negligible in front the exponentially small
term in $\nu$ of $\cP_F$. Hence, below, we first look for the changes using
just $\cP_F$, later we will discuss the contribution of the sum $S_A$. The
factor $\cP_F$ depends on $m_1$, $m_2$ and $\nu$. We shall check later that
$\cP_Q$ gives no relevant contribution to $C_{m_1,m_2}^{(1)}$.

\begin{figure}[ht]
\begin{center}
\begin{tabular}{cc}
\hspace{-6mm}\epsfig{file=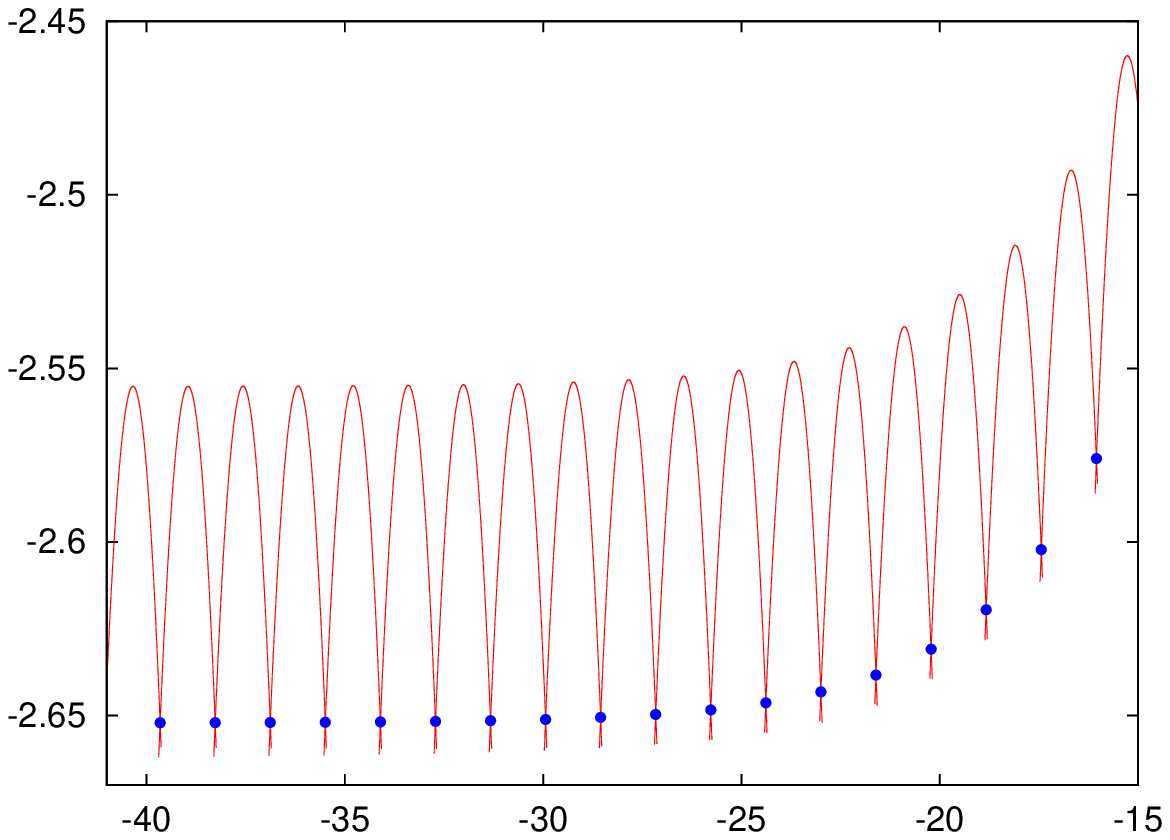,width=9cm} & 
\hspace{-9mm}\epsfig{file=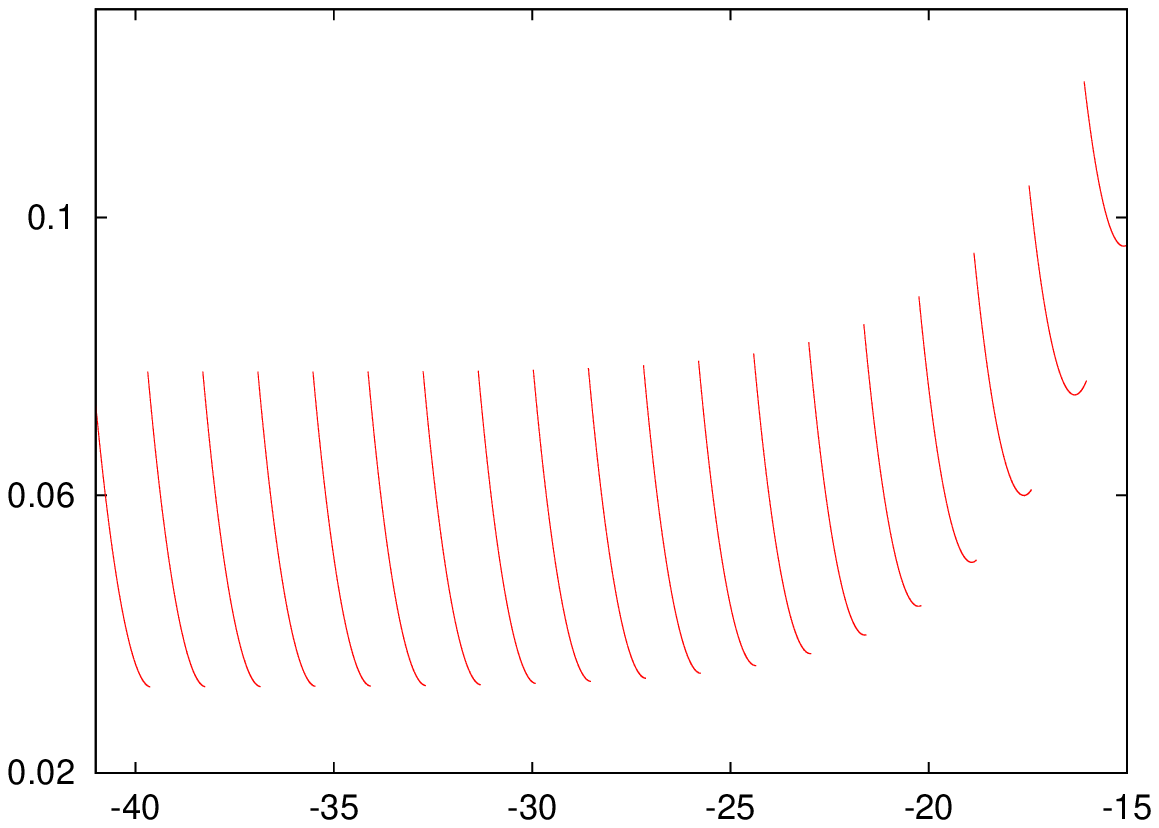,width=9cm} 
\end{tabular}
\end{center}
\vspace{-8mm}
\caption{We consider $\gamma=(\sqrt{5}-1)/2$ and $\epsilon=10^{-3}$. In both 
plots the horizontal variable is $\log_2(\nu)$. Left: $\sqrt{\nu} \log(
C_{m_1,m_2}^{(1)}/\epsilon)$, the points correspond to the changes of dominant
harmonic. The rightmost change corresponds to $m_1=55\to m_1=89$, while the
leftmost to $m_1=196418\to m_1=317811$. Right: $\sqrt{\nu} \log(\cP_M\cP_Q)$.}
\label{cdh}
\end{figure}

In what follows, given $m_1/m_2 \approx \gamma$, we study the contribution of
the different factors to $C_{m_1,m_2}^{(1)}$.

{\bf The contribution of $\cP_F$.}
We write $\cP_F=\cP_F(m_1)$ to explicitly note its dependence on $m_1$. Using
Stirling's formula we approximate $\log m_1! \approx m_1 (\log m_1-1)$ (i.e. we ignore the term $\sqrt{2\pi m_1}$),
one has
\begin{align} \label{cPF}
\log(\cP_F(m_1)) \! \approx \!  & - m_1 \left( \frac{\log(\rho_c)}{\gamma} + \frac{\log 2}{2} + (\log m_1 -1) + \frac{\pi}{2 c_{s,m_1/m_2} \nu m_1^2} +\log(dc_{s,m_1/m_2} \nu m_1) \right)  \nonumber \\
= & - m_1(K+\log(L) +B/L),
\end{align}
where
\[ K=\log(d)+\log(2)/2+\log(\rho_c)/\gamma-1,\quad L=c_{s,m_1/m_2} \nu m_1^2,  \quad B=\pi/2.\] 

{\bf The contribution of $\cP_M$.} To take into account the effect of the factor  $\cP_M$ we need to identify the
dominant term of $S_A$.  From (\ref{sumAi}) and (\ref{recurPi}) it
follows that the quotient of two consecutive terms in the sum $S_A$
is \begin{equation} \label{quotientA} \frac{A_i}{A_{i-1}}=
\frac{m_1+2i}{2(m_1+2i-2)(m_1+i)i} \left( \frac{s}{\nu d} \right)^2
\left(1+(m_1+2i-2)^2 (\nu/s)^2 \right).  \end{equation}

We look for the index $i$ corresponding to the term with maximum value of the
sum $S_A$ for a fixed value of $\nu$. It is useful to introduce $I=i/m_1$ and
look for the index $I$ instead.  From
(\ref{quotientA}), one gets
\begin{equation} \label{quotientAA}
 \frac{A_i}{A_{i-1}}=\frac{1}{2d^2} \left( \frac{1}{I+I^2} \left(
\frac{s}{m_1 \nu} \right)^2 + \frac{1}{I+I^2}  + 4 \right) \left(1 +
\mathcal{O}(m_1^{-1}) \right). 
\end{equation}

From this quotient one deduces that the sequence $\{A_i\}_i$ is increasing for
small values of $i$ and it becomes decreasing for large values of
$i$ provided $d>\sqrt{2}$ (recall that we choose $d=7$ in the concrete example).  The
maximum value is achieved when $A_i \approx A_{i-1}$. 
Then, ignoring the terms of relative value $\mathcal{O}(m_1^{-1})$ one gets the
following equation 
\begin{equation} \label{eqI}
(2d^2-4) (I^2+I)=1+\left( \frac{s}{\nu m_1} \right)^2,
\end{equation}
from which one can determine the index $i=m_1 I$ of the maximum term of
$S_A$. Hence, taking into
account the expression (\ref{sumAi}), the factor $\cP_M$ is
\begin{equation} \label{cPMdetall}
\cP_M = \frac{m_1! k}{d (\sqrt{2} w d)^{2 m_1 I}  (m_1(1+I))! (m_1 I)!} \, \Pi_k(w), 
\end{equation}
where $k=m_1(1+2I)$ and $w=\nu/s$. From the recurrence relation (\ref{recurPi}) one gets
\[ \log(\Pi_k(w))=\!\!\!\!\sum_{j=k-2(-2)0}\!\!\!\!\log(1+j^2w^2),\]
where the index $j$ runs with step $-2$ (and finishes at $j=1$ whenever $k$ is odd).
Approximating the previous sum by an integral one has
\begin{align}
\log(\Pi_k(w))  \approx & \frac{1}{2} \int_0^{k-1}\hspace{-6mm}\log(1+j^2w^2) \, dj= \frac{1}{4w}\int_0^{(k-1)^2w^2}\hspace{-12mm} \log(1+z)\frac{dz}{\sqrt{z}} \nonumber\\
 = & \frac{1}{2w}\left[\sqrt{z}(\log(1\!+\!z)\!-\!2)\!+\!2\arctan(\sqrt{z}) \right]_0^{(k-1)^2w^2}\nonumber \\
 = & \frac{k-1}{2}\left(\log(1\!+\!(k\!-\!1)^2w^2)\!-\!2\right)+\frac{\arctan((k\!-\!1)w)}{w}.
\end{align}

{\bf The definition of the universal function $\Psi_1(L)$.} 
We define now the universal function $\Psi_1(L)$ from the previous
contributions of $\cP_F$ and $\cP_M$. We will check below that the contribution
of $\cP_Q$ is not important in the sense that $\DFt{1} \approx \cP_F \cP_M$
is accurate enough to detect the changes of dominant harmonics.


First, we recall from (\ref{cPF}) that $\log(\cP_F)/m_1 \approx -(K+\log(L)+B/L)$ with $K$ and $B$ independent of $L$.
Let us denote by $\Psi_{1,1}(L)=-(K+\log(L)+B/L)$, and note that it slightly depends on $\gamma$ through $K$.
Next, we obtain an approximation of $\log(\cP_M)/m_1$ that only depends on $L=c_{s,m_1/m_2} \nu m_1^2$. Equation (\ref{eqI}) can be rewritten as
$(2d^2-4)(I^2+I)-1=1/L^2$, so that given $L$ we can obtain the index $I=I_*$ that
determines $\cP_M$. From (\ref{cPMdetall}), after skipping some constant terms
and higher order terms in $m_1^{-1}$, one gets
\begin{align} \label{cPM}
\log(\cP_M)/m_1 \approx &  -2I_*\log(L d)-(1+I_*)\log(1+I_*)-I_*\log(I_*)+I_*(2-\log(2)) \\
 & +  (1+2I_*)(\log(1+((1+2I_*)L)^2)-2)/2+\arctan((1+2I_*)L)/L, \nonumber
\end{align}
where the terms of the first line come from the prefactor of $\Pi_{k}(w)$ in
(\ref{cPMdetall}) and the terms of the second one are related to $\Pi_{k}(w)$
after taking logarithms. Let us denote by $\Psi_{1,2}(L)$ the right hand side of (\ref{cPM}).

Now we define 
\begin{equation}\label{funPsi}
\Psi_1(L) := \Psi_{1,1}(L) + \Psi_{1,2}(L), 
\end{equation}
which depends on the parameters $c$ and $d$ (and slightly on $\gamma$ through $K$) but does
not depend explicitly on the approximant $m_1/m_2$ of $\gamma$.
The universal function $\Psi_1(L)$ provides an approximation of $\sqrt{c_{s,m_1/m_2} \nu}
\log(|\DFt{1}|/\epsilon)$ as a function of the parameter $L= c_{s,m_1/m_2} \nu m_1^2$. In
Fig.~\ref{psi_ncs} we show the function $\Psi_1(L)$ as a function of $L$.  We
can see that it has the properties described in Conjecture~\ref{mainresult}.

\begin{figure}
\begin{center}
\epsfig{file=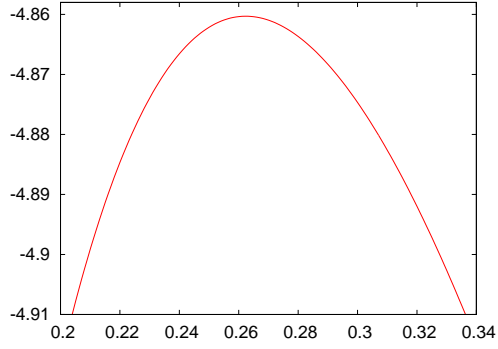,width=7cm}
\caption{The universal function $\Psi_1(L)$ as a function of $L$.}
\label{psi_ncs}
\end{center}
\end{figure}

{\bf The factor $\cP_Q$ plays no role.}
Here we check that $\cP_Q$ becomes not relevant as $\nu \rightarrow 0$ or,
equivalently, as $m_1 \rightarrow \infty$.  Ignoring the terms
$\mathcal{O}(m_1^{-1})$ in (\ref{quotientAA}) we obtain
\[\frac{A_i}{A_{i-1}}=\frac{1}{2d^2(I+I^2)}((1+2I)^2+L^{-2}). \] This
quotient depends on $L$ and $I$. For a fixed value of $L$ the quotient $A_i/A_{i-1}$ is
a monotonically decreasing function of $I$ (and hence of $i$) independently of
the value of $L$. Concretely one has
\[
\frac{\partial(A_i/A_{i-1})}{\partial I}=-\frac{(1+2I)(1+L^{-2})}
{2d^2(I+I^2)^2}.
\]

Recall that $I_*$ is the value of $I$ giving the quotient $A_i/A_{i-1}$
closest to one (i.e. $I_*$ corresponds to the maximum term of the sum
$S_{A}$ and, by definition, it determines the factor $\cP_M$). For $\delta>0$ fixed, let $I_{\pm}$
the values of $I$ for which one has $A_i/A_{i-1} = 1 \pm \delta$.
One has 
\[I_{\pm} = I_* \pm \partial{I_\pm}/\partial{\delta}|_{\delta=0} \, \delta + \mathcal{O}(\delta^2),\]
and one checks that
\[ \partial{I_{\pm}}/\partial{\delta}|_{\delta=0} = \frac{2 d^2 (I_* + I_*^2)}{(2d^2-4)(1+2I_*)} = \mathcal{O}(1), \]
meaning that $|I_+ - I_-|= \mathcal{O}(\delta)$. Since $I_\pm= m_1 i_\pm$, it follows
that $|i_+ -i_-|= \mathcal{O}(m_1 \delta)$.

We split the sum $S_{A} = \sum_{i \geq 0} A_i$ into three (say left/center/right) parts
\[S_{A}= \sum_{i=0}^{i_-} A_i + \sum_{i=i_-}^{i_+} A_i + \sum_{i=i_+}^{\infty} A_i =S_l+S_c+S_r.\]

We recall that $S_{A}= \cP_M \cP_Q$, where $\cP_M = A_{i_*}$, where $i_*=m_1 I_*$, hence
\[\cP_Q =  \frac{1}{A_{i_*}} \left( S_l + S_c + S_r \right) .\]
From  $|i_+ -i_-|= \mathcal{O}(m_1 \delta)$, it follows that $S_c=
\mathcal{O}(m_1 \delta) A_{i_*}$.  On the other hand, the terms in $S_l$ decay
as $A_i \leq (1+\delta) A_{i-1}$. Hence, $S_l \leq A_{i_*} \sum_{i=0}^{i_-} (1 + \delta)^{-i}  = \mathcal{O}( A_{i_*}/ \delta) $.
Similarly, for $S_r$ one has $A_i \leq (1- \delta) A_{i-1}$, hence $S_r = \mathcal{O}(A_{i_*}/\delta)$.
As a conclusion, one gets 
\[ \cP_Q = \mathcal{O}(m_1 \delta) + \mathcal{O}(\delta^{-1}). \]
Taking, for example, $\delta = m_1^{-1/2}$ one gets $\cP_Q = \mathcal{O}(m_1^{1/2})$, meaning
that the factor $\cP_Q$ can be ignored compared with the exponentially small terms since
its logarithm divided by $m_1$ is small compared with the other terms in $\Psi_1$.

{\bf The analogous function $\Psi_2(L)$.} For a fixed $m_1/m_2 \in \mathbb{Q}$ 
we define the function $\Psi_2(L)$ as 
\begin{equation} \label{G1G2rel}
 \Psi_2(L) = \Psi_1(L)- \frac{\sqrt{L}}{m_1} \log(L). 
\end{equation}
From Proposition~\ref{propquotient} one has that 
\[\Psi_2(L) \approx \sqrt{c_{s,m_1/m_2} \nu} \log( C_{m_1,m_2}^{(2)} ).\]

Assume that we are interested in the functions $\Psi_1(L)$ and $\Psi_2(L)$ for
values of $L \in [L_-,L_+]$ around their maxima. Then, the relation
(\ref{G1G2rel}) shows that $\Psi_2(L)$ tends to $\Psi_1(L)$ as
$\nu \rightarrow 0$, uniformly in $[L_-,L_+]$.

\subsection{The changes in the dominant harmonic of the splitting function} \label{sectchanges}

Several properties can be analysed from the derived universal functions
$\Psi_1$ and $\Psi_2$. 

First we look for the changes of the dominant harmonic in $\DFt{1}$ as $\nu$
varies. We expect that for most of the values of $\nu$ there is one dominant harmonic.
However, for some values of $\nu$ different harmonics can be of the same
order of magnitude. Our aim is to determine, for a given $\nu$ small enough,
which is (are) the dominant harmonic(s).  

Some general comments are in order.  As already said and according to
(\ref{dG1}) (resp. (\ref{dG2})), for $\nu$ small enough one expects the
dominant harmonic(s) of $\DFt{1}$ (resp. $\DFt{2}$) to be related with
the best approximants of $\gamma$.  That is, to get the dominant harmonic it is
enough to compare the harmonics associated to best approximants $m_1/m_2$ of
$\gamma$. Below we will restrict to best approximants and we will compare the
functions $\Psi_1$ associated to them. However, not all the harmonics
associated to best approximants become a dominant harmonic. Several examples
will be given in Section~\ref{Sec:Otherfreq}.
Finally, we note that, assuming that the amplitudes of the harmonics of the
Poincar\'e-Melnikov integral decay in an exponential way as in Remark~\ref{remark_k1k2},
at least one of every two consecutive best approximants of $\gamma$ becomes the
dominant harmonic of $\DFt{i}$, $i=1,2$, for a suitable range of $\nu$. In
Appendix~\ref{conseq_best} we consider that problem assuming two small consecutive quotients
between two large quotients of the CFE of $\gamma$. For a more general discussion see \cite{FonSimVie-2}.  

To determine which of the best approximants is associated to the dominant
harmonic requires to know the constants $c_{s,m_1/m_2}$ to be able to compare the
corresponding functions $\Psi_1$. If moreover one wants to look for the
asymptotic behaviour of the changes of dominant harmonic as $\nu \rightarrow 0$
one needs an asymptotic description of the values of $c_{s,m_1/m_2}$.
Next subsections deal with this question.

\subsubsection{The golden mean frequency.}  \label{goldenfreq}

For simplicity, first we consider $\gamma$ to be a quadratic irrational so that
its CFE is periodic. We shall prove in Lemma~\ref{lema_periodic_csn} that, in
this case, the values of the constants $c_{s,m_1/m_2}$ associated to the best
approximants of $\gamma$  are (asymptotically, as the order of the best
approximant tends to infinity) also periodic.  Moreover, for concreteness, we
focus on $\gamma = (\sqrt{5}-1)/2$ but other quadratic irrational numbers can
be similarly handled. 

As we shall discuss in Section~\ref{Sec:periodicity_csn}, for
$\gamma=(\sqrt{5}-1)/2$, one has $c_{s,n} \rightarrow \sqrt{5}(1+\gamma) =3+\gamma$
when considering best approximants of $\gamma$ and as the order of the best
approximant tends to infinity. The best approximants are quotients of
consecutive Fibonacci numbers. It turns out that all best approximants are
visible as a dominant harmonic in a corresponding interval of $\nu$. We look for
the sequence of values $\nu_j$ of $\nu$ for which the changes of dominant harmonic
take place, see Fig.~\ref{cdh} left.  Assume that the $j$-th best approximant
of $\gamma$ dominates at a specific value of $\nu=\nu_1^*$.
We first use the approximation $\DFt{1} \approx \epsilon \cP_F(m_1)$ where $m_1$ is the
numerator of the $j$-th best approximant. Assume that for $\nu=\nu_0^* < \nu_1^*$ the dominant
harmonic corresponds to the $(j+1)$-th best approximant of $\gamma$. Then there is a value $\nu=\nu_j$, corresponding to the change $m_1\to(1+\gamma)m_1$ of dominant harmonic, for which  
$\log(\cP_F(m_1)) = \log(\cP_F((1+\gamma)m_1))$.
This condition leads to the following equation for $L$
\[ L=\frac{\pi\gamma/(2(1+\gamma))}{2(1+\gamma)\log(1+\gamma)+K\gamma+\gamma
\log(L)}.\]
This equation, which is independent of $m_1$, can be solved by numerical
iteration and one obtains $L=L_l\approx 0.1690224$ for the values $c=5, d=7$ in our perturbation. This implies that asymptotically 
$\nu_{j+1} \approx \gamma^2 \nu_j$. Indeed, from $m_2 \approx m_1 (1+\gamma)$
it follows that $L= \nu_j m_1^2 c_{s,m_1/m_2} \approx \nu_{j+1} m_1^2 (1+\gamma)^2 c_{s,m_1/m_2}$
and then $\nu_{j+1} \approx \gamma^2 \nu_j$.
Accordingly, this agrees with Fig.~\ref{cdh} left where the values $\log_2(\nu_j)$ tend
to be, as $\nu \rightarrow 0$, separated by $2 \log_2(\gamma) \approx -1.38848$.

More concretely, let $F_j$ denote the Fibonacci sequence starting with $F_1=1$,
$F_2=2$, $F_3=3$, \dots.  We can compute the values $\nu=\nu_j$ where $\nu_j$
corresponds to the change  $m_1=F_j \to m_1=F_{j+1}$. With this notation the
blue points in Fig.~\ref{cdh} left correspond to the values of $\log_2 (\nu_j)$
for $9 \leq j \leq 26$.  Moreover, one has $\nu_j \sim \gamma^{2j} \hat{K}$,
for some $\hat{K}$.  In Fig.~\ref{limitK} we represent $\nu_j \gamma^{-2j}$ as
a function of $j$. We see that, for $j$ large enough, it tends to the constant
$\hat{K} \approx 0.0850$.

\begin{figure}[ht]
\begin{center}
\epsfig{file=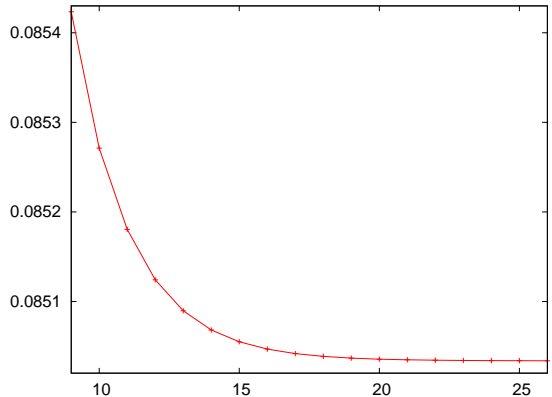,width=8cm} 
\end{center}
\caption{We represent  $\nu_j \gamma^{-2j}$, where $\nu_j$ are the values where
a change of dominant harmonic has been numerically detected, as a function of the index $j$
of the Fibonacci sequence $F_j$ (see text for details).  
}
\label{limitK}
\end{figure}

Let us describe a more general methodology to look for the changes of dominant
harmonic which takes into account the corrections due to the factor $\cP_M$.
Since for $\gamma=(\sqrt{5}-1)/2$ one has $c_{s,m_1/m_2}=c_{s,n} \to 3+\gamma
\approx 3.618034$ we introduce $\tilde{L}=L/c_{s,m_1/m_2}$ and we consider
$\hat{\Psi}_1(\tilde L):=\Psi_1(\tilde{L})/\sqrt{c_{s,m_1/m_2}}$. In Fig.~\ref{5ones} we represent the leftmost five
peaks of Fig.~\ref{cdh}  left as a function of the parameter $\tilde{L}$. They
correspond to $m_1=46368,75025,$ $121393,196418,317811$.  Also, in blue, we
represent the function $\hat{\Psi}_1(\tilde L)$.
We see in the right plot that, as $\nu$ decreases to 0, the curves tend to
$\hat{\Psi}_1(\tilde{L})$.

\begin{figure}[ht]
\begin{center}
\begin{tabular}{cc}
\hspace{-2mm}\epsfig{file=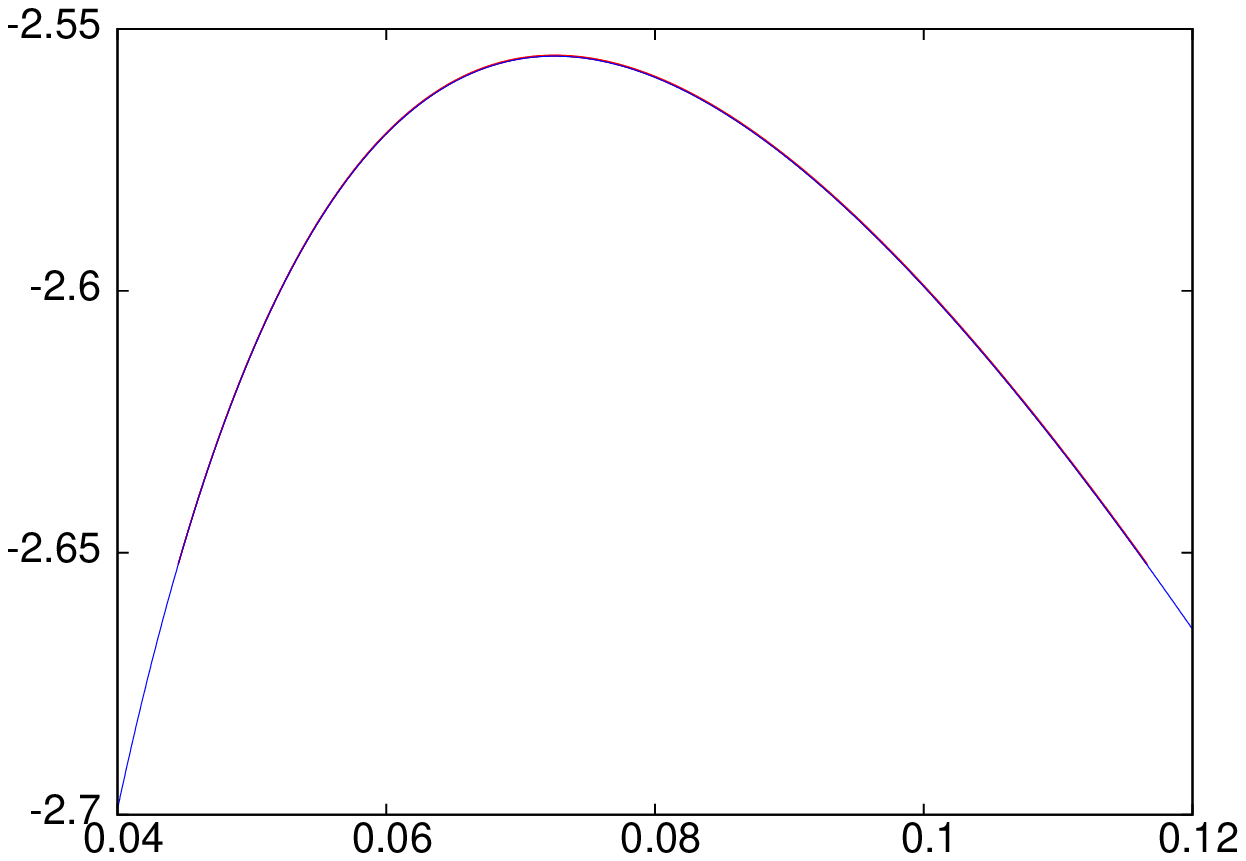,width=7cm} &
\hspace{-4mm}\epsfig{file=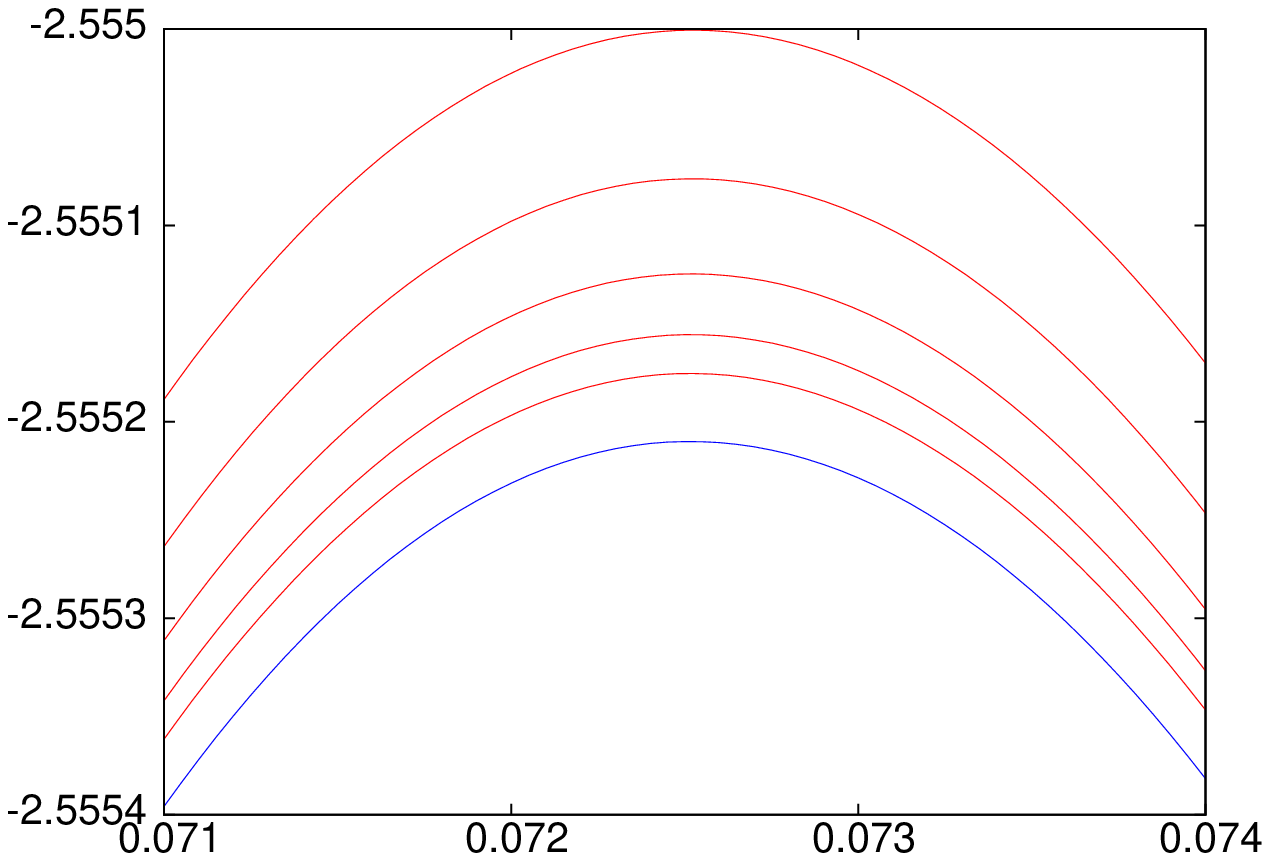,width=7cm}
\end{tabular}
\end{center}
\caption{Left: The five leftmost peaks of Fig.~\ref{cdh} as a function of
$\tilde{L}$ (in red). The function $\hat{\Psi}_1(\tilde{L})$ is also shown (in
blue). All of them almost coincide at this scale.  Right: Magnification of the
central zone of the left plot. We see that the peaks move down as $\nu$
decreases (and $m_1$ increases). They tend to $\hat{\Psi}_1(\tilde{L})$.}
\label{5ones}
\end{figure}

In Fig.~\ref{Psi} we represent the function $\hat{\Psi}_1(\tilde{L})$ as a function
of $\log(\tilde{L})$.  The maximum of $\hat{\Psi}_1(\tilde{L})$ is $\approx
-2.555210$, in good agreement with the numerical values shown in
Fig.~\ref{5ones} and in Fig.~\ref{cdh} left. It is achieved for
$\tilde{L} \approx 0.072529$. After a change of coordinates  the function
$\hat{\Psi}_1(\tilde{L})$ behaves as $-\log(\cosh(\tilde{L}))$, see \cite{DelGelJorSea97,DelGut05}.

\begin{figure}[ht]
\begin{center}
\hspace{-2mm}\epsfig{file=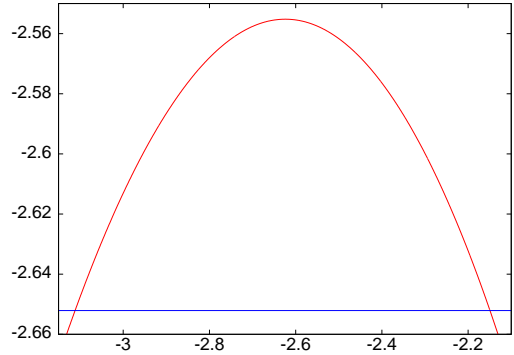,width=7cm}
\end{center}
\caption{We depict $\hat{\Psi}_1(\tilde{L})$ as a function of $\log(\tilde{L})$ for $\gamma=(\sqrt{5}-1)/2$ (see text for details).} 
\label{Psi}
\end{figure}

Let us consider two values of $\tilde{L}$, say $\tilde{L}_1$ and $\tilde{L}_2$,
$\tilde{L}_1<\tilde{L}_2$, corresponding to different harmonics.  Assume that
for $\nu>0$ small enough these harmonics are related to two consecutive best
approximants of $\gamma$, say $m_1/m_2$ and $m_2/m_3$ (the numerators
$m_1$ and $m_2$ are two consecutive Fibonacci numbers). Assume that the change
of harmonic takes place at $\nu=\nu_0^*$ then $\tilde{L}_1= m_1^2\nu_0^*$,
$\tilde{L}_2=m_2^2 \nu_0^*$ and $\hat{\Psi}_1(\tilde{L}_1)=\hat{\Psi}_1(\tilde{L}_2)$.
Moreover, if $m_1$ is large, then one has $m_2 \approx (1+\gamma) m_1$ and therefore
$\tilde{L}_2 \approx (1+\gamma)^2 \tilde{L}_1$. One obtains
$\tilde{L}_l\!\approx\!  0.044524$ as the asymptotic value of $\tilde{L}$ where
the change takes place. Notice that $L=L_l \approx 0.16109$, which is very
close to the value of $\tilde{L}_l$ obtained above using just $\cP_F$.  One has
$\hat{\Psi}_1(\tilde{L}_l) \approx -2.652115$, which is represented as an horizontal
line in Fig.~\ref{Psi}. We conclude that at the value $\nu=\nu_j \approx
\tilde{L}_l/ F_j^2$ takes place the change $m_1=F_j \to m_1= F_{j+1}$ of 
dominant harmonic of $\DFt{1}$. 

In Table~\ref{taulanubif} we can see that, for large range intervals of $\nu$,
both $\DFt{1}$ and $\DFt{2}$ have the same dominant harmonic. Indeed,
relation (\ref{G1G2rel}) implies, in particular, that the changes of dominant
harmonic in $\DFt{1}$ and in $\DFt{2}$ tend to coincide as $\nu
\rightarrow 0$. Concretely, denote by $\nu_j^{(i)}$ the sequence of values of
$\nu$ for which the dominant harmonic of $\DFt{i}$ changes,
the values of $\nu_j^{(1)}$ have been determined in Section~\ref{goldenfreq}. 
To look for the values $\nu_j^{(2)}$ we consider the
condition $\Psi_2(L_1) = \Psi_2(L_2)$ with $L_2=L_1(1+\gamma)^2$ which, by
(\ref{G1G2rel}), is equivalent to 
\[\Psi_1(L_1)  = \Psi_1(L_1(1+\gamma)^2) - \frac{\sqrt{L_1}}{m_1} \left( \gamma \log(L_1) + 2 (1+\gamma) \log(1+\gamma) \right).\] 
Note that, since $L = c_{s,m_1/m_2} \nu m_1^2$, when $\nu \rightarrow 0$ we recover the condition
that determines the values $\nu_j^{(1)}$. 
One has $ \nu_{j}^{(2)} = \tilde{L}_l^{(2)}/F_j^2, $ where
$\tilde{L}_l^{(2)}= \tilde{L}_l + \mathcal{O}(\sqrt{\nu})$, being $\tilde{L}_l
\approx 0.044525$. The values of $\nu_j^{(1)}$ and $\nu_j^{(2)}$, corresponding to the changes of dominant
harmonic in $\DFt{1}$ and $\DFt{2}$, respectively, are displayed in
Table~\ref{nu1_nu2}.  We have considered the range $\log_2(\nu) \in [-24,-16]$.
We refer to Fig.~\ref{nterms} left where the computation of the amplitude of
the splitting for this range of values of $\nu$ is shown. 
The best approximant $N_j/D_j = F_j/F_{j+1}$ corresponds to a dominant harmonic
for $\nu_j = \mathcal{O}(1/F_j^2)$.  Hence $\nu_j^{(1)} - \nu_j^{(2)}=
\tilde{L}_l/F_j^2 - (\tilde{L}_l + \mathcal{O}(\sqrt{\nu_j}))/F_j^2 =
\mathcal{O}(\nu_j \sqrt{\nu_j})$, as it is observed in the last
column of Table~\ref{nu1_nu2}.

\begin{table}
\begin{center}
\begin{tabular}{|c|c|c|c|c|c|}
\hline
 $N_j$ &  $N_{j+1}$  & $\log_2(\nu_j^{(1)})$ & $\log_2(\nu_j^{(2)})$&   $\nu_j^{(1)}-\nu_j^{(2)}$ & Coeff \\
\hline
  55  &   89   & -16.04563135  & -16.05223394 &  0.675040E-07 &  1.191635 \\
  89  &  144   & -17.43664042  & -17.44071697 &  0.159057E-07 &  1.190968 \\
 144  &  233   & -18.82665512  & -18.82917332 &  0.375102E-08 &  1.190692 \\
 233  &  377   & -20.21609319  & -20.21764898 &  0.884894E-09 &  1.190469 \\
 377  &  610   & -21.60516252  & -21.60612386 &  0.208812E-09 &  1.190355 \\
 610  &  987   & -22.99400932  & -22.99460338 &  0.492817E-10 &  1.190280 \\ 
\hline
\end{tabular}
\caption{
Values of $\nu_j^{(1)}$ and $\nu_j^{(2)}$ for which the change from the dominant harmonic related to the approximant $N_j/D_j$ to $N_{j+1}/D_{j+1}$ takes place.
The last column displays the value of the coefficient $\text{Coeff} \approx (\nu_j^{(1)}-\nu_j^{(2)})/\nu_m^{3/2}$, $\nu_m=(\nu_j^{(1)}+\nu_j^{(2)})/2$.
}
\label{nu1_nu2}
\end{center}
\end{table}

We remark that the previous comments assert that $\nu_j^{(1)}$ and
$\nu_j^{(2)}$, corresponding to changes of dominant harmonic in $\DFt{1}$
and $\DFt{2}$, tend to coincide as $\nu \rightarrow 0$. For values of $\nu
\in I_j=[\nu_j^{(2)},\nu_j^{(1)}]$ the dominant harmonic of each splitting
function is different. This has some dynamical consequences: 
according to Appendix~\ref{split-difusion} 
one expects to have a faster diffusion process in phase space (but taking place in
exponentially large times!) for values of $\nu \in I_j$ rather than
for values of $\nu$ outside the union of the intervals $I_j$. Numerical massive
investigations of the diffusion phenomena taking place for the example
considered in this work and for small enough values of $\nu$ so that the limit
behaviour can be observed would require a huge (nowadays prohibitive!) amount
of computing time. Nevertheless, we believe that some numerical explorations of
this model for moderate values of $\nu$ are of much interest. We postpone
them for future works.

\subsubsection{A general frequency $\gamma$}

The same strategy can be used to look for values $\nu_j$ for which there is a
change of dominant harmonic of $\DFt{1}$ (and of $\DFt{2}$) for general
$\gamma$. Consider approximants $m_1/m_2$ and $n_1/n_2$ of $\gamma$ such
that the related harmonics become dominant for $\DFt{1}$ (similar for $\DFt{2}$) in 
adjacent intervals of $\nu$.
The change of dominant harmonic for $\DFt{1}$ takes place for $\nu$ such that
 $\Psi_1(L_1)=\Psi_1(L_2)$, where $L_1=c_{s,m_1/m_2} \nu m_1^2$
and $L_2= c_{s,n_1/n_2} \nu m_1^2 $. 
Using that $L_2=L_1 n_1^2 c_{s,n_1/n_2} / (m_1^2 c_{s,m_1/m_2})$
the previous equation can be solved for $L_1$ (e.g. numerically by
simple iteration) to obtain the values of $\nu = \nu_j$ corresponding to the
changes.

As an illustrative example, we show in Fig.~\ref{em2} the results for the transcendental
frequency number $\gamma=e-2$. From its CFE properties it follows that the constants
$c_{s,m_1/m_2}$ become unbounded, see details in Section~\ref{sect_unboundedCFE}. On the other
hand, we see in the figure 
that all the harmonics related to best approximants become dominant in a
suitable range of $\nu$. We remark that for other $\gamma$ it might happen that
some best approximants will not be related to a dominant harmonic of $\DFt{1}$
(see examples in Section~\ref{Sec:Otherfreq}). 
Concretely, for $\gamma=e-2$, we show in Fig.~\ref{em2} the functions $\Psi_1$ in blue
lines and the points that correspond to the values of $\nu$ where a change of
the dominant harmonic takes place.
These values are obtained by comparing the functions $\Psi_1$ for
different approximants as explained above in this section. 
As an extra check, we have compared the
values of $\nu$ obtained by the previous procedure with the corresponding
values obtained if one computes the contribution of each harmonic $m_1/m_2$
using the complete expression (\ref{dG1}) for $C_{m_1,m_2}$. These
contributions are shown in red lines in the figure. We see that the blue lines
are good enough approximations of the red ones for $\nu$ small enough. Moreover the values of $\nu$
are almost coincident even for the rightmost part of the figure where the
agreement between the blue and red curves is not so good.

\begin{figure}[ht]
\begin{center}
\hspace{-2mm}\epsfig{file=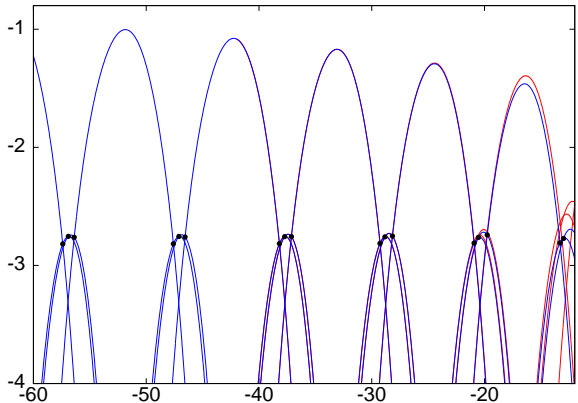,width=8cm}
\end{center}
\caption{For $\gamma=e-2$ we show the contribution of the different harmonics,
the range of $\nu$ where they become dominant and the changes. The horizontal
axis corresponds to $\log_2(\nu)$. In red we depict $\sqrt{\nu}
\log(C_{m_1,m_2})$. In blue the functions $\Psi_1(L)/\sqrt{c_{s,m_1/m_2}}$ obtained for
the corresponding best approximants. The points correspond to the values of
$\nu$ where there is a change of dominant harmonic. They are computed using the
functions $\Psi_1(L)$. Note that for $\log_2(\nu) < -25$ both red and blue
curves become almost coincident.}
\label{em2}
\end{figure}

\subsection{The effect of non-dominant terms of the splitting function} \label{Sec:non-dominant}

To find the dominant terms of the splitting $\DFt{1}$ we have considered
values of $\nu$ small enough (fixed) and have looked for the values of $m_1$ for which 
$L= c_{s,m_1/m_2} \nu m_1^2$ is the closest to the maximum of $\Psi_1(L)$. This term (or these
terms if, for example, we are close to a change of dominant harmonic) gives the
maximum contribution to the Melnikov function $\DFt{1}$
in (\ref{exprdG1dG2_1}). However, to assert that the splitting Melnikov function
$\DFt{1}$ is of the order of this/these dominant terms there are some
details to be checked.  As said in the Introduction, a theoretical proof must 
consider the effect of all the harmonics of the splitting function, bound the
effect of the ones related to approximants which are not best approximants, and
bound the effect of the best approximants which are non-dominant (for the values of
$\nu$ considered). In particular, one has to address the following questions. 
\begin{enumerate}
\item For a fixed $\nu$ we look for $m_1$ giving the most important terms in
$\DFt{1}$. Which is the effect of the other terms associated to best
approximants for this value of $\nu$?
\item Of course there are other approximants of $\gamma$ which are not best
approximants. We call them ``subapproximants''.  Which is their contribution to
$\DFt{1}$? Which are the corresponding constants $c_{s,m_1/m_2}$ related to each family
of subapproximants and which is their contribution to $\DFt{1}$?
\item Looking at the expression (\ref{exprdG1dG2_1}) of $\DFt{1}$ we see
that values of $k,i,j$ for which $s$ is large correspond to terms which make a
small contribution to the total sum. But there are infinitely many of these
terms. How to bound their total contribution?  
\end{enumerate}

Even if we are not going to address these questions formally, we want to provide
an idea of how useful can be the universal function $\Psi_1$ to investigate
such questions. For concreteness we focus on
$\gamma=(\sqrt{5}-1)/2$. We recall that in this case one has $c_{s,n}
\rightarrow 3+\gamma$ as $n \rightarrow  \infty$ (see
Section~\ref{Sec:periodicity_csn}). We proceed as follows.

\begin{enumerate}
\item To evaluate the function $\Psi_1(L)$ we consider the algorithm introduced
in Section~\ref{sectchanges}. Recall that $\Psi_1$ depends on
$\gamma,c_{s,m_1/m_2},c$ and $d$ but not on $\nu$. 

\item We compute the maximum of $\Psi_1(L)$. We denote by $\tilde{L}_M$ the
value of $\tilde{L}=L/c_{s,m_1/m_2}$ for which the maximum is attained.

\item We take $\nu$ small enough and we look for the integer $m_1$, among the
numerators of the best approximants, closest to $\sqrt{\tilde{L}_M/\nu}$.
Maybe there are two integer values at a similar distance and a bifurcation
takes place because the dominant harmonic of $\DFt{1}$ changes. For
$\gamma=(\sqrt{5}-1)/2$ this happens whenever $\Psi_1(L)=
\Psi_1(L(1+\gamma)^2)$.

\item If for the chosen value of $\nu$  there is an integer $m_1$ for which
$\tilde{L}= \nu m_1^2=\tilde{L}_M$, then we have to check that the value of $\Psi_1$ at
$\tilde{L}=\tilde{L}_{Mk}:=\tilde{L}_M(1+\gamma)^k$ for $k=\pm 2,\pm 4, \ldots$ is small
enough. 

\item If $\nu=\nu_B$ corresponds to a bifurcation then
$\Psi_1(\tilde{L}_B)=\Psi_1(\tilde{L}_B (1+\gamma)^2)$ for 
 $\tilde{L}_B= \nu m_1^2$.

\item It might be also interesting to look for values of $\nu$ for which there
is a dominant harmonic but there is a change of subdominant. This happens for
$\tilde{L}=\tilde{L}_C$ such that $\Psi_1(\tilde{L}_C(1+\gamma)^2) =
\Psi_1(\tilde{L}_C/(1+\gamma)^2)$.
\end{enumerate}

We note that we have performed all the computations for $c=5$ and $d=7$. What
happens in the limit cases, that is, either  for $d \rightarrow \sqrt{2}$ as a
function of $c$ or for $c \rightarrow 1$ as a function of $d$?  Note that when
$c \rightarrow 1$ the function $f(\theta)$, see (\ref{perturbedsystem}), tends to be
unbounded as well as its Fourier coefficients (\ref{expcs}).
The same thing happens for the function $g(y_1)$ when $d\to\sqrt{2}$ and its
power expansion.

In Fig.~\ref{varisL} we summarize some data obtained by the
implementation of the previous items. Concretely, in Fig.~\ref{varisL} top
left we show the points  $\tilde{L}_S$ where $S=M,B,C$.
The points with subscript $+$ and $++$
(resp.  $-$ and $--$) denote the values of $\tilde{L}_S$ for the next and the
second next approximants to $\gamma$.
Since we are dealing
with the golden frequency $\gamma$ we have considered the normalized function
$$
\hat{\Psi}_1(\tilde L)= \Psi_1(\tilde{L})/\sqrt{c_\infty}, \quad \text{being} \ c_\infty=3+\gamma \ \text{the limit value of } \{c_{s,n}\}_n,
$$
and we represent $\hat{\Psi}_1(\tilde L)$ as a function of $\log(\tilde L)$.
We also show the same function translated to the right and to the left by $2
\log(1+\gamma)$. These correspond to the functions
$\Psi_1$ associated to the previous and next best approximants of $\gamma$.
The top left plot corresponds to $c=5,d=7$. In the top center plot we represent
the same as in the top left one, but for values $c=1.1, d=1.5$ close to the
limit. 

In the top right plot of Fig.~\ref{varisL} we represent
$\log(-\hat{\Psi}_1(\tilde{L}))$ as a function of $\log \tilde{L}$ for
$c=5,d=7$, and we check that $\log(-\hat{\Psi}_1(\tilde{L}))$ behaves as $|\log
\tilde{L}|/2$ as follows from the expressions for $I$, $\cP_F$, $\cP_M$ and $\Psi$,
in (\ref{eqI}), (\ref{cPF}), (\ref{cPM}) and (\ref{funPsi}), respectively.
In the logarithmic scale used in the plot we clearly observe that, after
shifting the origin and scaling coordinates, $\hat{\Psi}_1$ behaves as
$\log(\cosh(\tilde{L}))$.

The dependence of the maximum value of $\hat{\Psi}_1(\tilde{L})$ as a function of $(c,d)$
forms the surface shown in the bottom row of Fig.~\ref{varisL}. We recall our
notation: the maximum of $\hat{\Psi}_1(\tilde{L})$ is achieved at
$\tilde{L}=\tilde{L}_M$. As expected all the maxima are negative values.

\begin{figure}[!ht]
\begin{center}
\begin{tabular}{ccc}
\hspace{-6mm}\epsfig{file=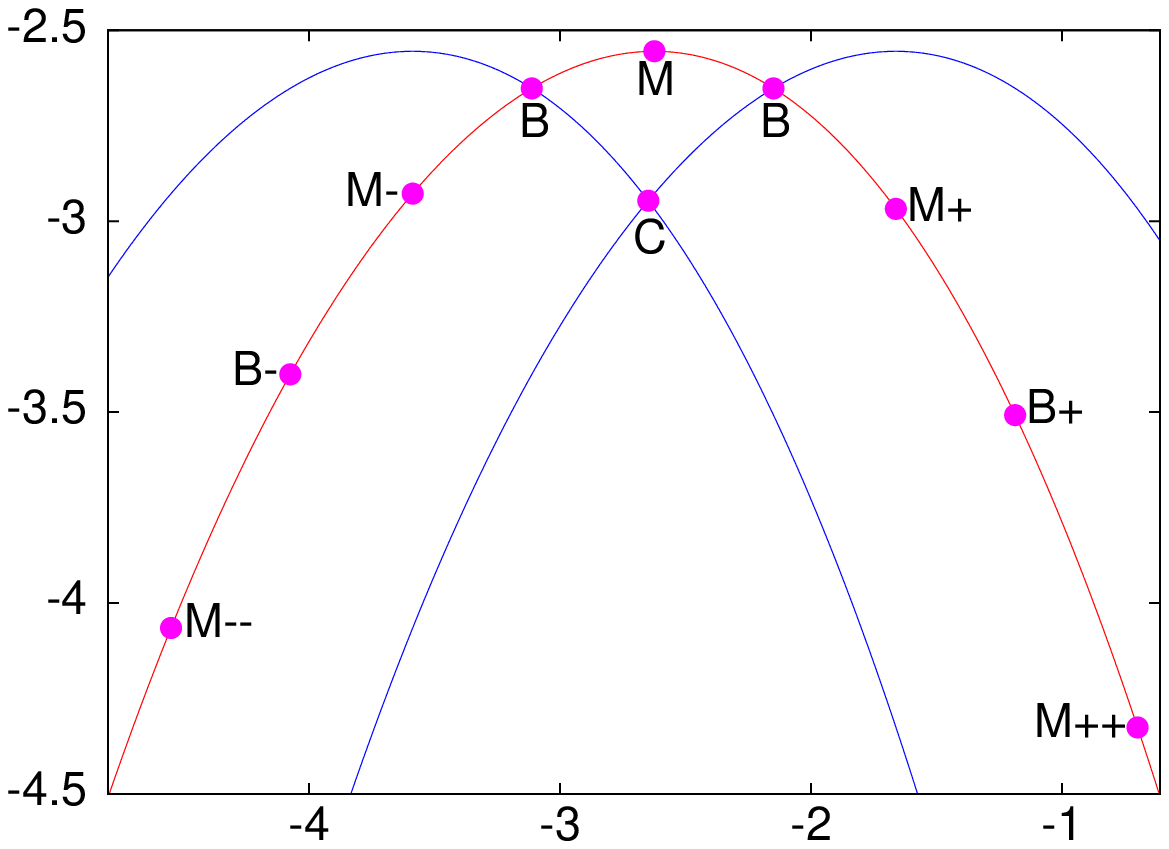,width=6cm} &
\hspace{-6mm}\epsfig{file=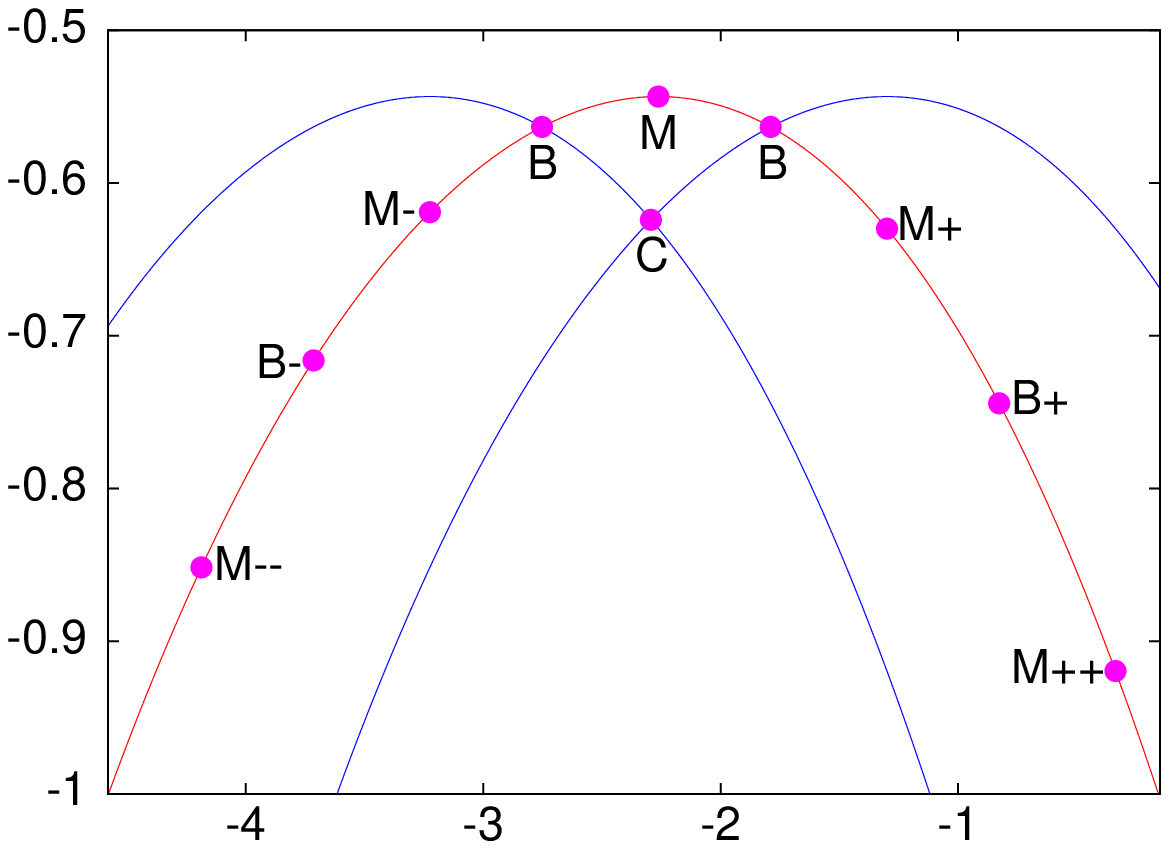,width=6cm} &
\hspace{-6mm}\epsfig{file=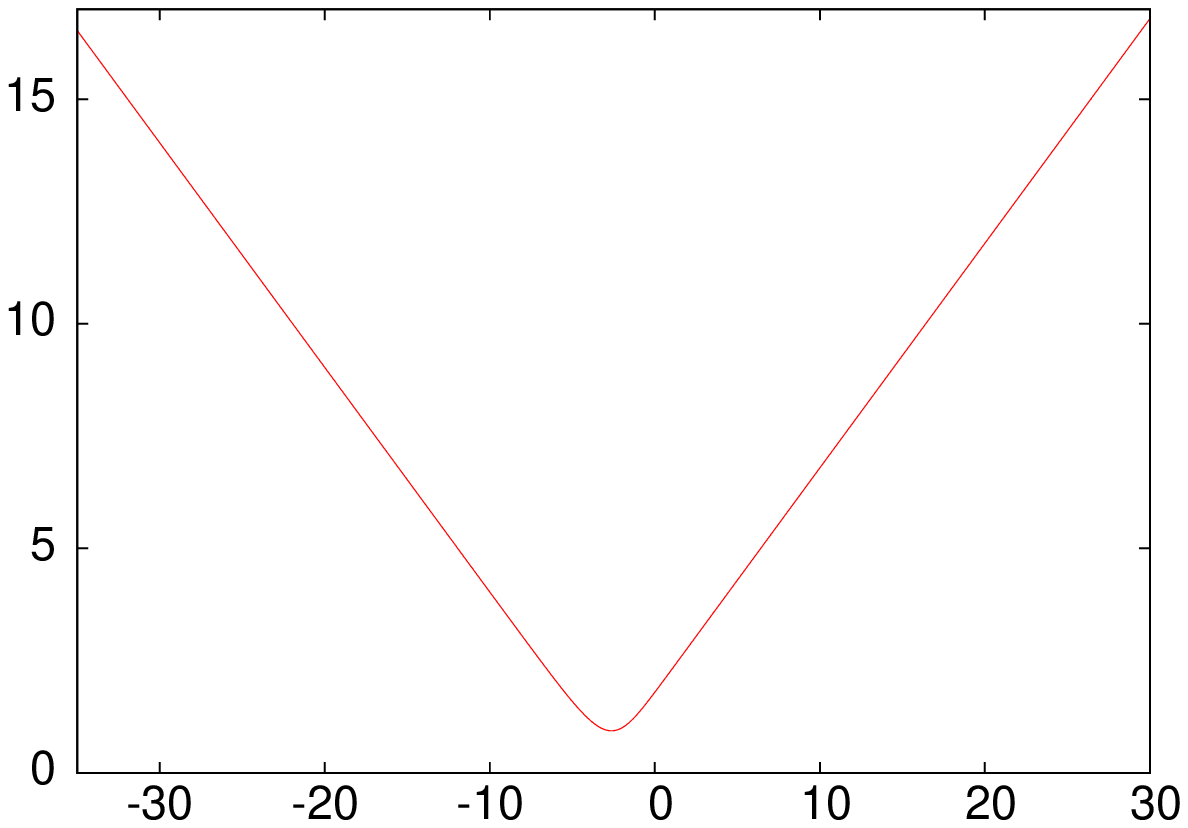,width=6cm}
\end{tabular}
\begin{tabular}{c}
\epsfig{file=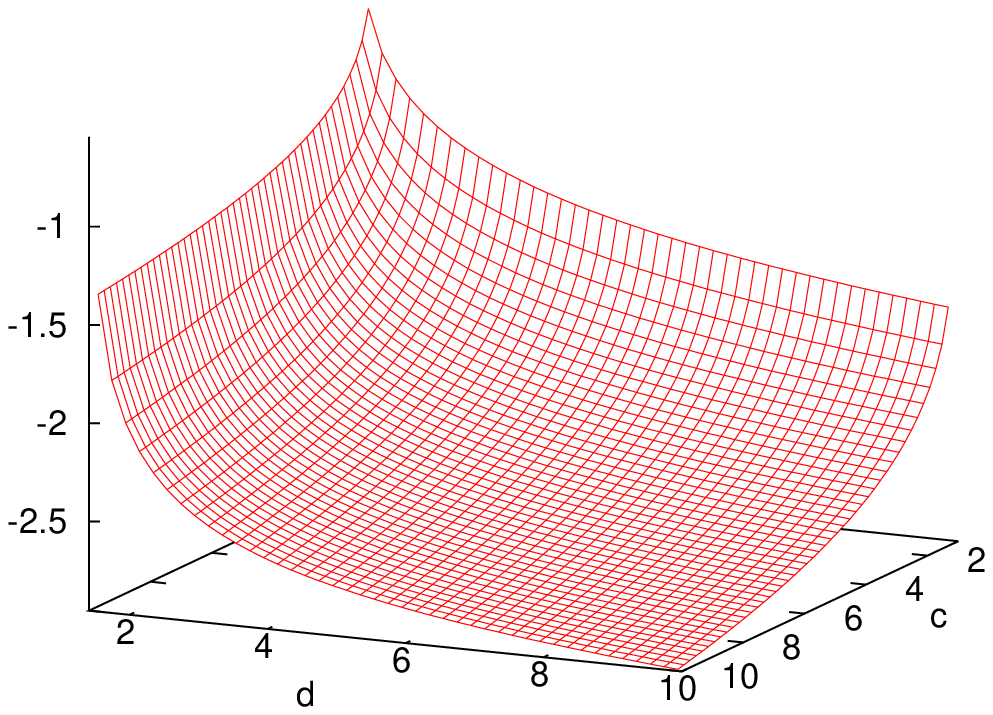,width=10cm}
\vspace{-8mm}
\end{tabular}
\end{center}
\vspace{-6mm}
\caption{$\gamma=(\sqrt{5}-1)/2$. Top left: $\hat{\Psi}_1(\tilde{L})$ as a function of $\log(\tilde{L})$ for the
values $c=5,d=7$. Also we show the same function translated to the left and to the right
by $2\log(1+\gamma)$. The marked points are: $M$ for maximum and
then $M_{++},M_{+},M_{-},M_{--}$ for the $m_1$ values of the previous and next approximants; 
$B$ for the change of dominant harmonic and then $B_{+},B_{-}$
for the nearby approximants too; $C$ for the subdominant harmonic change. Top center:
the same as in the top left but for 
$c,d$ values close to the limit: $c=1.1$, $d=1.5$. Top right: for $c=5,d=7$ we show $\log(-\hat{\Psi}_1(\tilde{L}))$ as a function of
$\log(\tilde{L})$.
Bottom: Maxima of $\hat{\Psi}_1(\tilde{L})$ as a function of $(c,d)$.}
\label{varisL}
\end{figure}

\subsection{The splitting function for different frequencies} \label{Sec:Otherfreq} 

In this section we illustrate what happens for several frequencies $\gamma$. We
show some computations for concrete cases, including the golden mean, for
comparison, in Fig.~\ref{othergamma}. The sequence of dominant harmonics and
the values $\nu=\nu_j$ at which the change of dominant harmonic takes place
depend on the CFE and not only on the
Diophantine properties of $\gamma$. In Fig.~\ref{othergamma} we represent the
contributions $C_{m_1,m_2}$ to $\DFt{1}/\epsilon$ as a function of
$\log_2(\nu)$ for different values of $\gamma$. The results for
$\gamma=(\sqrt{5}-1)/2$ are shown in the top left plot (case 0). Concretely, we
represent the contributions $C_{m_1,m_2}$ corresponding to the approximants of
the golden frequency with $m_1$ between $21$ and $514229$. Compare with
Fig.~\ref{cdh} left. Note that all the approximants become dominant in a
suitable range of $\nu$. However, as can be seen in the plots, this does not
happen for other frequencies $\gamma$.  For concreteness, below we consider the
following cases (the notation $10 \times 1$ in the CFEs below denotes ten consecutive quotients equal to one). 

\begin{align*}
\text{Case 0: } \gamma  &= (\sqrt{5}-1)/2 = [1,1,1,1,1,...] \approx 0.618033988749894848204. \\[0.2cm]
\text{Case 1: } \gamma  &= (55(1+b)+34)/(89(1+b)+55) \text{ with } b=(\sqrt{122}-10)/11, \\
             & \text{hence } \gamma=[10 \! \times \! 1, 1,10,1,1,10,1,1,10,1,...] \approx 0.6180512268192526496794.\\[0.2cm]
\text{Case 2: } \gamma &= (55(1+b)+34)/(89(1+b)+55) \text{ with } b=(\sqrt{140}-10)/20, \\
 & \text{hence } \gamma=[10 \! \times \! 1, 1,10,1,10,1,10,1,10...] \approx 0.6180513744611582707944. \\[0.2cm]
\text{Case 3: } \gamma &= [10 \! \times \! 1,2,3,4,5,6,7,8,9,10,...] \approx 0.6180206632934375446297.
\end{align*} 
For each one of the previous cases, we list the consecutive numerators of the
approximants of $\gamma$ for which the corresponding harmonic term of the
splitting function become dominant (in a suitable range of $\nu$). 
\begin{align*}
\text{Case 0: } & 21,34,55,89,144,233,377,610,987,1597,2584,4181,6765,10946,17711,28657,\\
        & 46368,75025,121393,196418,317811,514229. \\[0.2cm]
\text{Case 1: } & 21,34,89,945,1034,1979,20824,22803,43627,459073,502700. \\[0.2cm]
\text{Case 2: } & 21,34,89, 1034,12319,146794. \\[0.2cm]
\text{Case 3: } & 21,34,55,144,487,2092,10947,67774,485365.
\end{align*}
The contributions of the harmonic terms related to consecutive best
approximants to the total splitting are shown in Fig.~\ref{othergamma}.
In order to explain the results displayed in the figure for different 
frequencies $\gamma$, we investigate the Diophantine properties
of $\gamma$ and relate them to the properties of the constants $c_{s,m_1/m_2}$. 
\begin{figure}[ht]
\begin{center}
\begin{tabular}{cc}
\hspace{-3mm}\epsfig{file=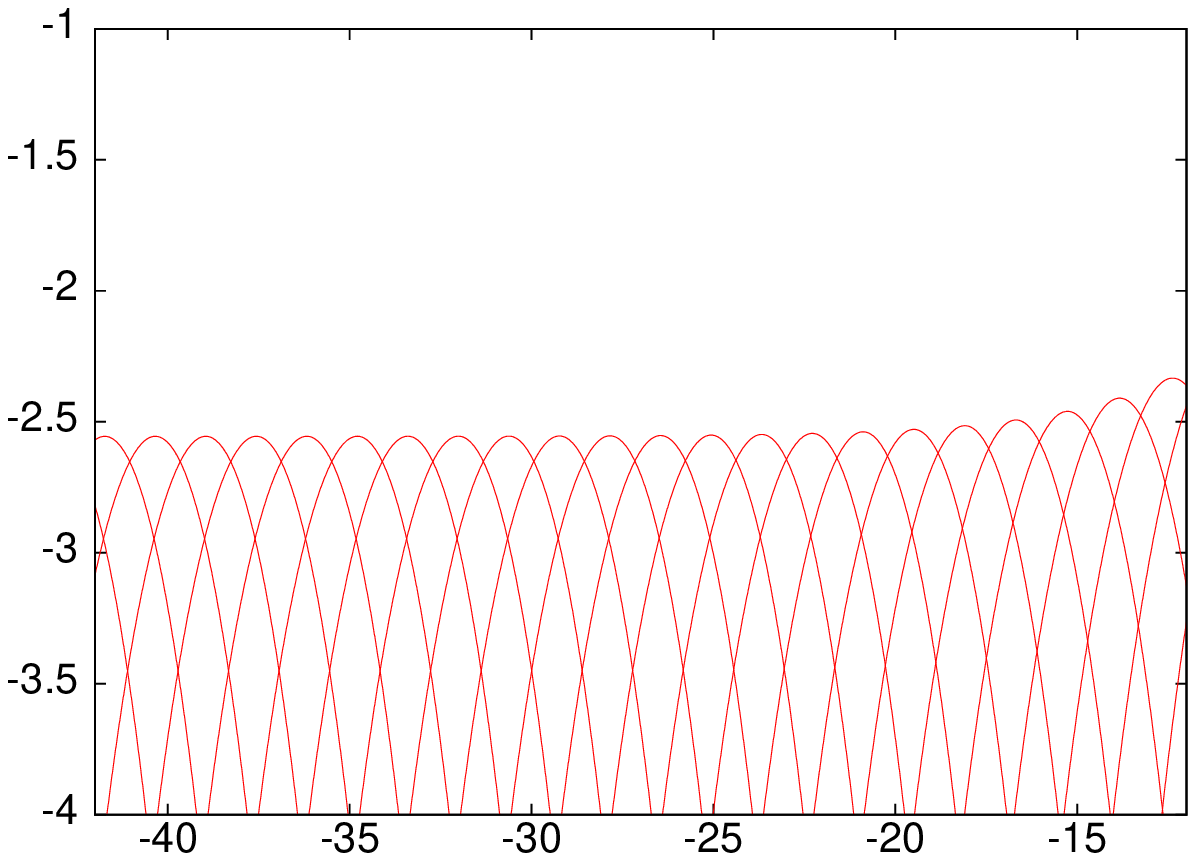,width=8cm} &
\hspace{-6mm}\epsfig{file=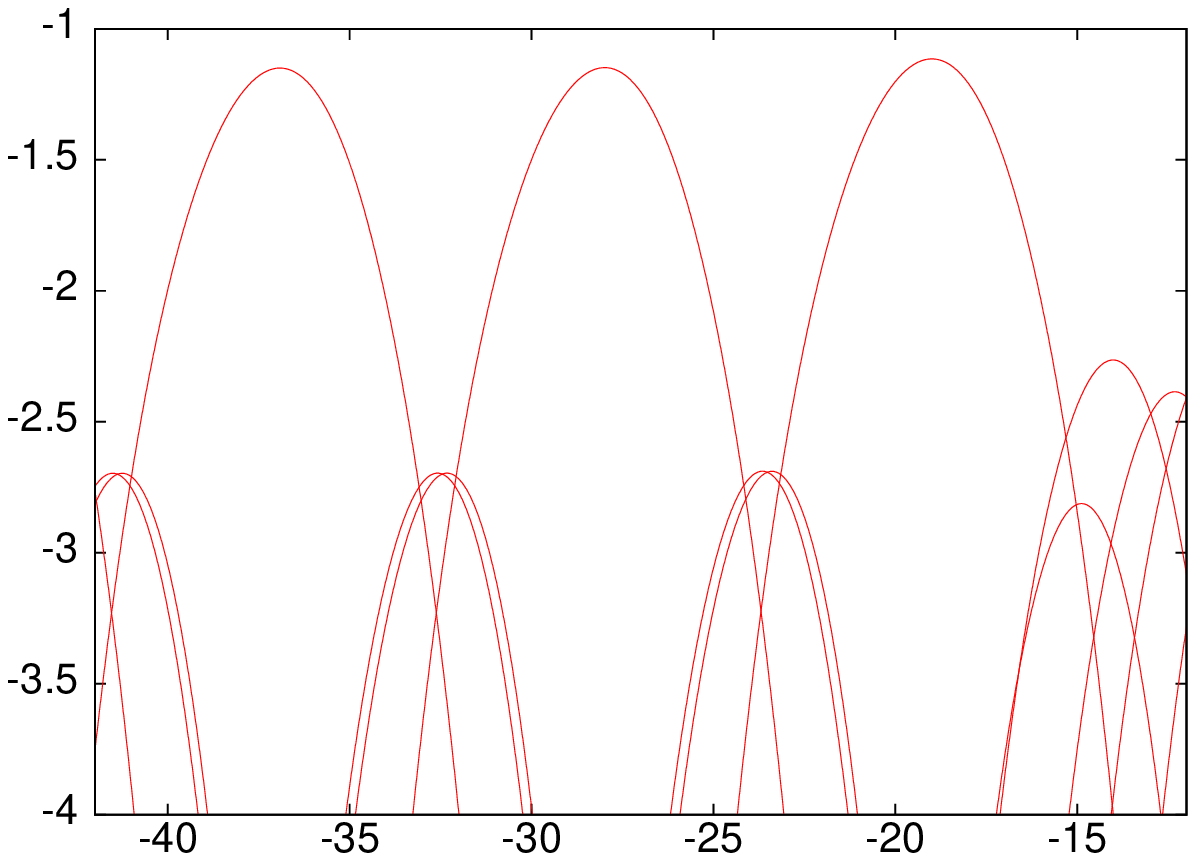,width=8cm} \\
\hspace{-3mm}\epsfig{file=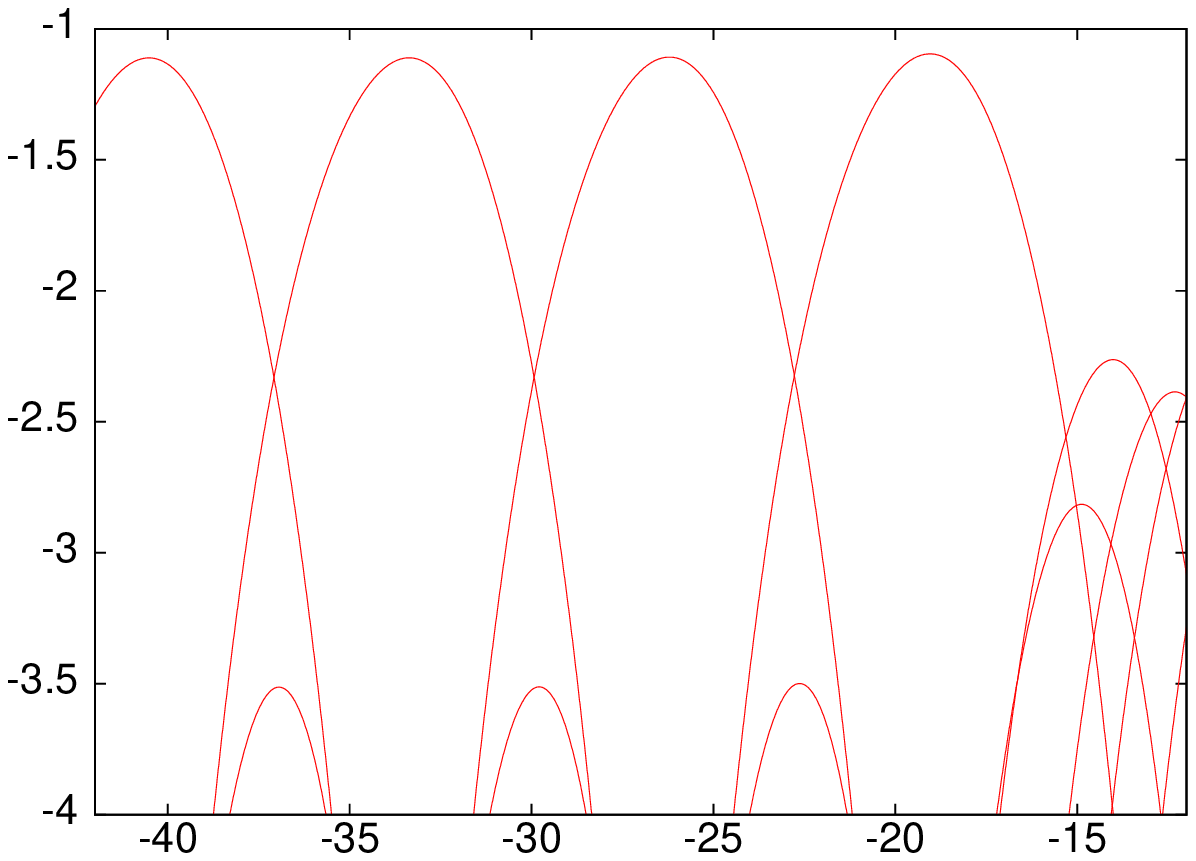,width=8cm} &
\hspace{-6mm}\epsfig{file=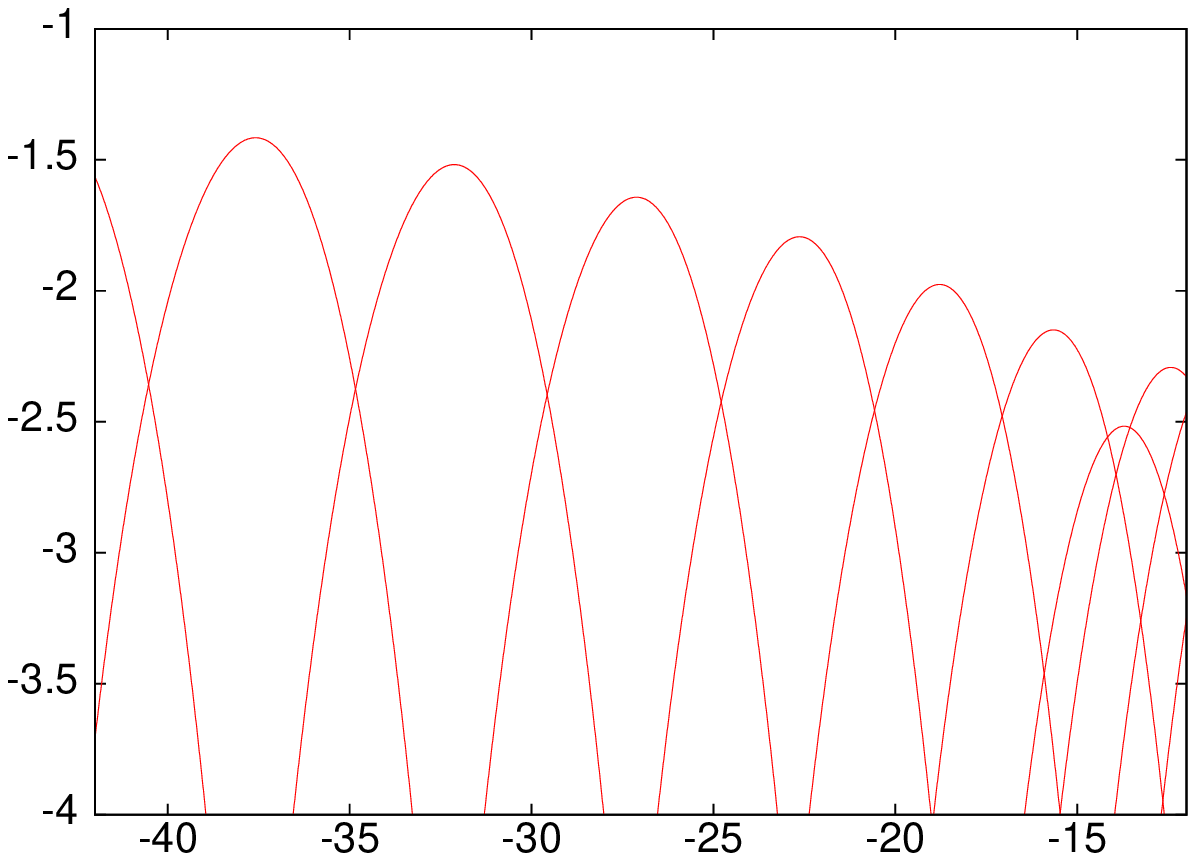,width=8cm} \\
\end{tabular}
\end{center}
\caption{We represent the values of $\sqrt{\nu} \log(C_{m_1,m_2})$ for the approximants
$m_1/m_2$ that contribute to $\DFt{1}/\epsilon$ within the rang of $\nu$ in the plots
($\log_2(\nu)$ ranges in the $x$-axis). 
Top left (Case 0): $\gamma=(\sqrt{5}-1)/2$.  
Top right (Case 1): $\gamma= (55(1+b)+34)/(89(1+b)+55)$, with $b=(\sqrt{122}-10)/11$,
Bottom left (Case 2): $\gamma= (55(1+b)+34)/(89(1+b)+55)$, with
$b=(\sqrt{140}-10)/20$. Bottom right (Case 3): $\gamma=[0;10 \! \times \! 1,2,3,4,5,6,7,8,9,10,...]$.
The same windows have been used in all plots for comparison.
}
\label{othergamma}
\end{figure}

\subsubsection{Periodicity of the constants $c_{s,m_1/m_2}$ for quadratic irrational frequencies} \label{Sec:periodicity_csn}

First, it turns out that for quadratic $\gamma \in \mathbb{R}\setminus
\mathbb{Q}$ the constants $c_{s,m_1/m_2}$ tend to be periodic when $\nu
\rightarrow 0$.  This is a consequence of the basic CFE property in
Lemma~\ref{lema_approximants_distance} below.

Let $\{q_j\}_{j\geq 0}$ be an infinite or finite sequence of natural numbers,
with $q_0 \geq 0$ and $q_j \geq 1$ for $j \geq 1$, which defines a CFE of a real number in the
usual way. Given a frequency $\gamma=[q_0;q_1,q_2,...] = q_0+\frac{1}{q_1 +
\frac{1}{q_2 + ...}}$, denote by $N_n/D_n=[q_0;q_1,\dots,q_n]$, $n\geq 0$, the
$n$-th order approximant of $\gamma$. Introducing $N_{-1}=1$, $D_{-1}=0$,
the following basic properties hold (see for example \cite{Khi64} for proofs).
For all $n \geq 1$,
\vspace{-0.2cm}
\begin{enumerate}
\item[(i)] $N_n = q_n N_{n-1} + N_{n-2}, \qquad D_n=q_n D_{n-1}+D_{n-2}$. 
\item[(ii)] $|D_n N_{n-1} - D_{n-1} N_n| =1$.
\item[(iii)] If $\beta_n=[0;q_{n+1}, ... ]$ then
$\displaystyle{\gamma = [q_0;q_1,q_2,...,q_{n-1},q_{n}+\beta_n] = \frac{N_n + \beta_n N_{n-1}}{D_n + \beta_n D_{n-1}}.}$
\item[(iv)] $\displaystyle{\frac{D_{n-1}}{D_n} =  [0;q_n,q_{n-1},...,q_1]}.$
\end{enumerate}
We introduce the notation $q_{+,n}=[q_{n+1};q_{n+2},...]$ and $q_{-,n}=[q_n;q_{n-1},...,q_1]$. 

\begin{lema}\label{lema_approximants_distance}
The distance between the $n$-th order approximant and $\gamma$, for arbitrary
$\gamma \in \mathbb{R} \setminus \mathbb{Q}$, satisfies 
$$
\left( D_n \left| D_n  \gamma - N_n \right| \right)^{-1} = [q_{n+1};q_{n+2},...]+[0;q_n,q_{n-1},...,q_1].
$$
\end{lema}

\begin{proof}
From properties (ii), (iii) and (iv) one has
$$
\left|\gamma - \frac{N_n}{D_n}\right| =  \frac{\beta_n}{D_n^2 (1+\beta_n \frac{D_{n-1}}{D_n})} = \frac{ \beta_n}{D_n^2(1+ \beta_n [0;q_n,q_{n+1},...,q_1])}.
$$
This implies the result.
\end{proof}

It is known that $\gamma \in \mathbb{R}$ is a quadratic irrational number if, and only if, its CFE is eventually periodic.
\begin{lema}  \label{lema_periodic_csn}
Let $\gamma$ be a quadratic irrational number with eventually $p$-periodic CFE. Let
$c_{s,n}=c_{s,N_n/D_n}=(N_n |D_n \gamma - N_n|)^{-1}$. Then, the sequence
of constants $\{c_{s,n} \}_{n \geq 1}$ is asymptotically $p$-periodic (as $n
\rightarrow \infty$).  
\end{lema}
\begin{proof}
The statement follows from the relation
$$
\left( D_n^2 \left| \gamma - \frac{N_n}{D_n} \right| \right)^{-1} = \frac{N_n}{D_n} \, c_{s,n}, 
$$
which, using the previous Lemma~\ref{lema_approximants_distance}, implies that
\begin{equation} \label{csnqpm}
c_{s,n}  \approx \frac{q_{+,n} + 1/q_{-,n}}{\gamma} (1 + \mathcal{O}(D_n^{-2})).  
\end{equation}

If $\gamma$ is a quadratic irrational number then, taking $n$ large enough, the
sequence of quotients of $q_{+,n}$ is periodic and the one of $q_{-,n}$ tends to be 
periodic, that is, its quotients are periodic except maybe some final ones that have small influence
on the value of $q_{-,n}$ if $n$ is large. This implies that $c_{s,n}$ tend to be
periodic with respect to $n$ with the same period as the CFE of $\gamma$. 
\end{proof}

In particular, for the values of $\gamma$ referred as Cases $0$, $1$, and $2$
in Section~\ref{Sec:Otherfreq} one has: 
\begin{align*}
\text{Case 0: } & c_{s,n} \rightarrow 3 + \gamma \approx 3.61803398 \text{ as } n \rightarrow \infty. \\[0.2cm]
\text{Case 1: } & \left\{c_{s,n}\right\}_n \text{ tend to be 3-periodic. One has }\\ 
                & \hspace{1cm} c_{s,n} \to 17.871271 \dots, \quad c_{s,n+1}=c_{s,n+2} \to 3.249322 \dots \text{ for } n=2 \, (\text{mod} \, 3).\\[0.2cm] 
\text{Case 2: } &\left\{c_{s,n}\right\}_n \text{ tend to be 2-periodic. One has,} \\
                & \hspace{1cm} c_{s,n} \to a \approx 1.91442978 \text { for $n$ even, and } \quad c_{s,n}\to 10 a \text{ for $n$ odd.}
\end{align*}
Then, for the $n$-th approximant of $\gamma$, say $N_n/D_n$, the corresponding maxima shown in
Fig.~\ref{othergamma} are approximated by $\Psi_M/\sqrt{c_{s,n}}$, being $\Psi_M \approx
-4.860298$, and they are located at $\nu \approx L_M/(N_n^2 c_{s,n})$, where $L_M
\approx 0.26236$. For example, in Case 2 the $4$-th visible maximum (from right
to left) shown in the bottom left panel of Fig.~\ref{othergamma} is related to
$N=1034$ and corresponds to $c_{s,n} \to 10a$. Accordingly its value is $\approx
-1.1108186876015$ and it is located at $\log_2(\nu) \approx -26.2172640940432$
in agreement with what is shown in the figure.

\subsubsection{Diophantine properties of frequencies with unbounded CFE} \label{sect_unboundedCFE}

In Case $3$ the frequency $\gamma$ has an unbounded CFE and $c_{s,n}$ tend to
infinity as $n \rightarrow \infty$. 

On the other hand, for $\gamma=e-2=[0;1,2,1,1,4,1,1,6,1,1,8,...]$ different
behaviours of the constants $c_{s,n}$ are mixed. The sequence of best
approximants $N_n/D_n$ of $\gamma=e-2$ is 
\[ 1/1, \ 2/3, \ 3/4, \ 5/7,\ 23/32,\ 28/39,\ 51/71,\ 334/465,\ 385/536,\ 719/1001, \dots \] 
The values $c_{s,n}$ associated to the approximant $N_n/D_n$ are such that the
subsequence $\{c_{s,3m+1}\}_{m \geq 0}$ tends to $\infty$ linearly with slope
$2/(3\gamma)$. The other two subsequences of $c_{s,n}$ are bounded,
being $c_{s,3m} < c_{s,3m+2}$ for all $m \geq 1$, and they both tend to $2/\gamma$.
This explains the bumps observed in Fig.~\ref{em2}.

We give further details on the Diophantine properties of the previous unbounded
CFE cases. For concreteness, we consider $\gamma_1=[0;1,2,3,4,5,6,7,8,...] \approx 0.69777465796400798200679$  and
$\gamma_2=e-2=[0;1,2,1,1,4,1,1,6,1,1,8,...]$.  As usual, to get Diophantine
approximations the idea is to look for a function $\phi(D_n)$ such that
$\phi(D_n) |D_n \gamma - N_n|$ is bounded from below. From the identity $c_{s,n}=(N_n | D_n \gamma - N_n|)^{-1}$
we can take $\phi(D_n)=N_n c_{s,n} \approx D_n \gamma c_{s,n}$, and we note
that the constants $c_{s,n}$ can be approximated from the quotients $q_n$ of
$\gamma$ using (\ref{csnqpm}).

\begin{lema} \label{Lg1}
Let $\gamma_1=[0;1,2,3,4,5,6,7,8,...]$. There exists a constant $c>0$ such that \footnote{Note that $\phi(q) < q^{\tau}$ for any $\tau>1$ if $q$ is large enough. Equivalently, $\gamma_1$ satisfies the Diophantine condition $|\gamma_1 - p/q| \geq c/q^{\tau}$ for any $\tau=2+\epsilon$, $\epsilon>0$ and a suitable constant $c=c(\epsilon)>0$.}
 \[|q \gamma_1 - p| \geq \frac{c}{\phi(q)}, \qquad \phi(q)=q \log(q)/\log(\log(q)),\]
for all $p,q \in \mathbb{Z}$ with $q\geq 3$.
\end{lema}

\begin{proof}
Since $q_n=n$ one has $q_{+,n} = (n+1)(1+\mathcal{O}(n^{-2}))$,
$q_{-,n}=n(1+\mathcal{O}(n^{-2}))$, and from (\ref{csnqpm}) it follows that
$c_{s,n} \gamma_1 \approx (n+1)(1+\mathcal{O}(n^{-2}))$.  To obtain an explicit
formula for $\phi(n)$ one has to relate $c_{s,n}$ with $D_n$. Note that $D_n =
D_{n-1} q_{-,n}$ and then $D_n$ equals $n!$ times a finite product of terms
that are convergent when $n \rightarrow \infty$ (because $\sum_{n\geq1}
n^{-2}=\pi^2/6$).  Stirling's approximation provides the relation $$\log(D_n) =
n \log n (1+\mathcal{O}(1/\log(n))),$$ which can be solved by Newton iteration
(note that the Newton-Kantorovich theorem guarantees that the iteration
starting with $n_0=\log(D_n)/ \log(\log(D_n))$ converges provided $D_n$ is
large enough) to obtain
$$
n = \frac{\log(D_n)}{\log(\log(D_n))} \left( 1 + \mathcal{O}\left( \frac{\log(\log(\log(D_n)))}{\log(\log(D_n))}\right)\right).
$$
We conclude that $\phi(D_n)=D_n \log(D_n)/\log(\log(D_n))$ ensures a positive lower
bound of the scaled difference $\phi(D_n) | D_n \gamma_1 - N_n|$.  If $D_n \leq q < D_{n+1}$ then  $|q-\gamma_1
p| \geq | D_n \gamma_1 - N_n| \geq c/\phi(D_n) \geq c/\phi(q)$.
Changing the constant $c$ we extend the inequality for $q \geq 3$.
\end{proof}

\begin{remark} 
Numerically we observe that $D_n/n! \rightarrow 2.2796$ as $n \rightarrow \infty$.
Let $\Pi_n= \phi(D_n)|D_n \gamma - N_n|$. The sequence $\{\Pi_n\}_{n > 1}$ (for
$n=1$ it is not defined!) reaches a minimum values for $n=6$ (that is when evaluated on the 
approximant $N_6/D_6=972/1393$ for which on has $\Pi_6
\approx 0.50201173$) while uniformly increases for $n>6$. For $n=1000$ one has
$\Pi_{1000} \approx 0.68014970$ and $\Pi_n \approx 0.758203198$ for $n=10^5$. The function $g(n)=|D_n\gamma_1 - N_n|n D_n
$ is such that $g(1)=1-\gamma_1$, it is monotonically increasing and it tends to one as
$1-\mathcal{O}(1/n)$ when $n\rightarrow \infty$ as expected.  
\end{remark}

We proceed similarly for $\gamma=\gamma_2=e-2$. The following lemma asserts
that $\gamma_2$ has similar Diophantine properties to the ones described in Lemma~\ref{Lg1} for
$\gamma_1$.
\begin{lema} \label{Lg2}
There exists a constant $c>0$ such that
\[ |q \gamma_2 - p| \geq c/\phi(q), \quad  \phi(q)=q \log(q)/\log(\log(q)),\]
for all $p,q \in \mathbb{Z}$ with $q \geq 3$.
\end{lema}

\begin{proof}
We have $q_n=2(n+1)/3$ if $n=2 \text{ (mod 3)}$ and $q_n=1$ otherwise. We consider $k=3j+2$ below.

One has $q_{+,3j+1}=[q_{3j+2};1,1,q_{+,3j+4}]$ with $q_{3j+2}=2(j+1)$ and $q_{+,3j+4}=\mathcal{O}(j)$,
which implies that $q_{+,3j+1}=(2(j+1)+1/2)(1+\mathcal{O}(j^{-2}))=(2(k+1)/3+1/2)(1+\mathcal{O}(k^{-2}))$.
From the relation $q_{+,3j}=q_{3j+1}+1/q_{+,3j+1}$ it follows that $q_{+,3j}=(1+3/(2k))(1+\mathcal{O}(k^{-2}))$.
Using that $q_{+,3j+2}=[1;1,q_{+,3j+4}]$ where $q_{+,3j+4}=(2(j+2)+1/2)(1+\mathcal{O}(j^{-2}))$ one obtains
$q_{+,3j+2}=(2-3/(2k))(1+\mathcal{O}(k^{-2}))$.

Moreover, since $q_{-,3j+2}=[2(j+1);1,1,q_{-,3j-1}]$ and
$q_{-,3j-1}=\mathcal{O}(j)$,  one gets
$q_{-,3j+2}=(2(k+1)/3+1/2)(1+\mathcal{O}(k^{-2}))$.  From
$q_{-,3j+1}=[q_{3j+1}; q_{3j}, q_{-,3j-1}]$, using that
$q_{-,3j-1}=2j(1+\mathcal{O}(j^{-1}))$, it follows that
$q_{-,3j+1}=(2-3/(2k))(1+\mathcal{O}(k^{-2}))$. Since
$q_{-,3j}=[q_{3j};q_{-,3j-1}]$ one obtains
$q_{-,3j}=(1+3/(2k))(1+\mathcal{O}(k^{-2}))$.

Summarizing, for $k=2 \text{(mod 3)}$, we obtain 
\[
 \begin{array}{lll}
 q_{+,k-2}  \approx 1+\frac{3}{2k}, \qquad & q_{+,k-1} \approx \frac{2(k+1)}{3}+ \frac{1}{2}, \qquad & q_{+,k} \approx 2-\frac{3}{2k}, \\
 q_{-,k-2}  \approx 1+\frac{3}{2k}, & q_{-,k-1} \approx 2-\frac{3}{2k}, & q_{-,k} \approx \frac{2(k+1)}{3}+ \frac{1}{2}, 
\end{array}
\]
with a relative error $\mathcal{O}(k^{-2})$ in all cases.
Using (\ref{csnqpm}) and the previous estimates we get $c_{s,k-2}\gamma_2 \approx 2$,
$c_{s,k-1}\gamma_2 \approx 2(k+1)/3+1$ and $c_{s,k} \gamma_2 \approx 2$. We conclude that
$c_{s,n}, n\geq 1$ is, at most, $\mathcal{O}(n)$.

Next, we look to the denominators $D_n$. We use the properties listed in the
items before Lemma~\ref{lema_approximants_distance}.  The recurrence $D_n=q_n D_{n-1} +
D_{n-2}$, implies that $D_{3j+2} = 2(j+1)( D_{3j}+D_{3j-1})+D_{3j}$ and, using
the identity $D_{3j}=D_{3j-1} q_{-,3j}$, it simplifies to 
\[ D_{3j+2}= \left( 2(j+1)(1+q_{-,3j})+q_{-,3j} \right) D_{3j-1}. \]
Since $q_{-,3j}=[q_{3j};q_{-,3j-1}]=(1+1/(2j))(1+\mathcal{O}(j^{-2}))$, one obtains
\[ D_{3j+2} = 4 \left(j + \frac{3}{2} \right) D_{3j-1} (1+\mathcal{O}(j^{-2})) = \frac{4k+10}{3} \, D_{3j-1} (1+\mathcal{O}(k^{-2})) .\]
From this recurrence we obtain $D_{3j+2}\approx 4^j \Gamma(j+5/2)$ or, equivalently,
\begin{equation} \label{Dn2}
 D_n \approx  4^{n/3} \Gamma(n/3+11/6), \qquad \text{for $n$ such that $n=2 (\text{mod }3$).} 
\end{equation}
Then, Stirling's approximation gives
\[ 3 \log D_n \approx n \log n - an + 4\log n, \qquad a=1+\log 3/4.\]
As in Lemma~\ref{Lg1}, we take $n_0= 3 \log(D_n)/\log(\log(D_n))$ and solve
this relation by (Newton) iteration (the convergence follows from the
Newton-Kantorovich theorem). We obtain 
\[ n = \frac{3 \log(D_n)}{\log(\log(D_n))} \left( 1 + \mathcal{O}\left( \frac{\log(\log(\log(D_n)))}{\log(\log(D_n))}\right)\right).\]
Now the proof finishes as the proof of Lemma~\ref{Lg1}. Since
$c_{s,n}=\mathcal{O}(n)$ then one takes 
$\phi(q)=q \log(q)/\log(\log(q))$ to have values
of $|q \gamma_2-p| \phi(q)$ bounded from below.  \end{proof}

We show in Fig.~\ref{phiq} left the values of $\xi_n=|D_n \gamma_2 -
N_n| \phi(D_n)$ as a function of $n$. In the right plot we only show the local
minima of $\xi_n$ (i.e. those corresponding
to  $n=1 (\text{mod }3$)). According to the theoretical predictions, the
minimum values tend to a constant.

\begin{figure}[ht]
\begin{center}
\begin{tabular}{cc}
\hspace{-4mm}\epsfig{file=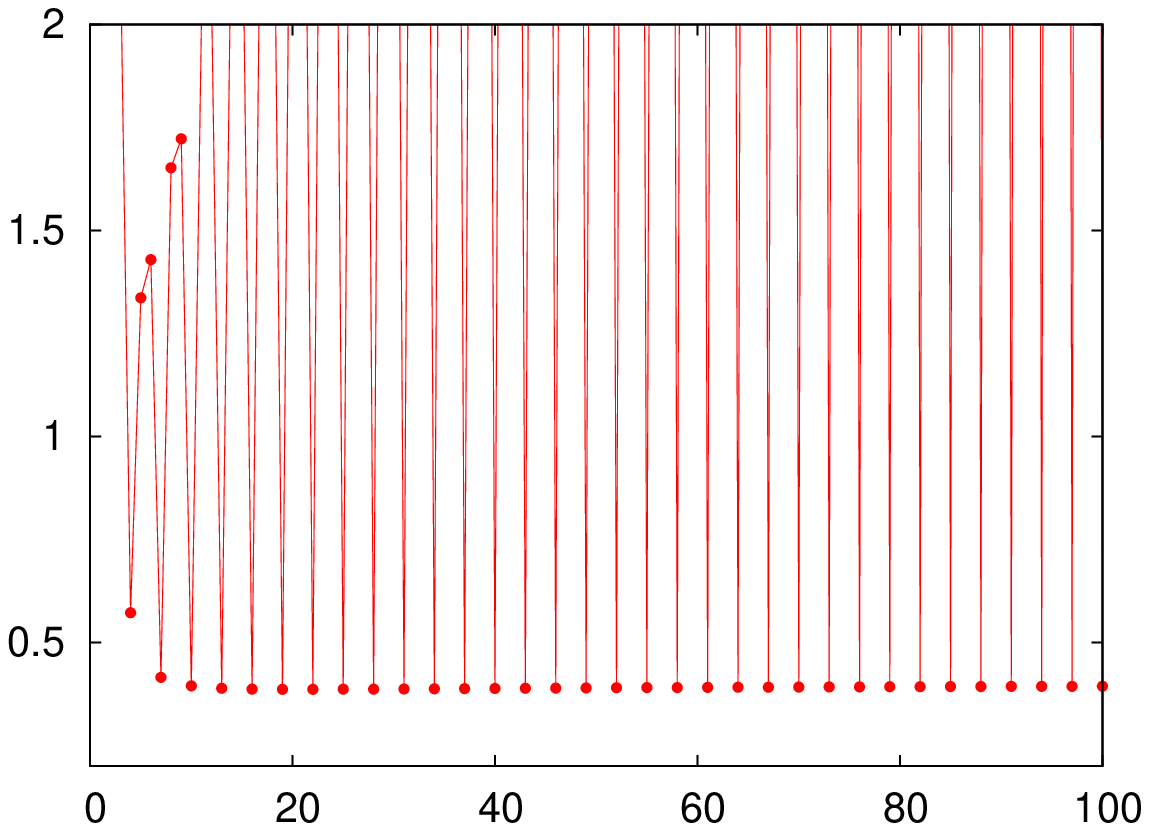,width=0.5\textwidth} &
\hspace{-6mm}\epsfig{file=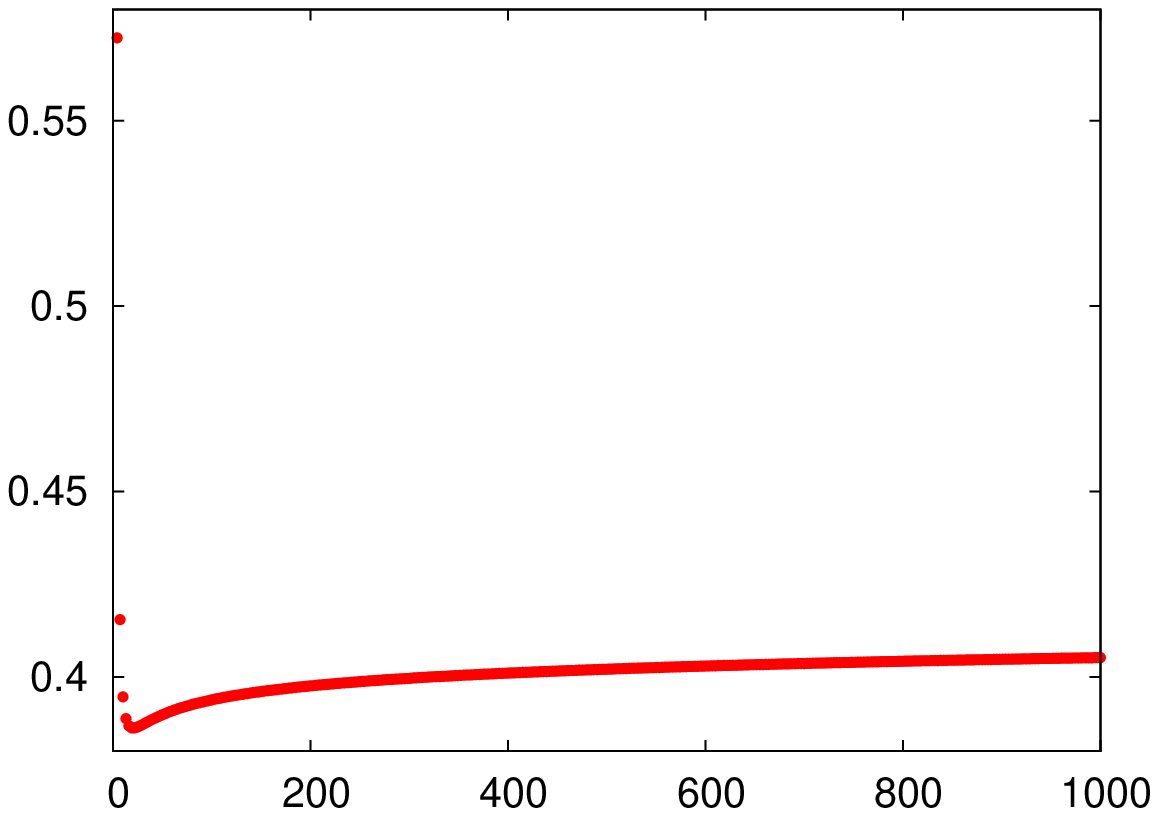,width=0.5\textwidth}
\end{tabular}
\end{center}
\vspace{-8mm}
\caption{Left: We represent $\xi_n=|D_n \gamma_2-N_n|\phi(D_n)$ as a function of $n$.
In the right plot we focus on the behaviour of the minima  (up to $n=1000$).
}
\label{phiq}
\end{figure}

\begin{remark}\label{rk7p3}
We have seen that the frequencies $\gamma_1$ and $\gamma_2$ satisfy a condition of the form
$|q \gamma -p| \geq c \log(\log(q))/(q \log(q))$, $q\geq 3$. We remark that the set
of irrational numbers $\gamma$ satisfying such a type of condition for some $c>0$ has zero measure.
By contrast, the set of irrational numbers $\gamma$ satisfying a
condition of the form $|q\gamma - p|\geq c/\phi(q)$ for $q\geq q_0$
for some $c>0$ and such that $1/\phi(q)$ is integrable in the range
$[q_0,\infty)$, has total measure. We refer to \cite{Khi64} for further
details on measure aspects of CFE. As examples we can consider numbers
of the form $\gamma=[0;1^m,2^m,3^m,4^m,5^m,\ldots]$ for
$1<m\in\mathbb{Z}_+$. They satisfy
$|q\gamma - p|\geq c (\log(\log(q)))^m/(q(\log(q))^m)$ for $q\geq 3$
and a positive constant $c$.
\end{remark}

\begin{remark} \label{rk7p4} 
It follows from the reasoning in Remark~\ref{remark_k1k2} that if $\gamma$ satisfies a Diophantine
condition of the form $|q \gamma - p| \geq c |q|^{-\tau}$, $\tau \geq 1$, $c>0$,
then the exponentially small part of the splitting is expected to have an exponent
of the form $-C/\nu^{1/(\tau+1)}$ with $C>0$. A similar reasoning shows that if $\gamma$ satisfies
a condition of the form $|q \gamma -p| \geq c \log(\log(q))/(q \log(q))$, $q\geq 3$, $c>0$,
as the ones considered in this section, then the exponent of the exponentially
small part becomes 
\[ \left(\frac{-C\sqrt{\log(\log(\sqrt{1/\nu}))}}{\sqrt{\nu\log(\sqrt{1/\nu})}}
\right)(1+o(1)),\quad C>0,\]
where the $o(1)$ terms are bounded by $\log(\log(1/\nu))/\log(1/\nu)$.
Consequently, if one represents $\log(|\Delta F_1^{1}|/\epsilon)$ multiplied by
$\sqrt{\nu \log(1/\nu)/\log(\log(1/\nu))}$ (instead of by $\sqrt{\nu}$ as a
function of $\log_2(\nu)$ as we did in Fig.~\ref{em2} and in Case 3 of
Fig.~\ref{othergamma}), then the maxima tend to a constant value. In the
multiplying factor we have neglected $\log(2)$ in front of $\log(\sqrt{1/\nu})$.
\end{remark}


\section{Conclusions and future work} \label{Sec:Conclusion}

In this work we have investigated the asymptotic properties of the splitting of
the invariant manifolds emanating from a complex-saddle fixed point of a 2-dof
Hamiltonian system $H_0(\bx,\by)$ undergoing a Hamiltonian-Hopf bifurcation (at
$\nu=0$) when acting a periodic forcing on the system. We have obtained
detailed information of the exponentially small behaviour, describing the
changes of dominant harmonic as $\nu \rightarrow 0$. As has been discussed through
the paper, for the concrete example considered, when using Poincar\'e-Melnikov method it remains to bound the effect of
\begin{itemize}
\item the terms in the first order Melnikov approximation not related to best
approximants and,
\item the non-dominant terms in the splitting function, bounding the effect of
higher order Melnikov approximations.
\end{itemize}

In any case, the detailed description presented in this paper takes advantage
of the concrete properties of the explicit periodic perturbation $\epsilon
H_1(\bx,\by,t)$ considered. In this sense, an interesting topic for future
works would be to consider other perturbations $H_1(\bx,\by,t)$, for example:
\begin{itemize}
\item perturbations having a finite number of harmonics (then higher order
Melnikov analysis could be required to analyse the changes of dominant harmonic).
\item perturbations leading to a periodic orbit from the fixed point of the
system.
\end{itemize}

As a consequence of the splitting of the invariant manifolds, in a
neighbourhood of the stable/unstable invariant manifolds there is a region
where rich dynamics appears. A desirable tool to investigate this dynamics
would be a suitable return map adapted to this problem. Such a return map
depends on two key ingredients: the return time to a suitable Poincar\'e
section and the splitting function. This is an interesting problem, motivated
by the slow diffusive expected properties (see Appendix~\ref{split-difusion}),
that we postpone to study in future works. Concretely, it involves:

\begin{itemize}
\item to construct a 4D separatrix map adapted to this problem. As said, this
requires not only the splitting function (see Section~\ref{Sec:Splitting}) but
also the passage time close to the complex-saddle point, and

\item to provide a description of the geometry of the phase space (resonance
web) and analyse the diffusive properties of the model. Note, however, that 
to observe the asymptotic behaviour requires very small values of $\nu$, outside
the range of interest of any physical application.
\end{itemize}

Finally, we also note that the case considered is somehow an intermediate
case between the 2-dof Hamiltonian case (in which the splitting behaves as the one
of a periodic perturbation of an integrable system) and the splitting of the separatrices
for a family of 4D symplectic maps undergoing a Hamiltonian-Hopf bifurcation (in which the
perturbation is not explicit). Then, a natural continuation of this work would
be to consider the analogous Hamiltonian-Hopf bifurcation for 4D symplectic
maps and to study the splitting of the invariant manifolds and the consequences
in the diffusion properties.

\appendix

\section{Autonomous perturbation of the system} \label{autonomous}

In this appendix we study the effect of an autonomous entire perturbation of $H_0$ in
(\ref{H0}) of the form
$$
H(\bx,\by) = H_0(\bx,\by) + \epsilon H_1(\bx,\by),
$$
where $H_1(\bx,\by) = \sum_{k_1,k_2,l_1,l_2} x_1^{k_1} x_2^{k_2} y_1^{l_1} y_2^{l_2}$, and where the sum
is considered over a finite number of indices $k_1,k_2,l_1,l_2 \in \mathbb{N}$.
We recall that $H_0$ depends on $\nu$.
We assume that the perturbation keeps zero as an equilibrium point.
In this case, $H(\bx,\by)$ is a first integral and both $W^{u/s}(\bf 0)$ lie on
$H=0$.  Consequently, the splitting can be measured by the variation of
$F_1=\Gamma_1(\bx,\by)=x_1 y_2 - x_2 y_1$ which is a first integral of the
unperturbed system.

Let ${\bf p_0}$ be a point in the 2-dimensional homoclinic connection of the
unperturbed system given by (\ref{homoH0}) with $t=0$
and angle $\psi_0$. We identify ${\bf p}_0$ with $\psi_0$.  Denote by
$\varphi(t,{\bf p_0})$ the unperturbed solution starting at ${\bf p_0}$.
The first-order Melnikov function to measure the variation of $F_1$ is given
by
$$
M_1(\psi_0)= \int_{-\infty}^{\infty} \{\Gamma_1(\varphi(t,{\bf p_0})),H_1(\varphi(t,{\bf p_0}))\} \ dt = \int_{-\infty}^{\infty} (D\Gamma_1 \cdot \mathbb{J} DH_1^\top)_{|\varphi(t,{\bf p_0})} \ dt,
$$
see \cite{Rob96} for details of the Melnikov method in this context.
By linearity it is enough to consider separately the effect of every individual monomial, i.e.
$H_1(x_1,x_2,y_1,y_2)=x_1^{k_1} x_2^{k_2} y_1^{l_1} y_2^{l_2}$.  Let
$r=k_1+k_2+l_1+l_2$ be the degree of the monomial. One has
$$
M_1(\psi_0)= \int_{-\infty}^{\infty}  \left[ m \cos^{m-1}(\psi) \sin^{n+1}(\psi)  - n  \cos^{m+1}(\psi) \sin^{n-1}(\psi) \right] \ R(t) \ dt,
$$
where $R(t)=R_1(t)^k R_2(t)^l$, $k=k_1+k_2$, $l=l_1+l_2$, $m=k_1+l_1$, $n=k_2+l_2$ and $\psi=\psi(t)=t+\psi_0$, $\psi_0 \in \mathbb{R}$,
and $R_1(t),R_2(t)$ as given in Section~\ref{sec3p1}. 

We consider first the case $r$ odd. For the expansions below, we introduce
$$
C(\hat{n},\hat{m}) := \sum_{s=0}^{\hat{n}} (-1)^{s} \binom{\hat{n}}{s} \binom{\hat{m}}{(r+1)/2-s} = (-1)^{\hat{n}} \sum_{s=0}^{\hat{n}} (-1)^{s} \binom{\hat{n}}{s} \binom{\hat{m}}{(r-1)/2-s},
$$
where $r = \hat{m} + \hat{n}$ and where the equality follows from Pascal's rule.
Here, if $b>a$, we assume $\binom{a}{b}=0$.
Consequently, if one considers just the contribution of the dominant
first-order harmonic one has
$$
M_1(\psi_0) \approx \frac{m C_1 + n C_2}{2^{r-1} }  \int_{-\infty}^{\infty}  \frac{(e^{i \psi} - (-1)^n e^{-i\psi})}{2 i^{n-1}} R(t) \ dt, 
$$
where $C_1=C(n+1,m-1)$  and $C_2= C(n-1,m+1)$. We note that $m C_1 + n C_2 \neq
0$. The evaluation of the previous integral reduces to a linear combination
(with coefficients depending on $\psi_0$) of integrals of the form 
$$
\int_{-\infty}^{\infty}  \sin(t) \sinh^k (\nu t) \cosh^{-(2k+l)} (\nu t)  \ dt
\quad \text{ and } \quad 
\int_{-\infty}^{\infty}  \cos(t) \sinh^k (\nu t) \cosh^{-(2k+l)} (\nu t)  \ dt.
$$
One of the two previous integrals vanishes (depending on the parity of $k$).
The other can be evaluated by residues and one obtains
$$
M_1(\psi_0) \approx A \frac{e^{-\frac{\pi}{2 \nu}}}{\nu^{r+k}},
$$
where $A=A(\psi_0)$ is a suitable constant depending on $C_1$, $C_2$ and either
$\sin(\psi_0)$ or $\cos(\psi_0)$.

The case $r$ even can be handled similarly but, in this case, the dominant
harmonic is the second order one. One obtains
$$
M_1(\psi_0) \approx \tilde{A} \frac{e^{-\frac{\pi}{\nu}}}{\nu^{r+k}},
$$
for a suitable constant $\tilde{A}=\tilde{A}(\psi_0)$. Note that for a fixed value of $\nu$ the splitting
size is expected to be much smaller when $r$ is even than when it is odd.

The coefficient $A$ and $\tilde{A}$ above follow for a simple monomial but in the case of a polynomial
one can have cancellations.

\begin{remark}
\small
If one uses $F_2=\Gamma_2 - \Gamma_3 + \Gamma_3^2$ instead of $F_1=\Gamma_1$ to measure the splitting then
one obtains
$$
M_2(\psi_0) = \mathcal{O}\left(\frac{1}{\nu} M_1(\psi_0)\right) =  \mathcal{O}\left(\frac{A}{\nu^{r+k+1}} \ e^{- \frac{ \pi}{ 2 \nu }} \right).
$$
In general the prefactor depends on the first integral we use.
\end{remark}

{\bf Example.} If $H(\bx,\by)=H_0(\bx,\by)+\epsilon y_1^5$ (i.e. $r=5$, $k=0$, $l=5$, $m=5$, $n=0$)
then
$$
M(\psi_0)= \int_{-\infty}^{+\infty} 5 \ \cos^4 \psi \  \sin \psi \ R_2(t)^5 \, dt = \sum_{j=0}^{5} A_{j,0,5}(\psi_0) I_{0,5}(j), 
$$
where
$$
I_{0,5}(j)= \int_{-\infty}^{+\infty} \ \sin^{j}t \ \cos^{5-j} t \ \frac{1}{\cosh^{5} (\nu t)} \, dt ,
$$
and  $A_{0,0,5} = s_0^4 c_0$,
$$
A_{j,0,5}(\psi_0) =  (-1)^{j} 20 \sqrt{2}  \left[ \binom{4}{j} s_0^{j-1} c_0^{j+1} + \binom{4}{j-1} s_0^{5-j} c_0^j \right], \quad 1 \leq j \leq 4,
$$
and  $A_{5,0,5} = c_0^4 s_0,$ being $s_0=\sin(\psi_0)$ and $c_0= \cos(\psi_0)$.
It is easier to look directly for the contribution of the first harmonic, we know that it is non-zero because $r=5$ is odd.
One has $\cos^4  \psi
\sin \psi = (\sin \psi)/8 + {\tt hoh}$, where ${\tt hoh}$ denotes the terms with
higher order harmonics. Writting $\sin(\psi)=\sin t \cos \psi_0 + \cos t \sin \psi_0$, one
sees that only the term in $\cos t \sin \psi_0 \cosh^{-5} (\nu t)$ contributes to the Melnikov integral, which is reduced
to
\begin{equation} \label{formexautonom}
 \frac{5}{32 \sqrt{2}}  \sin(\psi_0) \int_{-\infty}^{+\infty} \frac{\cos t}{\cosh^5 (\nu
t)} dt =  \frac{5 \pi \sin \psi_0}{384 \sqrt{2}} \ \frac{1}{ \nu^5} \ e^{-\frac{\pi}{2 \nu}} \ (1+ O(\nu^2)).
\end{equation}

We check the results by a direct numerical computation of the splitting between $W^u$ and
$W^s$ in terms of $F_1=\Gamma_1$. We consider $\epsilon=10^{-4}$.  The
invariant manifolds of the unperturbed system ($\epsilon=0$) intersect the
Poincar\'e section $\Sigma=\max \{y_1^2 + y_2^2\}$ in the curve $x_1=0$, $y_1^2
+ y_2^2=2$. These manifolds are contained in the level set $F_1^{-1}(0)$ (and
$F_2^{-1}(0)$). For $\epsilon \neq 0$, this is no longer true because of the
splitting. For different values of $\nu$, we propagate $N=100$ initial
conditions on the linear approximation at the origin of the invariant manifold
$W^{u}$ (using quadruple precision) up to $\Sigma$ and we evaluate $\Gamma_1$
on the reached points. We repeat the process for $N=100$ points on $W^s$.
Concretely, we select the initial conditions as follows: given a mesh of
equidistributed angles $\psi_0 \in [0,2\pi)$ and fixed $R_2=10^{-12}$, we set
$R_1=R_2(1-R_2^2/2)$ and we take $y_1^u=R_2 \cos(\psi_0), \ y_2^u=R_2
\sin(\psi_0), \ x_1^u=R_1 \cos(\psi_0),  \ x_2^u=R_1 \sin(\psi_0)$ as initial
condition on $W^u$, and $y_1^s=y_1^u, \ y_2^s=y_2^u, \ x_1^s=-x_1^u, \ x_2^s=
-x_2^u$, as initial condition on $W^s$.

Let $\theta=\arctan(y_2/y_1)$. To look for the behaviour of the splitting with
respect to $\nu$, we fit $\Gamma_1(W^u)$ (resp. $\Gamma_1(W^s)$), evaluated on
the invariant manifolds obtained by numerical integration in $\Sigma$,  by a function of the form
$\Gamma_1(\theta)= \sum_{k=1}^{6} a_{k}^u (\nu) e^{i k \theta}$ (resp. with
$a_k^s$ for $W^s$). Then we compute the splitting as the real part of the difference $\Delta
\Gamma=\Gamma_1(W^u)-\Gamma_1(W^s) = \sum_{k=1}^{6} a_{k} (\nu) e^{i k
\theta}$, where $a_k=a_k^u -a_k^s$, to obtain the amplitude of the first
harmonics of the splitting in terms of $\Gamma_1$. The results are shown in
Fig.~\ref{autonomosplit}.  A numerical fit shows quite a good agreement with
the predicted formula (\ref{formexautonom}).

\begin{figure} \begin{center}
\epsfig{file=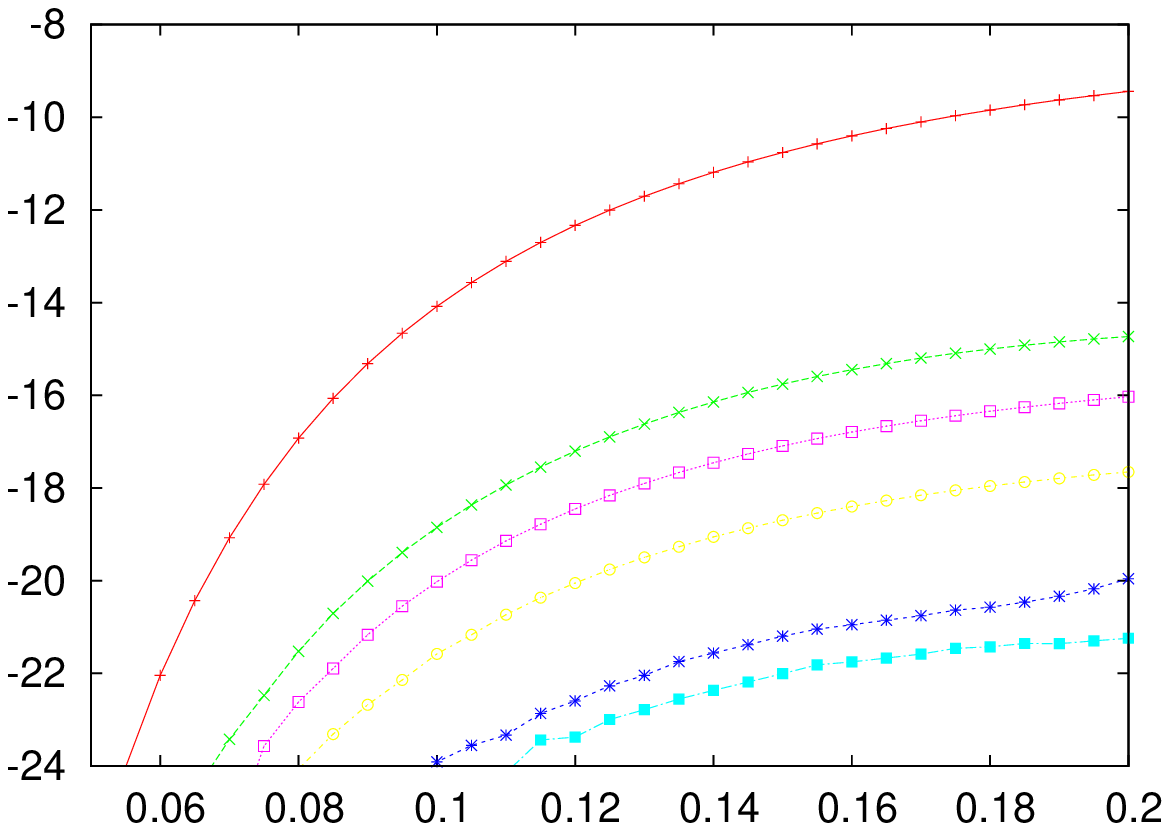,width=75mm,angle=0}
\epsfig{file=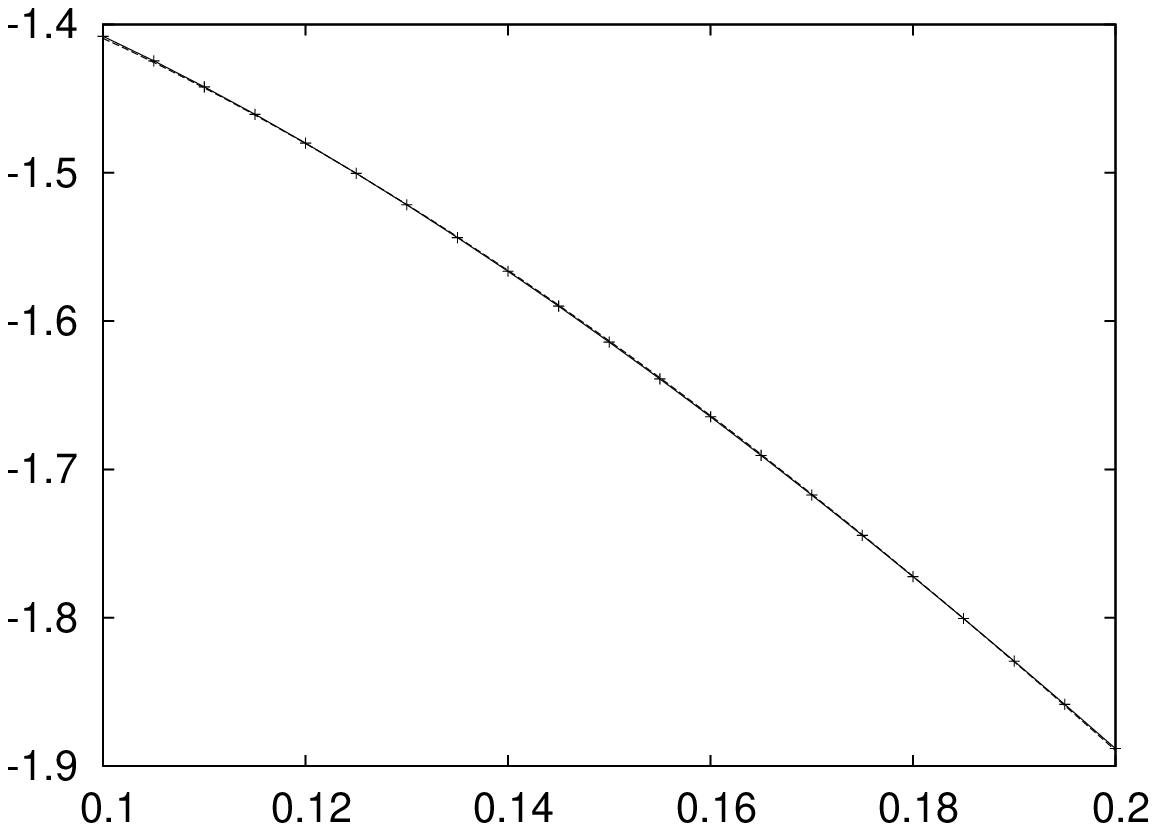,width=75mm,angle=0}
\caption{Left: We represent the amplitude $A_i$, $i=1,...,6$, of the main 6
harmonics of the difference $\Delta \Gamma=\Gamma_1(W^u)- \Gamma_1(W^s)$ with
respect to $\nu$. For $\nu=0.2$ the lines correspond, from top to bottom, to the
$i$-th harmonic ordered as follows: $i=1,2,4,6,3,5$. Right: A numerical fit of
$\nu \log(A_1)$ as a function of $\nu$ by a function $g(x)=a x+b x \log(x)+ c + d x^2
$. One obtains $c\approx -1.58898$ and $b\approx -5.51$ if fitting in the
interval shown. If one fixes $c=\pi/2$ then $b=-5.256$.  }
\label{autonomosplit}
\end{center}
\end{figure}

\section{An example: a periodic perturbation of the Duffing equation} \label{exempleduffing}

We consider a simple example to illustrate that the sum of all terms in a
series expansion can change the dominant exponent in the Melnikov function.

Our model is given by $H(x,y,t)=H_0(x,y)+\epsilon H_1(x,y,t)$ where
$$
H_0(x,y)=\frac{y^2}{2}-\frac{x^2}{2}+\frac{x^4}{4},
$$
is the so-called Duffing Hamiltonian, and
$$
H_1(x,y,t)=\frac{x^2}{d-x} \sin(\omega t) = f(x) \sin(\omega t),
$$
with $d > \sqrt{2}$. The unperturbed system $H_0(x,y)$ possesses a homoclinic
orbit $\varphi(t)=(x(t),y(t))=(\sqrt{2}/ \cosh(t),-\sqrt{2} \tanh(t) / \cosh(t))$.

This is a regular perturbation problem and hence the splitting of the invariant
manifolds for $\epsilon>0$ is expected to be $\mathcal{O}(\epsilon)$.
We consider the Melnikov function $M(\alpha)$ to compute the first order in
$\epsilon$ of this splitting. It is given by
\begin{align*}
M(\alpha)=& \int_{-\infty}^{+\infty} \{ H_0,H_1\}(\varphi(t),t+\alpha) \, dt= -\int_{-\infty}^{+\infty} y \frac{d}{dx}\left( \frac{x^2}{d-x} \right) \sin(\omega t) (\varphi(t),t+\alpha) \, dt \\ 
         =& \omega \cos(\omega \alpha) I_c - \omega \sin( \omega \alpha) I_s,
\end{align*}
where
$$
I_c =  \int_{-\infty}^{+\infty} \frac{x^2(t)}{d-x(t)} \cos(\omega t)\, dt, \qquad  I_s =  \int_{-\infty}^{+\infty} \frac{x^2(t)}{d-x(t)} \sin(\omega t)\, dt.
$$
We have that $I_s \equiv 0$ because $x(t)$ is an even function of $t$.

Note that $f(x(t))$ is a $2\pi \ii$-periodic function in $t$. It has
singularities at $\ii (\pi/2 + k \pi)$, $k\in \mathbb{Z}$, because it contains 
$\cosh(t)$. Also it has singularities at $(\pm s_0 + 2 k \pi) \ii$, $k \in
\mathbb{Z}$, where $s_0=\arccos(\sqrt{2}/d) \in (0,\pi/2)$, due to the
denominator. We just note that the singularity $\ii s_0 $ is closer to the real
axis of time than $\ii \pi/2$.

Let us proceed in two different ways to evaluate $I_c$.

{\bf Evaluation by residues.} One has
$$
I_c= \int_{-\infty}^{\infty} \frac{x^2(t)}{d-x(t)} e^{ \ii \omega t} dt.
$$
We integrate along the boundary of the rectangle of vertices $-R$, $R$, $R+2 \pi \ii$, $-R+2 \pi
\ii$, with $R>0$ and then we take the limit $R \rightarrow \infty$. Inside there
are four singularities of the integrand $F(t)=e^{\ii \omega t} x^2/(d-x)$: $\ii
s_0$, $\ii \pi/2 $, $\ii 3 \pi /2$ and $\ii ( 2 \pi - s_0)$. One gets
$$
I_c = 2\pi \ii \frac{1}{1- e^{-2 \pi \omega}} \sum \text{Res}(F,*),
$$
where $\sum \text{Res}(F,*)$ denotes the sum of the residues at the four singularities. The residues are
\begin{align*}
\text{Res}(F,\ii s_0) & = -\ii \frac{\sqrt{2} d}{\sqrt{d^2-2}} e^{-\omega s_0}, 
&\text{Res}(F,\ii (\pi- s_0)) &  = \ii \frac{\sqrt{2} d}{\sqrt{d^2-2}} e^{\omega (s_0 -2 \pi)},\\
\text{Res}(F,\ii \pi/2) & = \ii \sqrt{2} e^{-\omega \pi/2}, 
& \text{Res}(F,\ii 3 \pi/2)  & = -\ii \sqrt{2} e^{-3 \omega \pi /2 }.
\end{align*}
Note that under a fast frequency perturbation, for example if $\omega \sim
\nu^{-1}$ with $|\nu|$ small, the integral  $I_c$ is $\mathcal{O}(e^{-\omega
s_0})$ asymptotically when $\nu \rightarrow 0$.

{\bf Evaluation by series expansion of $f(x)$.} 
We expand $$f(x) = \frac{x^2}{d-x} = \frac{x^2}{d} \sum_{j \geq 0} \left(
\frac{x}{d} \right)^j,$$ and substitute this expansion into the
Poincar\'e-Melnikov integral. The previous series converges uniformly and the
series of the integrals also is convergent. One gets
\[ I_c=\sum_{j\geq 0}\frac{1}{d^{j+1}}\int_{-\infty}^{\infty}x^{j+2}(t)
\cos(\omega t)dt=\sum_{j\geq 0}\frac{\sqrt{2}^{j+2}}{d^{j+1}} I_{j+2}=
d\sum_{k\geq 2}\left(\frac{\sqrt{2}}{d}\right)^k I_k, \]
where
\[ I_k=\int_{-\infty}^{\infty} \frac{ \cos(\omega t) }{ \cosh^k(t)} dt. \]
Note that with the notation used in (\ref{I1I2}) one has $I_n=I_1(\omega,1,n)$.
Then the following recurrence holds for the integrals $I_n$
\[ I_1=\frac{\pi}{\cosh(\omega\pi/2)},\qquad I_2=\frac{\omega \pi}
{\sinh(\omega\pi/2)},\qquad I_n=\frac{\omega^2+(n-2)^2}{(n-1)(n-2)}I_{n-2},
\ n\geq 3. \]
Each $I_n$ is related to the monomial $x^{n}$ of the series of $f$.

Let us provide an idea of the behaviour of the series for $\omega$ large. A lower
bound can be obtained by neglecting the $(n-2)^2$ terms (compare to $\omega^2$) in the numerator 
of the recurrence of $I_n$. Letting aside constants, one can write
$I_1=\exp(-\omega\pi/2),\, I_2=\omega\exp(-\omega\pi/2)$ and then one has $I_n=
\omega^{n-1}\exp(-\omega\pi/2)$. The sum in the previous expression of $I_c$
becomes (for simplicity we consider the sum starting at $k=0$, which adds
relatively small contributions)
\[ \omega^{-1}\sum_{k\geq 0} \left(\frac{\sqrt{2}\omega}{d}\right)^k
\exp(-\omega\pi/2)=\omega^{-1} \exp(-\omega(\pi/2-\sqrt{2}/d)). \]
As mentioned, this is a lower bound of the sum. If we compare the exponential
part with the one obtained by residues, the relevant multiples of $-\omega$ in
the exponents are $\pi/2-\sqrt{2}/d$ and $s_0=\arccos(\sqrt{2}/d)$. They are quite
close for large $d$, but differ in an important way when $d$ tends to
$\sqrt{2}$.

To see the contribution of the largest term in the sum, let us look at a general
case: given a large value $z$, how the largest term in the series for $\exp(z)$
compares with the sum. Obviously the largest term is the term $n=[z]$, where
$[\phantom{+}]$ denotes the integer part. For simplicity we use $n=z$ and the
relative contribution is $(z^z/\Gamma(z+1))/\exp(z)\approx 1/\sqrt{2\pi z}$,
using Stirling formula.

Summarizing, the largest term in the series can give some idea of the total
contribution to the splitting, but only for large values of $d$.

%

\begin{remark}
This appendix seems to contradict the ideas in this work since, for the system
(\ref{system_intro}) considered, we claim that the dominant term of the series
expansion provides the correct order (that is the correct exponent in the exponentially small part) of the splitting function. The explanation
is a quasi-periodic effect: the contribution of the small divisors related to
the two frequencies to the terms related to the dominant
harmonics of the splitting function is of larger order than the contribution of
the other terms of the expansion. Let us give further details. The dynamics
along the homoclinic orbit of system (\ref{system_intro}) is slow. Scaling
$T=\nu t$ it becomes of order 1, as it is for the example of this appendix. The
singularities closest to the real axis are $\pm \ii \! \arccos(\sqrt{2}/d)$ and
$\pm \ii \! \log(c+\sqrt{c^2-1})$.  The angles playing a role are
$\psi=T/\nu+\psi_0$ and $\theta=\gamma T/\nu+\theta_0$, hence fast angles.  For
$c=5$ and $d=7$, the closest singularities are located at $\pm  \ii \sigma_d$,
$\sigma_d \approx 1.367365$ and at $\pm \ii \sigma_c$, $\sigma_c \approx
2.292431$. Hence under a periodic perturbation one would expect an
exponentially small splitting in $-\sigma_d/\nu$.  However, under a
quasi-periodic perturbation, the dominant harmonic has a combination of the
fast angles $\psi_0$, $\theta_0$. This combination defines an angle with a
frequency of order $1/\sqrt{\nu}$ if $\gamma$ is a quadratic irrational, hence
a fast angle but slower than the angles $\psi_0$ and $\theta_0$. For $\gamma$
with other Diophantine properties the frequency of the fast angle of the
dominant harmonic changes but the situation is analogous, see
Remarks~\ref{remark_k1k2} and \ref{rk7p4}.
\end{remark}

\section{A comment on the regularity of the non-autonomous perturbation} \label{Sec:Regularity}

Here we consider the unperturbed Hamiltonian (\ref{H0}). A general perturbation
$H_1(\bx,\by,\theta)$, $\theta=\gamma t + \theta_0$, $\gamma \in
\mathbb{R}\setminus \mathbb{Q}$, $\theta_0 \in [0,2\pi)$, analytic w.r.t.
$\bx,\by$ in a compact set $K \subset \mathbb{C}^2$ and periodic in $\theta$,
will give rise to two sequences of changes in the dominant harmonic of
$\DFt{1}$ and $\DFt{2}$.  Here, we show that if the perturbation
$H_1(\bx,\by,\theta)$ is of class $\mathcal{C}^{p}$ in $\theta$, for some $1
\leq p<\infty$, then the amplitude of the dominant term of the Melnikov
function $M(\psi_0,\theta_0)$ that measures the splitting between the invariant
manifolds $W^{u/s}({\bf 0})$ is expected to remain constant as $\nu \rightarrow
0$. On the other hand, if $H_1({\bf x},{\bf y},\theta) \in
\mathcal{C}^{\omega}$ then, generically, it is expected to have a decay of the
amplitude of the dominant term of $M(\psi_0,\theta_0)$ as $\nu \rightarrow 0$.
The reason is the following. Consider the Fourier expansion
$$
H_1({\bf x},{\bf y},\theta)=\sum_{k \in \mathbb{Z}} A_k({\bf x},{\bf y}) e^{i k \theta}.
$$
We recall that $H_0=G_1 + \nu G_2$, and that $G_1$ evaluated along the unperturbed separatrix contains a factor periodic in $t$.
On the other hand, $H_1$ is assumed to be periodic in $\theta$. As explained in
Section~\ref{sect_expr_PoincMel}, in the exponential part of the Melnikov
function there appear terms with 
$$
s = j - k \gamma, \qquad j,k \in \mathbb{Z}.
$$
Assume that $\gamma$ is a Diophantine number and that exist constants
$\tilde{c} >0 $ and $\tau \geq 1$ such that
$$
|s|>\frac{\tilde{c}}{|k|^{\tau}}.
$$
The term whose exponent has the smallest $|s|$ becomes the dominant term of $M(\psi_0,\theta_0)$.
This term is related to a best approximant of $\gamma$.
For the best approximants of $\gamma$ one has $|s| = c |k|^{-\tau}$ for
suitable $c>0$ (related to the values $c_{s,m_1/m_2}$ introduced in
Section~\ref{sect_PoincMel}). The exponentially small part of the dominant term
of $M(\psi_0,\theta_0)$ is expected to be of the form 
$$
\exp\left(-\frac{c}{\nu |k|^{\tau}}\right).
$$
We include the effect of the prefactor in front of this exponential part so that
we can get more accurate information on the behaviour of the amplitude of
$M(\psi_0,\theta_0)$. We distinguish two cases according to the regularity of
$H_1$ w.r.t $\theta$.

\begin{enumerate}

\item If $H_1({\bf x},{\bf y},\theta)$ is of class $\mathcal{C}^{p}$ in
$\theta$ then the Fourier coefficients decay as
$$
|A_k({\bf x,\bf y})| = \mathcal{O}(|k|^{-p}).
$$
The dominant term $d(k)$ of $M(\psi_0,\theta_0)$, considering that $\{G_1,H_1\}$ has a similar expression to $H_1$,
is then expected to be of the order
$$
d(k)\approx K \frac{1}{|k|^p} \exp\left(-\frac{c}{\nu |k|^{\tau}}\right),
$$
for some $k\in \mathbb{Z}$, where $K \in \mathbb{R}_+$ is a constant.
The maximum of $d(k)$ is attained for $k=k_M=(c \tau / \nu p)^{1/\tau}$ which implies that
$$
d(k_M) \approx C \nu^{p/\tau}, \qquad \text{ where } \quad C= k(p/c e \tau)^{p/\tau}.
$$
As a conclusion, the dominant term behaves as a power law with respect to $\nu$ and
the exponential part of the amplitude remains constant.

\item If, on the other hand, $H_1({\bf x},{\bf y},\theta)$ is of class
$\mathcal{C}^{\omega}$ in $\theta$ then the Fourier coefficients decay as
$$
|A_k({\bf x},{\bf y})| = \mathcal{O}(e^{-|k|/\rho}),
$$
where $\rho>0$ is the distance to the closest singularity of $H_1({\bf x},{\bf
y},\theta)$ from the real time axis. Assuming again that $\{G_1,H_1\}$ has a
similar expression to $H_1$, the dominant term of $M(\psi_0,\theta_0)$ is
expected to be of the order
$$
d(k) \approx K \exp\left(-\frac{|k|}{\rho}\right) \exp\left(-\frac{c}{\nu |k|^{\tau}}\right).
$$
It follows that the maximum of $d(k)$ is attained for $k=k_M=(c \tau \rho/ \nu
)^{1/(\tau+1)}$ which gives 
$$
d(k_M) \approx K \exp\left(-C/\nu^{1/(\tau+1)} \right), \quad \text{ where } \quad C=
\frac{\tilde{k}^{1 /(\tau +1)}}{\rho} + \frac{c}{\tilde{k}^{\tau /(\tau+1)}}, \quad \tilde{k}=c \tau \rho.
$$
As a conclusion the dominant term amplitude decreases exponentially. See related 
comments in Remark~\ref{remark_k1k2} and Remark~\ref{rk7p4}.
\end{enumerate}

\section{Splitting functions, splitting volume and diffusion properties} \label{split-difusion}

As shown in Table~\ref{taulanubif} for large ranges of $\nu$ the dominant
harmonics of both splitting functions coincide. This has some dynamical consequences.
Note that this situation can happen in many other systems, and also in
situations not necessarily related to the Hamiltonian-Hopf scenario.

In our framework the unperturbed system has a 2-dimensional homoclinic surface
given by the invariant manifolds of the origin that coincide ($W^u =W^s$ for
the unperturbed system) and which is given by $G_1=G_2=0$. The invariant
manifolds $W^u$ and $W^s$ split for $\nu>0$.  Let $\DFt{1}$ and $\DFt{2}$
the splitting functions, measuring the displacement of $W^u(0)$ and $W^s(0)$
with respect to the unperturbed manifolds. They can be expressed as a Fourier
series with combinations of two angles $\psi_0,\theta_0 \in [0,2\pi)$ of the form
$\sin(m_1 \psi_0 - m_2 \theta_0)$. For $\psi_0=\theta_0=0$ one has a homoclinic
point. We look for a basis of $\mathcal{T}_{W^u}(0,0)$ (resp.
$\mathcal{T}_{W^s}(0,0)$), the tangent space to $W^u$ (resp. to $W^s$) at the
homoclinic point, and we define the splitting volume at the homoclinic point to
be the determinant between the four vectors of the basis (suitably normalized
if necessary), see \cite{LocMarSau03}.

Assume that, in the fundamental torus $\mathcal{T}_\Sigma$ where the splitting
between $W^u$ and $W^s$ is measured, $W^u$ is represented as a graph $g_u:
\mathbb{R}^2 \rightarrow \mathbb{R}^4$, where
$g_u(\psi_0,\theta_0)=(\psi_0,\theta_0, F_1^u(\psi_0,\theta_0),
F_2^u(\psi_0,\theta_0))$.  Similarly, we consider
$g_s(\psi_0,\theta_0)=(\psi_0,\theta_0,F_1^s(\psi_0,\theta_0),F_2^s(\psi_0,\theta_0))$
the graph representation of $W^s$ in the fundamental domain. One has 
$$
\mathcal{T}_{W^u}(0,0) = \langle v_1, v_2 \rangle, \qquad
\mathcal{T}_{W^s}(0,0) = \langle w_1, w_2 \rangle,
$$
where
\begin{align*}
v_1\!=\!\frac{\partial g_u}{\partial \psi_0}(0,0)\!=\!\left. \left(1,0,\frac{\partial F_1^u}{\partial \psi_0},  \frac{\partial  F_2^u}{\partial \psi_0} \right) \right|_{(0,0)}^\top, & 
\ v_2\!=\!\frac{\partial g_u}{\partial \theta_0}(0,0) \!=\! \left. \left(0,1,\frac{\partial F_1^u}{\partial \theta_0}, \frac{\partial F_2^u}{\partial \theta_0} \right)\right|_{(0,0)}^\top,\\
w_1 \!= \!\frac{\partial g_s}{\partial \psi_0}(0,0)\!=\!\left. \left(1,0,\frac{\partial  F_1^s}{\partial \psi_0}, \frac{\partial  F_2^s}{\partial \psi_0} \right) \right|_{(0,0)}^\top, &
\ w_2 \!=\!\frac{\partial g_s}{\partial \theta_0}(0,0)\!=\!\left. \left(0,1,\frac{\partial F_1^s}{\partial \theta_0}, \frac{\partial F_2^s}{\partial \theta_0} \right) \right|_{(0,0)}^\top.
\end{align*}
The splitting volume at the homoclinic point at $(\psi_0,\theta_0)=(0,0)$ is defined as
$$
V=\det(v_1,v_2,w_1,w_2),
$$
and one has
$$
V=a_1 b_2-b_1 a_2,
$$
where
$$
a_i=  \frac{\partial{F_i^u}}{\partial \psi_0}(0,0) - \frac{\partial{F_i^s}}{\partial \psi_0}(0,0), \quad b_i=  \frac{\partial{F_i^u}}{\partial \theta_0}(0,0)-\frac{\partial{F_i^s}}{\partial \theta_0}(0,0), \quad i=1,2.\\ 
$$

The volume $V$ is a quantity related to local diffusive properties: if $V>0$ then
the system generically shows some diffusion. If $V=0$ there is no
possibility of having ``first order'' diffusion (nothing prevents on having a
much slower diffusion process if the manifolds have non-transversal
intersection or intersect transversally at some other homoclinic). 

Note that if $\DFt{1}(\psi_0,\theta_0) = c \DFt{2}(\psi_0,\theta_0)$ for some $c \in \mathbb{R} \setminus
\{0\}$, $(\psi_0,\theta_0)\in \mathcal{T}_\Sigma$, then $a_1=c a_2$ and $b_1=c b_2$, and $V=0$.
For example, as an illustration, if one has 
$$
\DFt{1}= A \sin(m_1 \psi_0 - m_2 \theta_0), \text{ and } \DFt{2}= B \sin(m_1 \psi_0 - m_2 \theta_0),
$$
for $m_1,m_2$  related to an approximant of $\gamma= (\sqrt{5}-1)/2$ then the
splitting volume is 0, and no ``first order'' diffusion is expected.
The behaviour of $V$ as a function of $\nu$ is displayed in Fig.~\ref{volum}.
We observe on the right plot that for $\nu=\nu_1$, where $\log_2{\nu_1}\approx -8.391$,
it seems $V=0$. The same happens for $\nu=\nu_3$, $\log_2{\nu_3} \approx -11.202$.
In fact, for $\nu_1$ and $\nu_3$ there is a change of sign of the determinant $V$. 
These values of $\nu$ are close to local maxima of the values of $\DFt{i}$,
see Figs.~\ref{nterms} and \ref{ampli2}. 
For values of $\nu$ in an interval around $\nu_1$ (respectively, around
$\nu_3$) the dominant harmonic for both $\DFt{1}$ and $\DFt{2}$
corresponds to $(m_1,m_2)=(5,8)$ (respectively, to $(m_1,m_2)=(13,21)$).  The
nearby harmonics are much smaller and, hence, the splitting functions $\DFt{i}$
are close to be proportional. Between $\nu_1$ and $\nu_3$ there is a
range of $\nu$ values for which the dominant harmonic is $(m_1,m_2)=(8,13)$.
This harmonic has the maximum contribution to $\DFt{i}$ for $\nu=\nu_2$,
$\log_2(\nu_2)\approx -9.85$, where $V$ has a minimum, but the signs of the
dominant harmonic and the nearby ones play a role and $V\neq 0$.  We postpone
further investigations on the relation between the behaviour of $V$ and the
relative position of the invariant manifolds in different ranges of the small
parameter for future works.

\begin{figure}[ht]
\begin{center}
\epsfig{file=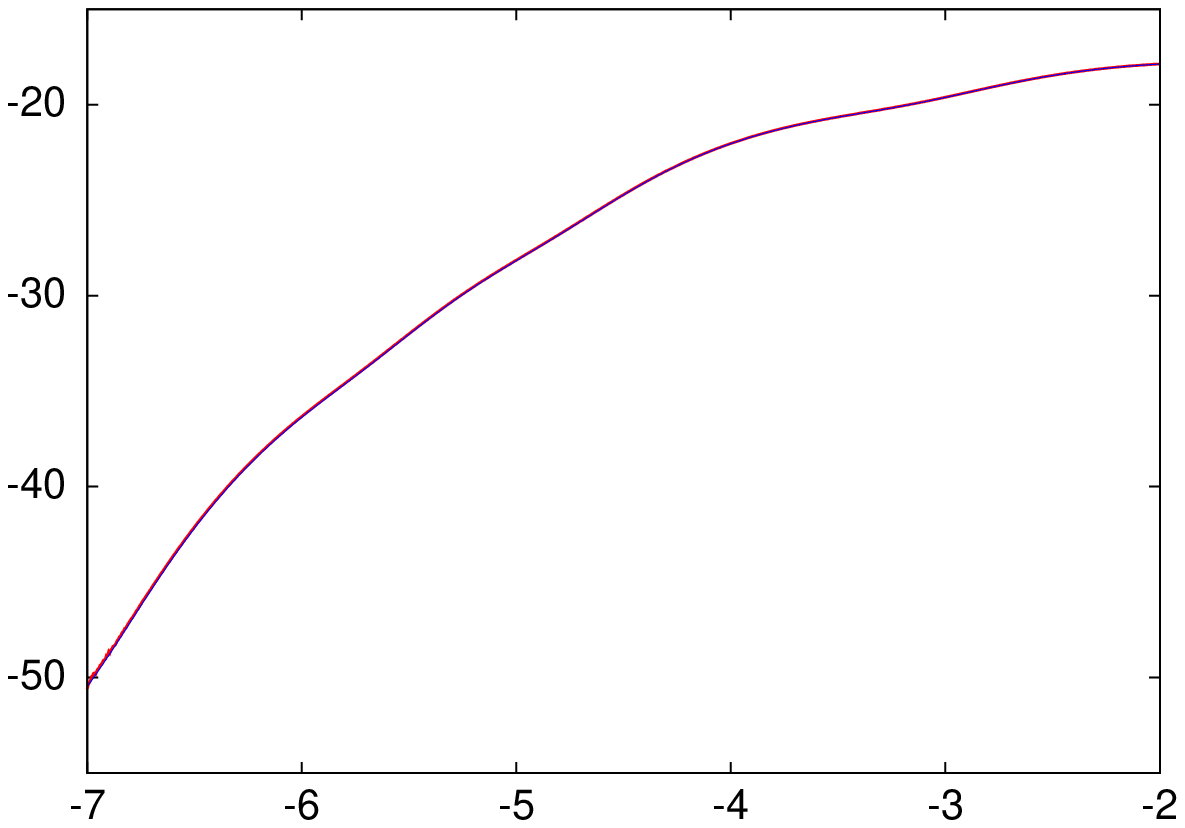,width=7cm} 
\epsfig{file=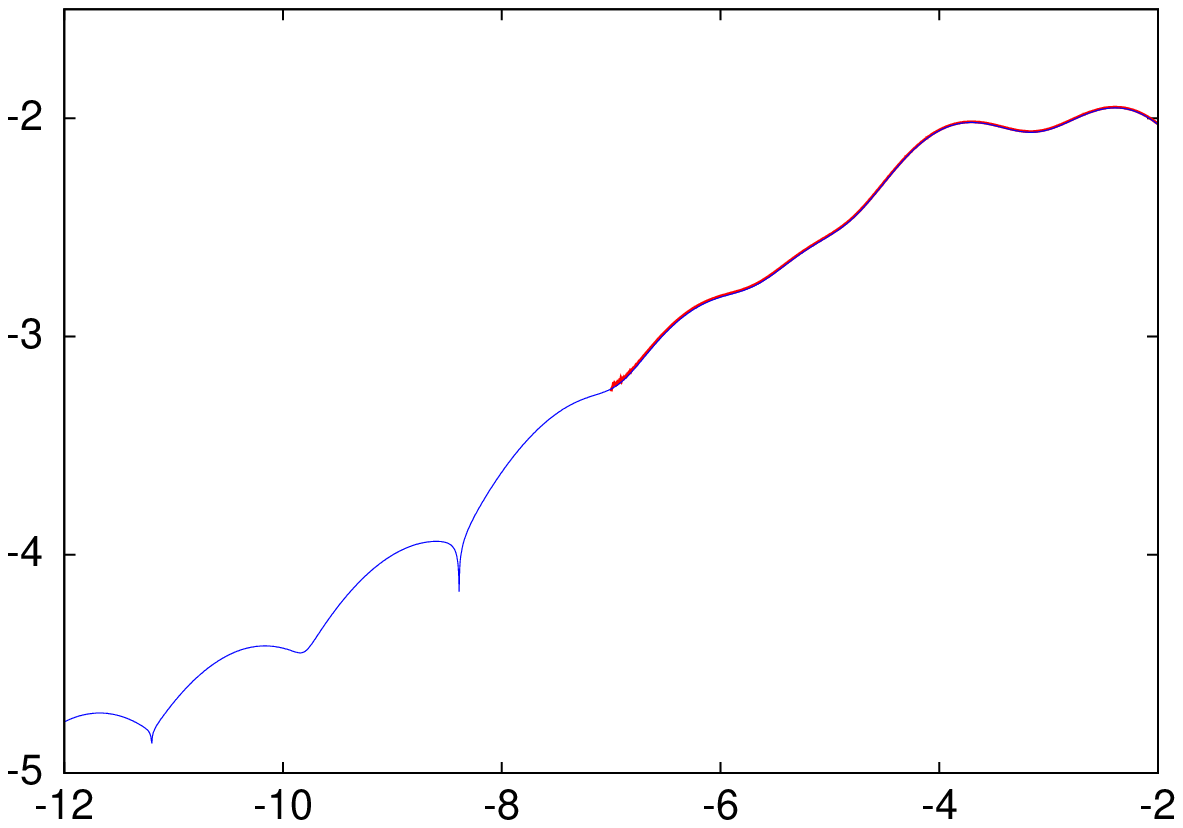,width=7cm}
\end{center}
\caption{We consider $\epsilon=10^{-3}$ and $\gamma=(\sqrt{5}-1)/2$. Left: We display $\log(|V|)$ as a
function of $\log_2(\nu)$. The red points display $V$ computed from a direct computation
via a numerical computation of the invariant manifolds. The blue
points is an estimate of $V$ based on the dominant terms of the Poincar\'e-Melnikov approximation.
The blue points almost coincide with the red ones and, hence, red points are almost hidden.
Right: We plot $\log(|V|/\epsilon^2)\sqrt{\nu}$ for $\nu<2^{-7}$, where $V$
is estimated from  the dominant terms Poincar\'e-Melnikov approximation.}
\label{volum}
\end{figure}

\section{On the visible and hidden harmonics of the splitting function related to best approximants}\label{conseq_best}

In Section~\ref{Sec:Otherfreq} we have shown examples where the approximant
associated to one of the best approximants of $\gamma$ is hidden. By hidden harmonic
we refer to an harmonic related to a best approximant which never becomes
dominant: for all values of $\nu$ there is (are) other harmonic(s) which give a
larger contribution to the splitting function. Nevertheless we do not have
found examples where two (or several) successive harmonics related to best
approximants of $\gamma$ are hidden. 

In this appendix we explain why some best approximants give a smaller
contribution to the splitting, but they are dominant in some narrow range (see
Fig.~\ref{em2} for the case $\gamma=e-2$ and, similarly, the top right plot
in Fig.~\ref{othergamma} for other $\gamma$) while others are not dominant (as seen in the bottom
left plot in Fig.~\ref{othergamma}). We emphasize that this is not related to
the concrete Diophantine properties of $\gamma$ but to the different behaviour
of the error of the approximations which, as explained in
Section~\ref{Sec:periodicity_csn}, is due to the successive quotients in
the CFE of $\gamma$.

Of course, it is easy to produce examples in which several successive best
approximants are not dominant. Simply consider that the harmonics related to
these approximants have an extremely small amplitude.  Here we will consider
the case when the amplitudes decrease uniformly in an exponential way and we
will show that it is not possible to have two consecutive non-dominant best
approximants if some nearby quotients are large. See \cite{FonSimVie-2} for more general cases. 

Assume we have a perturbation depending on a spatial component $x$ and on time and that
the variable $x$ rotates with constant angular velocity, say 1, while the
temporal part has frequency $\gamma$. We also assume that the amplitudes of the
harmonics, due to $x$ and to $t$ decrease exactly in an exponential way.
Let $N_n/D_n$ be a best approximant to $\gamma$ and assume, similarly
to what was presented in Remark~\ref{remark_k1k2} in
Section~\ref{sect_expr_PoincMel}, that the contribution to the
Poincar\'e-Melnikov integral is of the form
\begin{equation*} 
\epsilon A\nu^B\exp(-N_n\rho_1-D_n\rho_2)\exp\left(\frac{-C s}{\nu}\right),
\end{equation*}
where $s=|N_n-\gamma D_n|$ and that in the problem at hand $C=\pi/2$. As we are
interested in what happens for fixed $\epsilon$ and a given value of $\nu$ the
factor $\epsilon A\nu^B$ is irrelevant. As shown in
Section~\ref{Sec:periodicity_csn}, $s$ can be expressed as $(N_n c_{s,n})^{-1}$,
where $c_{s,n}$ is given in (\ref{csnqpm}).

To find the maximal contribution is equivalent to find the minimum of minus the
exponent. Furthermore one has $D_n=(N_n\mp s)/\gamma$ where the sign $\mp$
depends on the sign of $s$. Hence
\begin{equation*} 
N_n\rho_1+D_n\rho_2=N_n(\rho_1+\rho_2/\gamma)(1+\cO(N_n^{-2}))
\end{equation*}
and we also recall that
\begin{equation*} 
c_{s,n}=((q_{+,n}+1/q_{-,n})/\gamma)(1+\cO(N_n^{-2})),
\end{equation*}
where $q_{+,n}$ and $q_{-,n}$ are expressed in terms of quotients of the CFE.
Assuming we are dealing with large values of $N_n$ the relative contributions
of order $N_n^{-2}$ will be neglected for simplicity. Furthermore one can
introduce $\rho=\rho_1+\rho_2/\gamma$ to have the following simpler expression
for minus the exponent
\begin{equation*} 
\rho N_n+\frac {C \gamma}{\nu N_n (q_{+,n}+1/q_{-,n})}.
\end{equation*}
Dividing by $\rho$, scaling $\nu$ as $\hat{\nu}=\nu \rho/(C \gamma)$ and denoting
$q_{+,n}+1/q_{-,n}$ by $\hat{c}_{s,n}$, the functions to be studied are of the
form
\begin{equation*} 
T_n(\hat{\nu})= N_n+\frac{1}{\hat{\nu}N_n\hat{c}_{s,n}}.
\end{equation*}

In what follows we rename $\hat{\nu}$ and $\hat{c}_{s,n}$ as $\nu$ and
$c_{s,n}$, for simplicity. Furthermore, looking at Figs.~\ref{ampli2},
\ref{em2} and \ref{othergamma}, we shall concentrate, as we said, on values of
$N_n$ which are large (i.e., $\nu$ small) to see, under the current
assumptions, which are the best approximants that are not giving the dominant
contribution in the splitting. We also note that if $N_n<N_{n+1}$ then
$T_n(\nu)<T_{n+1}(\nu)$ for $\nu$ large and $T_n(\nu)> T_{n+1}(\nu)$ for $\nu$
small. This requires to check that the sequence $\{1/N_nc_{s,n}\}_n$ is
decreasing. It is an easy check and we refer to \cite{FonSimVie-2} for the
details. Hence, $T_n$ and $T_{n+1}$ coincide only at one point. This is a transversality property.

Assume $\gamma$ is given. We fix our attention to the numerators $N_{n-1},N_n$
and to the quotients $q_{n+j}\;,j=1,5$. To simplify the notation we shift the
subindices by $n$. For completeness we also take into account the value
$\alpha_+=1/[q_6;q_7,\ldots]$. We assume that $q_1,q_3,q_4$ are relatively
small and that $q_2,q_5$ very large, say of the order of a large number $Q$,
which hereafter will be considered as a parameter. Looking only at the dominant
contributions (i.e., neglecting terms of relative size of $\cO(Q^{-1})$) we have
the numerators
\begin{equation*} 
N_1,\quad N_2=q_2N_1,\quad N_3=q_3q_2N_1,\quad N_4=q_2N_1(1+q_4q_3).
\end{equation*}
The relevant values of $c_{s,j}$ are
\begin{equation*} 
c_{s,1}=q_2,\quad c_{s,2}=q_3+1/q_4,\quad c_{s,3}=q_4+1/q_3,\quad c_{s,4}=q_5.
\end{equation*}
From this we can compute the values of $T_j,j=1,4$, skipping terms relatively
$\cO(Q^{-1})$:
\begin{equation*} 
T_1(\nu)=N_1+\frac{1}{\nu N_1q_2},\quad T_2(\nu)=q_2N_1+\frac{q_4}
{\nu q_2N_1(1+q_3q_4)},
\quad T_3(\nu)=q_3q_2N_1+\frac{1}{\nu q_2N_1(1+q_3q_4)},
\end{equation*}
\[ T_4(\nu)=q_2N_1(1+q_4q_3)+\frac{1}{\nu q_2 N_1(1+q_3q_4)q_5}.\]

Now we look for a value of $\nu$, say $\nu^*$, for which one has $T_1=T_4$. One has $\nu^*=(N_1^2
q_2^2(1+q_3q_4))^{-1}$ which allows to compute the value $T_1(\nu^*)=T_4(\nu^*)$. But,
as done in the different figures, it is suitable to multiply $T_i$ by $\sqrt{\nu^*}$.
In this way one obtains plots similar to the ones in the figures, just changing
the sign. Let us denote $\sqrt{\nu^*}T_j(\nu^*)$ by $\hat{T}_j(\nu^*)$ and then 
\begin{equation*} 
\hat{T}_1(\nu^*)=\sqrt{1+q_3q_4}=\hat{T}_4(\nu^*).
\end{equation*}
At that value of $\nu^*$ there is a change:  $T_1<T_4$ for $\nu > \nu^*$ while
$T_1>T_4$ otherwise. But, what are $T_2$ and $T_3$ doing? A simple computation gives
\begin{equation*} 
\hat{T}_2(\nu^*)=\frac{1+q_4}{\sqrt{1+q_3q_4}},\quad
  \hat{T}_3(\nu^*)=\frac{1+q_3}{\sqrt{1+q_3q_4}}.
\end{equation*}
Therefore, if $q_3<q_4$ one has $\hat{T}_3(\nu^*)<\hat{T}_2(\nu^*)\leq
\hat{T}_1(\nu^*)$, and the last inequality becomes an equality if, and only if,
$q_3=1$. This, together with the transversality property mentioned above,
implies that the best approximant associated to $N_3$ is dominant.  In a
similar way, if $q_4<q_3$ one has $\hat{T}_2(\nu^*)<\hat{T}_3(\nu^*)\leq
\hat{T}_1(\nu^*)$ with equality between the last two if, and only if, $q_4=1$.

Summarizing, under the current conditions, if either $q_3$ or $q_4$ are
different from 1, then at least one of them is the dominant term in a range
between the dominance of the best approximant associated to $N_1$ and the one
associated to $N_4$.

It remains to investigate what happens if $q_3=q_4=1$. In that case the approximation
neglecting relative terms $\cO(Q^{-1})$ gives that $T_1(\nu),\;T_2(\nu),\;
T_3(\nu)$ and $T_4(\nu)$ intersect at the same point. Numerical evidence is
that  $T_2(\nu)$ or $T_3(\nu)$ or both of them dominate in some narrow range.
Which one dominates or in which cases both of them dominate in suitable ranges
depends, mainly, on the ratio $q_2/q_5$. Let us prove the theoretical facts
confirming this numerical evidence.

Hence we take $q_3=q_4=1$ and recompute $N_j$ and $c_{s,j},j=1,\ldots,4$, taking into
account terms whose relative value is $\cO(Q^{-1})$. We note that $N_2=q_2N_1+N_0=
N_1(q_2+N_0/N_1)$. For simplicity we denote $N_0/N_1$ by $\alpha$. One easily
obtains
\begin{equation*} 
N_2=N_1(q_2+\alpha),\quad N_3=N_1(q_2+\alpha+1),\quad N_4=N_1(2q_2+2\alpha+1).
\end{equation*}
In a similar way we obtain the following values for $c_{s,j}$, 
\begin{equation*} 
c_{s,1}=q_2+\frac{1}{2}+\alpha,\quad c_{s,2}=2-\frac{1}{q_5}+\frac{1}{q_2},
\quad c_{s,3}=2-\frac{1}{q_2}+\frac{1}{q_5},\quad c_{s,4}=q_5+\alpha_+ +
\frac{1}{2},
\end{equation*}
where we recall that $\alpha_+=1/[q_6;q_7,\ldots]$.

Let us denote by $\nu_{i,j}$ the value of $\nu$ for which $T_i=T_j$. One has
\begin{equation*}
\nu_{i,j}=\frac{(N_i c_{s,i})^{-1}-(N_j c_{s,j})^{-1}}{N_j-N_i}.
\end{equation*}
It is immediate to obtain expressions for $\nu_{i,j},\;1\leq i<j\leq 4$, but for
shortness we display the values obtained after shifting and scaling:
$\nu_{i,j}^*=(\nu_{i,j}N_1^2 q_2^2-1/2)q_2$. They are
\begin{equation*} 
\nu_{1,2}^*=\frac{1}{4}-\frac{q_2}{4q_5}-\alpha,\quad
   \nu_{1,3}^*=-\frac{1}{4}+\frac{q_2}{4q_5}-\alpha,\quad
   \nu_{1,4}^*=-\frac{1}{4}-\frac{q_2}{4q_5}-\alpha,
\end{equation*}
\[ \nu_{2,3}^*=\left(\frac{q_2}{2q_5}-\frac{1}{2}\right)q_2,\quad
   \nu_{2,4}^*=-\frac{3}{4}-\frac{q_2}{4q_5}-\alpha,\quad
   \nu_{3,4}^*=-\frac{1}{4}-\frac{3q_2}{4q_5}-\alpha.\]
We note that to have a more precise value for $\nu_{2,3}^*$ it would be
necessary to carry out some expansion with relative order $\cO(Q^{-2})$,
specially if $q_2$ and $q_5$ are close. But it is not necessary for our
purposes.

From the previous expressions one has
\begin{equation*} 
\max\{\nu_{2,4}^*,\nu_{3,4}^*\}<\nu_{1,4}^*<\min\{\nu_{1,2}^*,\nu_{1,3}^*\},
\end{equation*}
showing that when $\nu$ decreases, either $T_2(\nu)$ or $T_3(\nu)$ intersect
$T_1(\nu)$ before the intersection of $T_1(\nu)$ and $T_4(\nu)$ and, on the
other side, either $T_2(\nu)$ or $T_3(\nu)$ intersect $T_4(\nu)$ after the
intersection of $T_1(\nu)$ and $T_4(\nu)$.

Hence, either the harmonic associated to $N_2$ or the one associated to $N_3$
(or both of them in some small ranges) dominate the splitting between the
dominances of the harmonics associated to $N_1$ and to $N_4$. The numerical
evidence in many examples is that both of them can be seen only if 
$|q_2-q_5| \leq 1$.

If $q_2<q_5$ but they are not so close, then $\nu_{1,2}^*>\nu_{1,3}^*$ and
$\nu_{2,4}^*<\nu_{3,4}^*$, showing that between the dominances related to $N_1$
and to $N_4$ there is a range of dominance of $N_2$. The roles of $N_2$ and
$N_3$ are exchanged if $q_5<q_2$ but they are not so close.

\begin{remark}
We have considered some transcendental frequencies like $\pi$ and
$\exp({\sqrt{5}-1}/2)$. From the $10^7$ first quotients of the CFE we observe
that it is reasonable to accept them as ``typical'' irrational numbers: the
geometric mean of the quotients of the CFE tends to the Khinchin constant (see,
e.g.,  \url{https://oeis.org/A002210} for numerical values) and the ratio of
increase of the denominators, measured as $\lim_{n\to\infty}\log(D_n)/n$ tends
to the L\'evy's constant $\pi^2/(12\log(2))$. In all the cases the hidden
harmonics that have been detected (around 27.9\% of the total number of
approximants) are isolated. Further details can be found in \cite{FonSimVie-2}.
\end{remark}

\section*{Acknowledgments}
This work has been supported by grants MTM2016-80117-P (Spain) and
2017-SGR-1374 (Catalonia). We also thank the MINECO grant MDM-2014-0445 (Spain).
We are specially indebted to A. Delshams, M.  Gonchenko and V. Gelfreich for
several discussions on related topics. We also thank J. Timoneda for
maintaining the computing facilities of the Dynamical Systems Group of the
Universitat de Barcelona, that have been largely used in this work.

\addcontentsline{toc}{section}{References}
\bibliographystyle{plain}
\bibliography{fsvcs}

\begin{thebibliography}{10}

\bibitem{DelGelJorSea97}
A.~Delshams, V.~Gelfreich, A.~Jorba, and T.M. Seara.
\newblock Exponentially small splitting of separatrices under fast
  quasiperiodic forcing.
\newblock {\em Comm. Math. Phys.}, 189(1):35--71, 1997.

\bibitem{DelGonGut14-3}
A.~Delshams, M.~Gonchenko, and P.~Guti{\'e}rrez.
\newblock Continuation of the exponentially small transversality for the
  splitting of separatrices to a whiskered torus with silver ratio.
\newblock {\em Regul. Chaotic Dyn.}, 19(6):663--680, 2014.

\bibitem{DelGonGut14}
A.~Delshams, M.~Gonchenko, and P.~Guti{\'e}rrez.
\newblock Exponentially small asymptotic estimates for the splitting of
  separatrices to whiskered tori with quadratic and cubic frequencies.
\newblock {\em Electron. Res. Announc. Math. Sci.}, 21:41--61, 2014.

\bibitem{DelGonGut14-2}
A.~Delshams, M.~Gonchenko, and P.~Guti{\'e}rrez.
\newblock Exponentially small lower bounds for the splitting of separatrices to
  whiskered tori with frequencies of constant type.
\newblock {\em Internat. J. Bifur. Chaos Appl. Sci. Engrg.}, 24(8):1440011, 12,
  2014.

\bibitem{DelGut05}
A.~Delshams and P.~Guti\'errez.
\newblock Exponentially {S}mall {S}plitting of {S}eparatrices for {W}hiskered
  {T}ori in {H}amiltonian {S}ystems.
\newblock {\em Journal of Mathematical Sciences}, 128(2):2726 --2745, 2005.

\bibitem{DelRam99}
A.~Delshams and R.~Ram\'{\i}rez-Ros.
\newblock Singular separatrix splitting and {M}elnikov method: An experimental
  study.
\newblock {\em Experiment. Math.}, 8(1):29--48, 1999.

\bibitem{VdM82}
J.C.~Van der Meer.
\newblock Nonsemisimple {$1:1$} resonance at an equilibrium.
\newblock {\em Celestial Mech.}, 27(2):131--149, 1982.

\bibitem{Elpetal87}
C.~Elphick, E.~Tirapegui, M.E. Brachet, P.~Coullet, and G.~Iooss.
\newblock A simple global characterization for normal forms of singular vector
  fields.
\newblock {\em Physica D}, 29(1-2):95--127, 1987.

\bibitem{FonSimVie-2}
E.~Fontich, C.~Sim\'o, and A.~Vieiro.
\newblock On the ``hidden'' harmonics associated to best approximants due to
  quasi-periodicity, 2018.
\newblock Preprint.

\bibitem{GaiGel11}
J.P. Gaiv{\~a}o and V.~Gelfreich.
\newblock Splitting of separatrices for the {H}amiltonian-{H}opf bifurcation
  with the {S}wift-{H}ohenberg equation as an example.
\newblock {\em Nonlinearity}, 24(3):677--698, 2011.

\bibitem{GelLaz01}
V.G. Gelfreich and V.F. Lazutkin.
\newblock Splitting of separatrices: perturbation theory and exponential
  smallness.
\newblock {\em Russian Math. Surveys}, 56(3):499--558, 2001.

\bibitem{GuaSea12}
M.~Guardia and T.M. Seara.
\newblock Exponentially and non-exponentially small splitting of separatrices
  for the pendulum with a fast meromorphic perturbation.
\newblock {\em Nonlinearity}, 24(5):1367--1412, 2012.

\bibitem{Han07}
H.~Han{\ss}mann.
\newblock {\em Local and semi-local bifurcations in {H}amiltonian dynamical
  systems: Results and examples}.
\newblock Lecture Notes in Mathematics, 1893. Springer-Verlag, 2007.

\bibitem{Khi64}
A.Ya. Khinchin.
\newblock {\em Continued Fractions}.
\newblock The University of Chicago Press, Chicago, Ill.-London, 1964.

\bibitem{LerUma92}
L.M. Lerman and Ya.L. Umanski{\u\i}.
\newblock Classification of four-dimensional integrable hamiltonian systems and
  poisson actions of $\mathbb{R}^2$ in extended neighborhoods of simple
  singular points. {I}.
\newblock {\em Russian Acad. Sci. Sb. Math.}, 77:511--542, 1994.

\bibitem{LocMarSau03}
P.~Lochak, J.-P. Marco, and D.~Sauzin.
\newblock On the splitting of invariant manifolds in multidimensional
  near-integrable {H}amiltonian systems.
\newblock {\em Mem. Amer. Math. Soc.}, 163(775), 2003.

\bibitem{McSMey03}
P.D. McSwiggen and K.R. Meyer.
\newblock The evolution of invariant manifolds in {H}amiltonian-{H}opf
  bifurcations.
\newblock {\em J. Differential Equations}, 189(2):538--555, 2003.

\bibitem{MeyerHall}
K.R. Meyer and G.~Hall.
\newblock {\em Introduction to {H}amiltonian dynamical systems and the
  {$N$}-body problem}.
\newblock Applied Mathematical Sciences, 90. Springer-Verlag, New York, 1992.

\bibitem{MeySch71}
K.R. Meyer and D.S. Schmidt.
\newblock Periodic orbits near {${\cal L}_{4}$} for mass ratios near the
  critical mass ratio of {R}outh.
\newblock {\em Celestial Mech.}, 4:99--109, 1971.

\bibitem{Nei84}
A.~Neishtadt.
\newblock The separation of motions in systems with rapidly rotating phase.
\newblock {\em Prikladnaja Matematika i Mekhanika}, 48:133--139, 1984.

\bibitem{PalYan00}
J.~Palaci\'an and P.~Yanguas.
\newblock Reduction of polynomial {H}amiltonians by the construction of formal
  integrals.
\newblock {\em Nonlinearity}, 13(4):1021--1054, 2000.

\bibitem{Rob96}
C.~Robinson.
\newblock Melnikov method for autonomous {H}amiltonians.
\newblock In Donald~G. Saari and Zhihong Xia, editors, {\em Hamiltonian
  dynamics and celestial mechanics}, Contemporary Mathematics, pages 45--53.
  American Mathematical Society, Providence, RI, 1996.

\bibitem{Sal04}
D.A. Salamon.
\newblock The {K}olmogorov-{A}rnold-{M}oser theorem.
\newblock {\em Math. Phys. Electron. J.}, 10:3--37, 2004.

\bibitem{San82}
J.A. Sanders.
\newblock Melnikov's method and averaging.
\newblock {\em Cel. Mech.}, 28:171--181, 1982.

\bibitem{Sim94}
C.~Sim\'o.
\newblock Averaging under {F}ast {Q}uasiperiodic {F}orcing.
\newblock In J.~Seimenis, editor, {\em Hamiltonian Mechanics: Integrability and
  Chaotic Behaviour}, volume 331 of {\em NATO Adv. Sci. Inst. Ser. B Phys.},
  pages 13--34, Toru\'n, Polland, 1994. Plenum Press, New York.

\bibitem{SimVall01}
C.~Sim\'o and C.~Valls.
\newblock A formal approximation of the splitting of separatrices in the
  classical {A}rnold's example of diffusion with two equal parameters.
\newblock {\em Nonlinearity}, 14:1707--1760, 2001.

\bibitem{Sok74}
A.G. Sokolski{\u\i}.
\newblock On the stability of an autonomous {H}amiltonian system with two
  degrees of freedom in the case of equal frequencies.
\newblock {\em J. Appl. Math. Mech.}, 38:741--749, 1974.
\newblock Translated from Prikl. Mat. Meh. 38, 791--799 (Russian), 1974.

\bibitem{Will37}
J.~Williamson.
\newblock On the normal forms of linear canonical transformations in dynamics.
\newblock {\em American Journal of Mathematics}, 59(3):599--617, 1937.

\end{thebibliography}
\end{document}